\documentclass[11pt]{amsart}
\usepackage{mathrsfs}
\usepackage{amssymb}
\usepackage{verbatim}
\usepackage{hyperref}
\usepackage{color}
\usepackage{amsfonts}
\usepackage{mathrsfs}
\usepackage{amsmath}
\usepackage{amssymb}
\usepackage{bbm}

\begin{document}
\newtheorem{theorem}{Theorem}
\newtheorem{proposition}[theorem]{Proposition}
\newtheorem{conjecture}[theorem]{Conjecture}
\def\theconjecture{\unskip}
\newtheorem{corollary}[theorem]{Corollary}
\newtheorem{lemma}[theorem]{Lemma}
\newtheorem{sublemma}[theorem]{Sublemma}
\newtheorem{observation}[theorem]{Observation}
\theoremstyle{definition}
\newtheorem{definition}{Definition}
\newtheorem{notation}[definition]{Notation}
\newtheorem{remark}[theorem]{Remark}
\newtheorem{question}[definition]{Question}
\newtheorem{questions}[definition]{Questions}
\newtheorem{example}[definition]{Example}
\newtheorem{problem}[definition]{Problem}
\newtheorem{exercise}[definition]{Exercise}

\numberwithin{theorem}{section} \numberwithin{definition}{section}
\numberwithin{equation}{section}

\def\earrow{{\mathbf e}}
\def\rarrow{{\mathbf r}}
\def\uarrow{{\mathbf u}}
\def\varrow{{\mathbf V}}
\def\tpar{T_{\rm par}}
\def\apar{A_{\rm par}}

\def\reals{{\mathbb R}}
\def\torus{{\mathbb T}}
\def\heis{{\mathbb H}}
\def\integers{{\mathbb Z}}
\def\naturals{{\mathbb N}}
\def\complex{{\mathbb C}\/}
\def\distance{\operatorname{distance}\,}
\def\support{\operatorname{support}\,}
\def\dist{\operatorname{dist}\,}
\def\Span{\operatorname{span}\,}
\def\degree{\operatorname{degree}\,}
\def\kernel{\operatorname{kernel}\,}
\def\dim{\operatorname{dim}\,}
\def\codim{\operatorname{codim}}
\def\trace{\operatorname{trace\,}}
\def\Span{\operatorname{span}\,}
\def\dimension{\operatorname{dimension}\,}
\def\codimension{\operatorname{codimension}\,}
\def\nullspace{\scriptk}
\def\kernel{\operatorname{Ker}}
\def\ZZ{ {\mathbb Z} }
\def\p{\partial}
\def\rp{{ ^{-1} }}
\def\Re{\operatorname{Re\,} }
\def\Im{\operatorname{Im\,} }
\def\ov{\overline}
\def\eps{\varepsilon}
\def\lt{L^2}
\def\diver{\operatorname{div}}
\def\curl{\operatorname{curl}}
\def\etta{\eta}
\newcommand{\norm}[1]{ \|  #1 \|}
\def\expect{\mathbb E}
\def\bull{$\bullet$\ }

\def\xone{x_1}
\def\xtwo{x_2}
\def\xq{x_2+x_1^2}
\newcommand{\abr}[1]{ \langle  #1 \rangle}

\newcommand{\Norm}[1]{ \left\|  #1 \right\| }
\newcommand{\set}[1]{ \left\{ #1 \right\} }
\def\one{\mathbf 1}
\def\whole{\mathbf V}
\newcommand{\modulo}[2]{[#1]_{#2}}
\def \essinf{\mathop{\rm essinf}}
\def\scriptf{{\mathcal F}}
\def\scriptg{{\mathcal G}}
\def\scriptm{{\mathcal M}}
\def\scriptb{{\mathcal B}}
\def\scriptc{{\mathcal C}}
\def\scriptt{{\mathcal T}}
\def\scripti{{\mathcal I}}
\def\scripte{{\mathcal E}}
\def\scriptv{{\mathcal V}}
\def\scriptw{{\mathcal W}}
\def\scriptu{{\mathcal U}}
\def\scriptS{{\mathcal S}}
\def\scripta{{\mathcal A}}
\def\scriptr{{\mathcal R}}
\def\scripto{{\mathcal O}}
\def\scripth{{\mathcal H}}
\def\scriptd{{\mathcal D}}
\def\scriptl{{\mathcal L}}
\def\scriptn{{\mathcal N}}
\def\scriptp{{\mathcal P}}
\def\scriptk{{\mathcal K}}
\def\frakv{{\mathfrak V}}
\def\C{\mathbb{C}}
\def\R{\mathbb{R}}
\def\Rn{{\mathbb{R}^n}}
\def\Sn{{{S}^{n-1}}}
\def\M{\mathcal{M}}
\def\N{\mathcal{N}}
\def\Q{{\mathbb{Q}}}
\def\Z{\mathbb{Z}}
\def\I{\mathcal{I}}
\def\D{\mathcal{D}}
\def\S{\mathcal{S}}
\def\supp{\operatorname{supp}}
\def\dist{\operatorname{dist}}
\def\essi{\operatornamewithlimits{ess\,inf}}
\def\esss{\operatornamewithlimits{ess\,sup}}
\author{Jarod Hart}
\address{Higuchi Biosciences Center\\University of Kansas\\ Lawrence, KS 66044\\United States}
\email{jvhart@ku.edu}
\author{Feng Liu}
\address{Feng Liu \\
        College of Mathematics and Systems Science\\
Shandong University of Science and Technology\\ Qingdao 266590\\ People's Republic of China}

\email{ liufeng860314@163.comn}

\author{Qingying Xue}
\address{Qingying Xue\\
        School of Mathematical Sciences\\
        Beijing Normal University \\
        Laboratory of Mathematics and Complex Systems\\
        Ministry of Education\\
        Beijing 100875\\
        People's Republic of China}
\email{qyxue@bnu.edu.cn}

\thanks{The second author was supported partly by NSFC (No. 11526122), SRF-SUST-RT (No. 2015RCJJ053), RAF-OYSSP (No. BS2015SF012) and SP-OYSTTT-CMSS (No. Sxy2016K01). The third named author was supported partly by NSFC (Nos. 11471041, 11671039) and NSFC-DFG (No. 11761131002).\\ \quad Corresponding author: Qingying Xue, email: qyxue@bnu.edu.cn}


\date{\today}
\title[{Regularity and continuity of local Multilinear Maximal}]{Regularity and continuity of local Multilinear Maximal type operators}
\maketitle

\begin{abstract}This paper will be devoted to study the regularity and continuity
properties of the following local multilinear fractional type maximal
operators, 
$$\mathfrak{M}_{\alpha,\Omega}(\vec{f})(x)=\sup\limits_{0<r<{\rm dist}(x,\Omega^c)}\frac{r^\alpha}{|B(x,r)|^m}\prod\limits_{i=1}^m\int_{B(x,r)}|f_i(y)|dy,\quad \hbox{for \ }0\leq\alpha<mn,$$
where $\Omega$ is a subdomain in $\mathbb{R}^n$, $\Omega^c=\mathbb{R}^n\backslash\Omega$ and $B(x,r)$
is the ball in $\mathbb{R}^n$ centered at $x$ with radius
$r$. Several new
pointwise estimates for the derivative of the local
multilinear maximal function $\mathfrak{M}_{0,\Omega}$ and the fractional maximal functions $\mathfrak{M}_{\alpha,\Omega}$  $(0<\alpha< mn)$ will be presented. These estimates will not only enable us to establish
certain norm inequalities for  these operators in Sobolev spaces, but also give us the opportunity  to obtain
the bounds of these operators on the Sobolev space with
zero boundary values. 
\end{abstract}

\section{Introduction}

Let $m$ be a positive integer and $0\leq\alpha<mn$. Suppose that $\Omega$ is a subdomain in $\mathbb{R}^n$ and  $\vec{f}=(f_1,\ldots,f_m)$ with each
$f_j\in L_{\rm loc}^1(\Omega)$, then the local multilinear 
fractional maximal operator $\mathfrak{M}_{\alpha,\Omega}$ is defined by 
$$\mathfrak{M}_{\alpha,\Omega}(\vec{f})(x)=\sup\limits_{0<r<{\rm dist}(x,\Omega^c)}\frac{r^\alpha}{|B(x,r)|^m}\prod\limits_{j=1}^m\int_{B(x,r)}|f_j(y)|dy,$$
where $\Omega^c=\mathbb{R}^n\backslash\Omega$ and $B(x,r)$ is a ball in $\mathbb{R}^n$ centered at $x$ with radius $r$.  When $\alpha=0$, we also use the notation $\mathfrak M_\Omega=\mathfrak M_{0,\Omega}$, in which case were refer to $\mathfrak M_\Omega$ as the multilinear local maximal operator or just maximal operator when the meaning is clear in context. The above supremum should be taken over all $r>0$ in the global case  $\Omega=\mathbb{R}^n$, which we then denote the operators above by $\mathfrak{M}_{\alpha}$ and $\mathfrak{M}$ when $\alpha=0$. It is easy to see that $\mathfrak M$ and $\mathfrak M_\alpha$ coincide with the $m$-linear Hardy-Littlewood maximal function \cite{LOPTT} and the fractional maximal function \cite{CX}, respectively. Moreover, in the linear case $m=1$, the operator $\mathfrak{M}$ reduces to the classical Hardy-Littlewood maximal operator $M$ and $\mathfrak{M}_{\alpha}$ reduces to the classical fractional maximal function $M_{\alpha}$ with $0<\alpha<n$.

It is well-know that the Hardy-Littlewood maximal operator $M$ is bounded on $L^p(\R^n)$ for $1<p\leq\infty$, and that as a consequence of the sublinearity of $M$ it is also continuous on these spaces.  It was shown by Kinnunen \cite{Ki} that the maximal operator $M$ is also bounded on $W^{1,p}(\R^n)$ for $1<p\leq\infty$, and it is tempting to immediately conclude that $M$ is continuous on these spaces too.  However, since weak derivative operators are not sublinear, the continuity of $M$ on $W^{1,p}(\R^n)$ does not follow immediately from boundedness.  This now leads us to a natural question
: Does $M$ map $W^{1,p} (\mathbb{R}^n)$ continuously into $W^{1,p}(\mathbb{R}^n)$?  This question was posed by Haj{\l}asz and Onninen in \cite[Question 3]{HO} where it was attributed to Iwaniec.  A complete answer for this continuity problem was first given by Lurio in \cite{Lu1}. Then, similar results of continuity were soon extended to a kind of bilinear maximal functions \cite{CM2} and to local maximal functions \cite{Lu2}.   To emphasize the need to address such subtle details, it should be noted that bounded non-sublinear operators need not to be continuous (see \cite{AL4} for example).  We extend this question to $\mathfrak M_\Omega$ and $\mathfrak M_{\alpha,\Omega}$, and show that these operators are indeed continuous on Sobolev spaces in many situations.

The regularity theory of the Hardy-Littlewood maximal operator generalizes to that of the operators $\mathfrak M_\Omega$ and $\mathfrak M_{\alpha,\Omega}$ in exactly the way one would expect, as long as $\alpha=0$, $\Omega=\R^n$, or $|\Omega|<\infty$.  This can be observed in the development of maximal operator regularity theory since the late 1990's, which is described in more detail later in this introduction.  However, there are initially unexpected and striking differences in the regularity theory for $\mathfrak M_{\alpha,\Omega}$ when $\alpha>0$ and $\Omega$ is a proper subdomain of $\R^n$ with infinite measure.  When $m=1$, Heikkinen, Kinnunen, Korvenp\"a\"a, and Tuominen showed pointwise gradient estimates \cite{HKKT} for $\mathfrak M_{\alpha,\Omega}$ for the local situation both when $|\Omega|<\infty$ and when $|\Omega|=\infty$, and the estimates proved for $|\Omega|<\infty$ are sufficient to conclude that $\mathfrak M_{\alpha,\Omega}$ is bounded on $W^{1,p}(\Omega)$ for $1<p<\infty$.  However, the pointwise estimates that can be applied when $|\Omega|=\infty$ are not readily applicable to Sobolev boundedness results.  In the infinite measure situation, it appears that there are no available results for the mapping, boundedness, or continuity properties of $\mathfrak M_{\alpha,\Omega}$.  We will also address this issue by providing conditions  sufficient for boundedness and continuity of $\mathfrak M_{\alpha,\Omega}$ whether $\Omega$ has finite or infinite measure.

Given the motivation above, we are interested in two types of results which are themselves very closely related.  They are:
\begin{itemize}
\item Gradient estimates -- Pointwise a.e. estimates for $|\nabla\mathfrak M_\Omega\vec f(x)|$ and $|\nabla\mathfrak M_{\alpha,\Omega}\vec f(x)|$,
\item Mapping properties -- Boundedness and continuity of $\mathfrak M_\Omega$ and $\mathfrak M_{\alpha,\Omega}$ from products of Sobolev spaces $W^{1,p_1}(\Omega)\times\cdots\times W^{1,p_m}$ into a Sobolev space $W^{1,q}(\Omega)$ (and in some cases $W_0^{1,p}(\Omega)$ rather than $W^{1,p}(\Omega)$).
\end{itemize}

The work in this article deals with three kind of variations on the traditional Hardy-Littlewood maximal operator:  Global $\Omega=\R^n$ versus local $\Omega\subsetneq\R^n$ theory (which can be further delineated to $\Omega\subsetneq\R^n$ with either $|\Omega|=\infty$ or $|\Omega|<\infty$), the critical index $\alpha=0$ versus fractional index $\alpha>0$ theory, and the linear $m=1$ versus multilinear $m>1$ theory.  Each of these presents various challenges, and different things are known about each situation.

The origins of these types of regularity estimates for maximal operators are in the global setting $\Omega=\R^n$, and they go back to the aforementioned work of Kinnunen \cite{Ki} for the critical $\alpha=0$ case and to Kinnunen and Saksman \cite{KS} for fractional operator where $\alpha>0$.  It was shown in those papers, respectively, that
\begin{align*}
&|\nabla Mf(x)|\lesssim  M(\nabla f)(x)\;\;\text{ and }\;\;\;|\nabla M_\alpha f(x)|\lesssim  M_\alpha(\nabla f)(x)
\end{align*}
a.e. $x\in\R^n$ for appropriate $\alpha$ and $f$.  Sobolev boundedness properties of the form $W^{1,p}(\R^n)\rightarrow W^{1,p}(\R^n)$ and $W^{1,p}(\R^n)\rightarrow W^{1,q}(\R^n)$ follow immediately.  The recent work of Liu and Wu \cite{LW1} extends these results to the multilinear settings.  They show analogous pointwise estimates for the gradient of $\mathfrak M$, and that it is bounded from $W^{1,p_1}(\R^n)\times\cdots\times W^{1,p_m}(\R^n)$ into $W^{1,q}(\R^n)$ for appropriate indices satisfying $\frac{1}{q}=\frac{1}{p_1}+\cdots+\frac{1}{p_m}$ (though they don't prove continuity of the operator on these spaces).

The first results addressing the local $\Omega\subsetneq\R^n$ theory were proved by Kinnunen and Lindqvist \cite{KL} for the critical $\alpha=0$ index operator $\mathfrak M_\Omega$.  They proved results exactly analogous to the global situation from the last paragraph, which are of the form
\begin{align*}
|\nabla \mathfrak M_\Omega f(x)|\lesssim\mathfrak M_\Omega (\nabla f)(x)
\end{align*}
a.e. $x\in\Omega$ for appropriate functions $f$.  This estimate also implies that $\mathfrak M_\Omega$ is bounded on $W^{1,p}(\Omega)$ for $1<p\leq\infty$.  Other notable work for the linear version of $\mathfrak M_\Omega$ with $\Omega\subsetneq\R^n$ include that of Luiro \cite{Lu2} and Haj{\l}asz and Onninen \cite{HO}

Less is known about the local theory $\Omega\subsetneq\R^n$ for the fractional operators $\mathfrak M_{\alpha,\Omega}$.  Though this aspect of the theory contains some surprising results and new challenges.  This can be observed in the recent work of Heikkinen, Kinnunen, Korvenp\"a\"a, and Tuominen \cite{HKKT}.  They showed that the local fractional maximal operator satisfies
\begin{align*}
|\nabla\mathfrak M_{\alpha,\Omega}f(x)|\lesssim \mathfrak M_{\alpha-1,\Omega}f(x)+\mathcal S_{\alpha-1,\Omega}f(x)
\end{align*}
and, if in addition $|\Omega|<\infty$, they show that
\begin{align*}
|\nabla\mathfrak M_{\alpha,\Omega}f(x)|\lesssim \mathfrak M_{\alpha,\Omega}(\nabla f)(x)+ \mathfrak M_{\alpha-1,\Omega}f(x)
\end{align*}
a.e. $x\in\Omega$ for appropriate functions $f$.  Here $\mathcal S_{\alpha,\Omega}f$ is the local spherical maximal operator, defined
$$\mathcal{S}_{\alpha,\Omega}f(x)=\sup\limits_{0<r<{\rm dist}(x,\Omega^c)}\frac{r^{\alpha}}{|\partial B(x,r)|}\int_{\partial B(x,r)}|f(y)|d\mathcal{H}^{n-1}(y),$$
where $d\mathcal{H}^{n-1}$ is the normalized $(n-1)$-dimensional Hausdorff measure.  It is surprising to see the extra terms appearing on the right hand side of these estimates, but they cannot be omitted entirely, which was shown in \cite{HKKT}.  It was also shown in \cite{HKKT}, as a consequence of these estimates, that $f\in L^p(\Omega)$ implies $|\nabla\mathfrak M_{\alpha,\Omega}f|\in L^q(\Omega)$.  If it is assumed in addition that $|\Omega|<\infty$, then they also show that $\mathfrak M_{\alpha,\Omega}$ maps $W^{1,p}(\Omega)$ into $W^{1,q}(\Omega)$ for appropriate indices $p$, $q$, and $\alpha$.

The topic of regularity of maximal operator has been population one lately, which has received a lot of attention from many authors.  In addition to the work already mentioned, we refer the readers to consult \cite{ACL,AL,CS,CM2,HM,Ko1,Ko2,Ku,LCW,LM,LW2,Lu1,Lu2,Ta}.

In this article, we extend the pointwise gradient estimates and Sobolev mapping properties of $\mathfrak M_\Omega$ and $\mathfrak M_{\alpha,\Omega}$ to the multlinear setting, which includes some improvements of the linear theory for the Sobolev mapping properties of $\mathfrak M_{\alpha,\Omega}$.  Our results for $\mathfrak M_\Omega$ are those that should be expected based on the local linear and the global multilinear ones.  In terms of mapping properties, we show that $\mathfrak M_\Omega$ is bounded from $W^{1,p_1}(\Omega)\times\cdots\times W^{1,p_m}(\Omega)$ into $W^{1,q}(\Omega)$ where $\frac{1}{q}=\frac{1}{p_1}+\cdots+\frac{1}{p_m}$.  In the linear setting, we provide conditions for $\mathfrak M_{\alpha,\Omega}$ to be bounded from $W^{1,p}(\Omega)$ into $W^{1,q}(\Omega)$ that are applicable even $\Omega$ is of infinite measure, which is the first result for such domains.

The paper is organized as follows. In Section 2, we present our main results. Section 3 will be devoted to give certain pointwise estimates and the corresponding norm estimates for the gradient of local multisublinear fractional maximal functions. In Section 4 we will prove the bounds for local multisublinear fractional maximal operators on the Sobolev space with zero boundary values. The continuity property of the local multisublinear fractional maximal operators on the Sobolev space will be presented in Section 5.  We would like to remark that the main ideas in our proofs are motivated by \cite{HKKT,KL,Lu2}.  However, we work in some situations that were not previously address that require new tools and yield interesting results that are substantially different even in the linear setting.  In particular, we will address some subtle issues that arise when working with domains $\Omega\subset\R^n$ of infinite measure, which end up being tied to whether Sobolev spaces on $\Omega$ permit a Sobolev embedding into Lebesgue spaces.



Throughout this paper, the characteristic function of a set $E$ is denoted by $\chi_{E}$. The constant $C$, in general, is a positive constant whose value is not necessarily the same as each occurrence. We use the following conventions $\prod_{j\in\emptyset}a_j=1$ and $\sum_{j\in\emptyset}a_j=0$.  The Sobolev space $W^{1,p}(\Omega)$, $1\leq p\leq\infty$, consists of functions $f\in L^p(\Omega)$ whose first distributional partial derivatives $D_if$, $i=1,\ldots,n$, belong to $L^p(\Omega)$. We endow $W^{1,p}(\Omega)$ with the norm $\|f\|_{W^{1,p}(\Omega)}=\|f\|_{L^p(\Omega)}+\|\nabla f\|_{L^p(\Omega)},$ where $\nabla f=(D_1 f,\ldots,D_nu)$ is the weak gradient of $f$.

\section{Main results}

In order to state a few of our results, we will need the notion of domains $\Omega$ that permit a Sobolev embedding.  A subdomain $\Omega\subset\R^n$ admits a $p$-Sobolev embedding for a given $1<p<n$ if $W^{1,p}(\Omega)$ continuously embeds into $ L^{\widetilde p}(\Omega)$ where $\frac{1}{\widetilde p}=\frac{1}{p}-\frac{1}{n}$.  We will impose this Sobolev embedding property on $\Omega$ in order to obtain results for $\mathfrak M_{\alpha,\Omega}$ whenever $\alpha>0$ and $\Omega$ is a strict subdomain of $\R^n$ with infinite measure.  A sufficient condition for this Sobolev embedding property of $\Omega$ is summarized as follows:  If there exists a bounded linear operator $E$ from $W^{1,p}(\Omega)$ into $W^{1,p}(\R^n)$ such that $Ef=f$ on $\Omega$ for all $f\in W^{1,p}(\Omega)$, then $\Omega$ admits a $p$-Sobolev embedding.  Given such an extension operator $E$, the $p$-Sobolev embedding property for $\Omega$ is obvious because $\|f\|_{L^{\widetilde p}(\Omega)}\leq\|Ef\|_{L^{\widetilde p}(\R^n)}$.  In particular, any Lipschitz domain possesses the extension property, and hence the $p$-Sobolev embedding properties as well.  It is known that other domains are Sobolev extension domains, some of which are characterized by measure density conditions and other geometric properties.  More information on extension domains and admittance of Sobolev embeddings can be found, for example, in the work of Shvartsman \cite{Shv} or Harj{\l}asz, Koskela, and Tuominen \cite{HKT}.  In particular, the latter provides a nice collection of equivalent conditions for Sobolev spaces on $\Omega$ to permit extensions to $\R^n$, and hence sufficient conditions for the Sobolev embedding condition we use below. 
\subsection{Pointwise gradient estimates}
The following results are the pointwise estimates for the gradient of $\mathfrak{M}_{\Omega}$.
\begin{theorem} [\textbf{Pointwise estimates for $\mathfrak M_\Omega$}]\label{t:alpha=0pointwise}
Let $\vec{f}=(f_1,
\ldots,f_m)$ with each $f_j\in W^{1,p_j}(\Omega)$ for $1<p_j<\infty$. Let $\frac{1}{q}=\frac{1}{p_1}+\cdots+ \frac{1}{p_m}<1$ and $1<q<\infty$. Then $\mathfrak{M}_{\Omega} (\vec{f})\in W^{1,q}(\Omega)$. Moreover, for almost every $x\in\Omega$ and $\vec{f}^l=(f_1,\ldots, f_{l-1},|\nabla f_l|,f_{l+1},\ldots,f_m)$, it holds that 
$$|\nabla\mathfrak{M}_\Omega(\vec{f})(x)|\leq 2\sum\limits_{l=1}^m \mathfrak{M}_{\Omega}(\vec{f}^l)(x).$$
\end{theorem}
\begin{theorem}[\textbf{Pointwise estimates for $\mathfrak M_{\alpha,\Omega}$ with $W^{1,p_j}(\Omega)$ functions}]
\label{t:alpha>0pointwise}
Let $\vec{f}=(f_1,\ldots,f_m)$ with each $f_j\in W^{1,p_j}(\Omega)$. Suppose either of the following conditions hold:
\begin{enumerate}
\item[{(i)}]Let $1<p_j<\infty$, $\frac{1}{q}=\frac{1}{p_1}+\cdots+
\frac{1}{p_m}-\frac{\alpha-1}{n}$ and $1<q<\infty$.  Assume that $|\Omega|<\infty$.
\item[{(ii)}] Let $1<p_1,...,p_m<n$, $1\le \alpha<mn$ and $\frac{1}{q}=\frac{1}{p_1}+\cdots+\frac{1}{p_m}-\frac{\alpha+m-1}{n}$ with $1/m<q<\infty$.  Assume that $\Omega$ admits a $p_j$-Sobolev embedding for each $j=1,...,m$.
\end{enumerate}
Then, it holds that
$\mathfrak{M}_{\alpha,\Omega}(\vec{f})\in W^{1,q}(\Omega)$.
Moreover, for almost every $x\in\Omega$ and $\vec{f}^l=(f_1,\ldots,
f_{l-1},|\nabla f_l|,f_{l+1},\ldots,f_m)$, the following inequality holds:
$$|\nabla\mathfrak{M}_{\alpha,\Omega}(\vec{f})(x)|\leq \alpha\mathfrak{M}_{\alpha-1,\Omega}(\vec{f})+2\sum\limits_{l=1}^m
\mathfrak{M}_{\alpha,\Omega}(\vec{f}^l)(x),\eqno(2.1)$$
\end{theorem}
\begin{remark}
Note that in Theorem \ref{t:alpha>0pointwise} (ii) where $\Omega$ may have unbounded measure, we impose that $\Omega$ admits Sobolev embedding properties when in order to achieve our result.  This is an appropriate condition to impose in order to assure decay of Sobolev functions in $W^{1,p}(\Omega)$, which allows us to overcome the difficulties associated to working on a domain with unbounded measure. 
\end{remark}



\begin{theorem}[\textbf{Pointwise estimates for $\mathfrak M_{\alpha,\Omega}$ with $L^{p_j}(\Omega)$ functions}]\label{t:alpha>0pointwise2}
Let $\vec{f}=(f_1, \ldots,f_m)$ with each $f_j\in L^{p_j}(\Omega)$ and $p_j>{n}/{(n-1)}$. Let $1\leq\alpha<m\beta+1$ with $\beta=\min_{1\leq j\leq m}\{{(n-1)}/{p_j},n-{2n}/{((n-1)p_j)}\}$. Let $\frac{1}{q}=\frac{1}{p_1}+\cdots+\frac{1}{p_m}-\frac{\alpha-1}{n}$ and $1<q<\infty$. Then $\mathfrak{M}_{\alpha,\Omega}(\vec{f})\in W^{1,q}(\Omega)$.  Moreover, for a.e $x\in\Omega$, it holds that 
$$|\nabla\mathfrak{M}_{\alpha,\Omega}(\vec{f})(x)|\leq C\Big(\mathfrak{M}_{\alpha-1,\Omega}(\vec{f})+\sum\limits_{l=1}^m\mathcal{S}_{\bar{\alpha},\Omega}f_l(x)\prod\limits_{1\leq j\neq l\leq m}M_{\bar{\alpha},\Omega}f_j(x)\Big),\hbox{\ for }\bar{\alpha}=
{(\alpha-1)}/{m}.$$
\end{theorem}

Theorems \ref{t:alpha=0pointwise}-\ref{t:alpha>0pointwise2} will lead to the following norm inequalities for the gradient of $\mathfrak{M}_{\Omega}.$

\subsection{Boundedness of maximal operators on Sobolev spaces}
\begin{theorem}[\textbf{Sobolev boundedness for $\mathfrak M_\Omega$}]\label{BD1}
Let $\vec{f}=(f_1,\ldots,f_m)$ and $1\le j\le m$.
Let $1<p_j<\infty$,
$\frac{1}{q}=\frac{1}{p_1}+\cdots+\frac{1}{p_m}$ and $1<q<\infty$. If $f_j\in W^{1,p_j}(\Omega)$ for $j=1,...,m$, then
$$\|\mathfrak{M}_\Omega(\vec{f})\|_{W^{1,q}(\Omega)}\leq 2m\prod\limits_{j=1}^m\|f_j\|_{W^{1,p_j}(\Omega)}.$$
\end{theorem}
\begin{theorem}[\textbf{Sobolev boundedness for $\mathfrak M_{\alpha,\Omega}$}]\label{t:boundedness}
Let $\vec{f}=(f_1,\ldots,f_m)$ and $1\le j\le m$. Then, we have
\begin{enumerate}
\item[{(i)}] Let each $1<p_j<\infty$. Let $\frac{1}{q}= \frac{1}{p_1}+\cdots+\frac{1}{p_m}-\frac{\alpha-1}{n}$ and $1<q<\infty$. If $|\Omega|<\infty$ and $f_j\in W^{1,p_j}(\Omega)$ for $j=1,...,m$, then
$$\|\mathfrak{M}_{\alpha,\Omega}(\vec{f})\|_{W^{1,q}(\Omega)}\leq C\prod\limits_{j=1}^m\|f_j\|_{W^{1,p_j}(\Omega)}.$$

\item[{(i$'$)}]  Let $1<p_1,...,p_m<n$, $1\le \alpha<mn$ and $\frac{1}{q}=\frac{1}{p_1}+\cdots+\frac{1}{p_m}-\frac{\alpha+m-1}{n}$ with $1/m<q<\infty$.  If $\Omega$ admits a $p_j$-Sobolev embedding for each $j=1,...,m$ and $f_j\in W^{1,p_j}(\Omega)$ for $j=1,...,m$, then
$$\|\mathfrak{M}_{\alpha,\Omega}(\vec{f})\|_{\dot{W}^{1,q}(\Omega)}\leq C\prod\limits_{j=1}^m\|f_j\|_{W^{1,p_j}(\Omega)}.$$

\item[{(ii)}] Let ${n}/{(n-1)}<p_j<\infty$,  $1\leq\alpha<m\beta+1$ with $\beta=\min_{1\leq j\leq m}\{{(n-1)}/{p_j},n-{2n}/{((n-1)p_j)}\}$, $\frac{1}{q}=\frac{1}{p_1}+\cdots+\frac{1}{p_m}-\frac{\alpha-1}{n}$ and $1<q<\infty$. If $|\Omega|<\infty$ and $f_j\in L^{p_j}(\Omega)$, then
$$\|\mathfrak{M}_{\alpha,\Omega}(\vec{f})\|_{W^{1,q}(\Omega)}\leq C\prod\limits_{j=1}^m\|f_j\|_{L^{p_j}(\Omega)}.$$
\end{enumerate}
\end{theorem}

\begin{remark}
We would like to make the following remarks:
\begin{enumerate}\item[{(a)}] Theorem \ref{t:alpha=0pointwise}
extends (i) of some results from \cite{KL, Lu2,HO}, and Theorems \ref{t:alpha>0pointwise}-\ref{t:alpha>0pointwise2} extend results from \cite{HKKT}, which correspond to the case $m=1$. Moreover, in (ii) of Theorem \ref{t:alpha>0pointwise}, the domain doesn't need to be bounded;

\item[{(b)}]The inequality (2.1) cannot be replaced by inequality
(1.1) in Theorem \ref{t:alpha>0pointwise}, see \cite[Example 4.1]{HKKT} for the
case $m=1$. Moreover, Theorem \ref{t:alpha>0pointwise} is sharp (see
\cite[Example 4.2]{HKKT} for the case $m=1$);

\item[{(c)}]In Theorems \ref{t:alpha>0pointwise} and \ref{t:alpha>0pointwise2} we had to assume that
$\alpha\geq1$. Indeed, when $0<\alpha<1$,
$\mathfrak{M}_{\alpha,\Omega}$ can be very irregular.
To see this, let $\vec{f}=(f_1,\ldots,f_m)$ with each
$f_j\equiv1$. Then $\mathfrak{M}_{\alpha,\Omega}(\vec{f})(x)=
{\rm dist}(x,\Omega^c)$ for all $x\in\Omega$. It
follows from \cite[Example 4.4]{HKKT} that there
exists a bounded open set $\Omega\subset\mathbb{R}^n$
such that $|\nabla\mathfrak{M}_{\alpha,\Omega}(\vec{f})|\notin
L^r(\Omega)$ for every $r>0$.
\end{enumerate}
\end{remark}
Recall that the Sobolev space with zero boundary values, denoted by $W_0^{1,p}(\Omega)$ with $1\leq p<\infty$, is defined as the completion of $\mathcal{C}_0^\infty(\Omega)$ with respect to the Sobolev norm. In 1998, Kinnunen and
Lindqvist \cite{KL} observed that $M_\Omega: W_0^{1,p}(\Omega)\rightarrow W_0^{1,p}(\Omega)$ is bounded for all $1<p<\infty$. Recently, if $|\Omega|<\infty$, Heikkinen et al. \cite{HKKT} showed that $M_{\alpha,\Omega}: L^p(\Omega)\rightarrow W_0^{1,q}(\Omega)$ is bounded for $p>{n}/{(n-1)}$, $1\leq\alpha<{n}/{p}$ and $q={np}/{(n-(\alpha-1)p)}$. 

We shall establish the following results.
\begin{theorem}[\textbf{ $W_0^{1,p}(\Omega)$ bounds for $\mathfrak{M}_\Omega$}]\label{t:Sobolev0}
Let $1<p_j,q<\infty, \frac{1}{q}=\frac{1}{p_1}
+\cdots+\frac{1}{p_m}$.   \begin{enumerate}
\item[{(i)}]  If each $f_j\in W_0^{1,p_j}
(\Omega)$. Then
$\mathfrak{M}_\Omega(\vec{f})\in W_0^{1,q}(\Omega)$;
\item[{(ii)}]  If each
$f_j\in W^{1,p_j}(\Omega)$. Then
$|\prod_{1\leq j\leq m}f_j|-\mathfrak{M}_\Omega(\vec{f})
\in W_0^{1,q}(\Omega)$.
\end{enumerate}
\end{theorem}
\begin{theorem}[\textbf{ $W_0^{1,p}(\Omega)$ bounds for $\mathfrak{M}_{\alpha,\Omega}$}]\label{t:Sobolev0'}
Let $\frac{1}{q}=
\frac{1}{p_1}+\cdots+\frac{1}{p_m}-\frac{\alpha-1}{n}$
and $1<q<\infty$.
\begin{enumerate}
\item[{(i)}] If each $f_j\in W^{1,p_j}
(\Omega)$ for $1<p_j<\infty$ and $|\Omega|<\infty$, then
$\mathfrak{M}_{\alpha,\Omega}(\vec{f})\in W_0^{1,q}(\Omega)$;
\item[{(ii)}]  If $\vec{f}=(f_1,\ldots,f_m)$ with each $f_j
\in L^{p_j}(\Omega)$ for $p_j>{n}/{(n-1)}$. Let $1\leq\alpha
<m\beta+1$ with $\beta=\min_{1\leq j\leq m}\{{(n-1)}/{p_j},
n-{2n}/{((n-1)p_j)}\}$. 
If $|\Omega|<\infty$, then $\mathfrak{M}_{\alpha,\Omega}
(\vec{f})\in W_0^{1,q}(\Omega)$.
\end{enumerate}
\end{theorem}
\begin{remark}
(i) of Theorem \ref{t:Sobolev0}
extends \cite[Corollary 4.1]{KL} and (ii) of Theorem \ref{t:Sobolev0'}
extends \cite[Corollary 3.5]{HKKT}, which correspond to
the case $m=1$.
\end{remark}
\subsection{Continuity of maximal operators on Sobolev spaces}

Now, we present the continuity of  $\mathfrak{M}_{\alpha,\Omega}$ on Sobolev
spaces.
\begin{theorem}[\textbf{Sobolev continuity for $\mathfrak M_\Omega$}]\label{t:alpha=0continuous}
Let $1<p_1,\ldots, p_m, q<\infty$ and $\frac{1}{q}=\frac{1}{p_1}+\cdots+ \frac{1}{p_m}$. Then the mapping $\mathfrak{M}_\Omega:W^{1,p_1}(\Omega)\times\ldots\times W^{1,p_m}(\Omega)\rightarrow W^{1,q}(\Omega)$
is continuous.
\end{theorem}

\begin{theorem}[\textbf{Sobolev continuity for $\mathfrak M_{\alpha,\Omega}$}]\label{t:alpha>0continuous}
Let $1<p_1,\ldots, p_m,q<\infty$, $1\leq\alpha<mn$, and $\frac{1}{q}=\frac{1}{p_1}+ \cdots+\frac{1}{p_m}-\frac{\alpha-1}{n}$.  If $|\Omega|<\infty$, then the mapping $\mathfrak{M}_{\alpha,\Omega}:W^{1,p_1}(\Omega)\times\ldots\times W^{1,p_m}(\Omega)\rightarrow W^{1,q}(\Omega)$ is continuous.
\end{theorem}

\section{Proofs of Theorems \ref{t:alpha=0pointwise}-\ref{t:boundedness}}

\subsection{Preliminaries}

The following norm estimates will provide a foundation for our analysis.

\begin{lemma}[\cite{HKKT}]\label{l:3.1}
Let $q={np}/{(n-\alpha p)}$. Then,  the following results are true.
\begin{enumerate}
\item[{(i)}] Let $p>1$, $0<\alpha<{n}/{p}$.  Then $\|M_{\alpha,\Omega}f\|_{L^q(\Omega)}\leq C\|f\|_{L^p(\Omega)}.$

\item[{(ii)}]  Let $n\geq 2$, $p>{n}/{(n-1)}$, $0\leq \alpha<\min\{{(n-1)}/{p},n-{2n}/{((n-1)p)}\}$. Then $\|\mathcal{S}_{\alpha,\Omega}f\|_{L^q(\Omega)}\leq C\|f\|_{L^p(\Omega)}.$\end{enumerate}
\end{lemma}

\begin{lemma}\label{l:3.2}
Let $\vec{f}= (f_1,\ldots,f_m)$ with each $f_j\in L^{p_j}(\Omega)$ for $p_j>1$. Let $0\leq\alpha<mn$ and $\frac{1}{q}= \frac{1}{p_1}+\cdots+\frac{1}{p_m}-\frac{\alpha}{n}$. Suppose that one of the following conditions holds:

\begin{itemize}
\item[{\rm (i)}] $\alpha=0$, $1\leq q\leq\infty$ and $1<p_1, \ldots,p_m\leq\infty$;

\item[{\rm (ii)}] $0<\alpha<n$, $1\leq q<\infty$ and $1<p_1, \ldots,p_m\leq\infty$;

\item[{\rm (iii)}] $n\leq\alpha<mn$, $1\leq q<\infty$ and $1<p_1,\ldots,p_m<\infty$.
\end{itemize}
Then, we have $\|\mathfrak{M}_{\alpha,\Omega}(\vec{f})\|_{L^q(\Omega)}\leq C\prod\limits_{j=1}^m\|f_j\|_{L^{p_j}(\Omega)}.$
\end{lemma}

\begin{remark}
Lemma \ref{l:3.2} follows easily from the norm estimate for $\mathfrak{M}_\alpha$ and the pointwise estimate $\mathfrak{M}_{\alpha,\Omega}(\vec{f})(x)\leq\mathfrak{M}_\alpha(\vec{f}\chi_{\Omega})(x)$ for every $x\in\Omega$. It also follows easily form (i) of Lemma \ref{l:3.1} and the fact that  $\mathfrak{M}_{\alpha,\Omega}(\vec{f})$ can be controlled by the products of $M_{\alpha_j,\Omega}f_j$ pointwisely for $1\le j\le m$ and $\alpha=\sum_{j=1}^m\alpha_j$.
\end{remark}

The following proposition will play a key role in the proofs of Theorems \ref{t:alpha=0pointwise}-\ref{t:boundedness}.

\begin{proposition}[\cite{HKKT,KL}]\label{p:3.3}
Let $1\leq p\leq\infty$.  If $f_k\rightarrow f$, $g_k\rightarrow g$ weakly in $L^p(\Omega)$
and $f_k\leq g_k$ $(k=1,2,\ldots)$ almost everywhere in $\Omega$,
then $f\leq g$ almost everywhere in $\Omega$.
\end{proposition}

Let $\delta(x)={\rm dist}(x,\Omega^c)$. According to Rademacher's theorem, as a Lipschitz function $\delta$ is differentiable a.e. in $\Omega$.  Moreover, $|\nabla\delta(x)|=1$ for a.e. $x\in\Omega$. For $0\leq\alpha<\infty$ and $0<t<1$, we define the multisublinear fractional average operator $\mathcal{A}_t^\alpha$ by
$$\mathcal{A}_t^\alpha(\vec{f})(x)=\frac{(t\delta(x))^\alpha}{|B(x,t\delta(x))|^m}\prod\limits_{j=1}^m\int_{B(x,t\delta(x))}f_j(y)dy,$$
where $\vec{f}=(f_1,\ldots,f_m)$ with each $f_j\in L_{\rm loc}^1(\Omega)$. If $\alpha=0$, we denote by $\mathcal{A}_t^\alpha=\mathcal{A}_t$.

The following pointwise estimate of the gradient of $\mathcal{A}_t^\alpha$ will be needed.

\begin{lemma}\label{l:3.4}
Let $\alpha\ge 1, 1< p_1,...p_m,q<\infty$, $\frac{1}{q}=\frac{1}{p_1}+\cdots+
\frac{1}{p_m}-\frac{\alpha-1}{n}$ with
$0<t<1$. 
Let $\vec{f}=(f_1,
\ldots,f_m)$ with each $f_j\in W^{1,p_j}(\Omega)$. If $|\Omega|<\infty$,
then
$|\nabla\mathcal{A}_t^\alpha(\vec{f})|\in L^q(\Omega)$.
Moreover, for almost every $x\in\Omega$ and $\vec{f}^l=(f_1,
\ldots,f_{l-1},|\nabla f_l|,f_{l+1},\ldots,f_m)$, we have 
$$|\nabla\mathcal{A}_t^\alpha(\vec{f})(x)|\leq \alpha\mathfrak{M}_{\alpha-1,\Omega}(\vec{f})+2\sum\limits_{l=1}^m
\mathfrak{M}_{\alpha,\Omega}(\vec{f}^l)(x).\eqno(3.1)$$
\end{lemma}

\begin{proof} 
We first prove (3.1) for
each $f_j\in W^{1,p_j}
(\Omega)\cap\mathcal{C}^\infty(\Omega)$. Fix $1\leq i\leq n$. By the Leibnitz
rule, it yields that
$$\begin{array}{ll}
D_i\mathcal{A}_t^\alpha(\vec{f})(x)
&=\displaystyle D_i\Big(\frac{(t\delta(x))^\alpha}{\varpi_n^m(t\delta(x))^{mn}}\Big)\prod\limits_{j=1}^m\int_{B(x,t\delta(x))}f_j(y)dy\\&\quad+\displaystyle\frac{(t\delta(x))^\alpha}{\varpi_n^m(t\delta(x))^{mn}}\sum\limits_{l=1}^m
D_i\Big(\int_{B(x,t\delta(x))}f_l(y)dy\Big)\prod\limits_{1\leq j\neq l\leq m}\int_{B(x,t\delta(x))}f_l(y)dy.
\end{array}$$
Fix $1\leq l\leq m$. For almost every $x\in\Omega$, the chain rule gives that
$$D_i\Big(\int_{B(x,t\delta(x))}f_l(y)dy\Big)=\int_{B(x,t\delta(x))}D_i f_l(y)dy+tD_i\delta(x)\int_{\partial B(x,t\delta(x))}f_l(y)d\mathcal{H}^{n-1}(y),\eqno(3.3)$$
where we have used the
fact that
$\frac{\partial}{\partial r}\int_{B(x,r)}f_l(y)dy=\int_{\partial B(x,r)}f_l(y)d\mathcal{H}^{n-1}(y).$

For almost every $x\in\Omega$, combining (3.2) with (3.3) yields that
$$\begin{array}{ll}
&D_i\mathcal{A}_t^\alpha(\vec{f})(x)\\
&=\displaystyle (\alpha-mn)(t\delta(x))^\alpha\frac{D_i\delta(x)}{\delta(x)}
\frac{1}{|B(x,t\delta(x))|^m}\prod\limits_{j=1}^m\int_{B(x,t\delta(x))}f_j(y)dy\\
&\quad+\displaystyle\sum\limits_{l=1}^m\frac{(t\delta(x))^\alpha}{|B(x,t\delta(x))|^m}\Big(\int_{B(x,t\delta(x))}D_i f_l(y)dy\Big)\prod\limits_{1\leq j\neq l\leq m}\int_{B(x,t\delta(x))}f_j(y)dy\\
&\quad+\displaystyle n(t\delta(x))^\alpha\frac{D_i\delta(x)}{\delta(x)}
\sum\limits_{l=1}^m\frac{1}{|\partial B(x,t\delta(x))|}\int_{\partial B(x,t\delta(x))}f_l(y)d\mathcal{H}^{n-1}(y)\\
&\quad\times\displaystyle\prod\limits_{1\leq j\neq l\leq m}\frac{1}{|B(x,t\delta(x))|}\int_{B(x,t\delta(x))}f_j(y)dy
\end{array}\eqno(3.4)$$Therefore, it holds that 
$$\begin{array}{ll}
{}&D_i\mathcal{A}_t^\alpha(\vec{f})(x)\\
&=\displaystyle\alpha(t\delta(x))^\alpha\frac{D_i\delta(x)}{\delta(x)}
\frac{1}{|B(x,t\delta(x))|^m}\prod\limits_{j=1}^m\int_{B(x,t\delta(x))}f_j(y)dy\\
&\quad+\displaystyle n(t\delta(x))^\alpha\sum\limits_{l=1}^m\frac{D_i\delta(x)}{\delta(x)}
\frac{1}{|B(x,t\delta(x))|^{m-1}}\prod\limits_{1\leq j\neq l\leq m}\int_{B(x,t\delta(x))}f_j(y)dy\\&\quad\times\displaystyle\Big(\frac{1}{|\partial B(x,t\delta(x))|}\int_{\partial B(x,t\delta(x))}f_l(y)d\mathcal{H}^{n-1}(y)-\frac{1}{|B(x,t\delta(x))|}\int_{B(x,t\delta(x))}f_l(y)dy\Big)\\&
\quad+\displaystyle\sum\limits_{l=1}^m\frac{(t\delta(x))^\alpha}{|B(x,t\delta(x))|^m}\Big(\int_{B(x,t\delta(x))}D_i f_l(y)dy\Big)\prod\limits_{1\leq j\neq l\leq m}\int_{B(x,t\delta(x))}f_j(y)dy
\end{array}\eqno(3.5)$$
On the other hand, we
use Green's first identity
$$\int_{\partial B(x,r)}f_l(y)\frac{\partial v}{\partial\nu}(y)d\mathcal{H}^{n-1}(y)
=\int_{B(x,r)}(f_l(y)\triangle v(y)+\nabla f_l(y)\cdot\nabla v(y))dy,\eqno(3.6)$$
where $\nu(y)=\frac{y-x}{r}$ is the unit outer normal of
$B(x,r)$. Take $v(y)=\frac{|y-x|^2}{2}$ and $r=t\delta(x)$.
Then $\nabla v(y)=y-x$, $\triangle v(y)=n$ and
$\frac{\partial v}{\partial\nu}(y)=t\delta(x)$. By
(3.6), we get
$$\begin{array}{ll}
&\displaystyle\frac{1}{|\partial B(x,t\delta(x))|}\int_{\partial B(x,t\delta(x))}f_l(y)d\mathcal{H}^{n-1}(y)-\frac{1}{|B(x,t\delta(x))|}
\int_{B(x,t\delta(x))}f_l(y)dy\\
&=\displaystyle\frac{1}{n}\frac{1}{|B(x,t\delta(x))|}\int_{B(x,t\delta(x))}\nabla f_l(y)\cdot(y-x)dy.
\end{array}$$
This together with (3.5) implies that, for almost every $x\in\Omega$, it holds that
$$\begin{array}{ll}
{}&\nabla\mathcal{A}_t^\alpha(\vec{f})(x)=\displaystyle\alpha(t\delta(x))^\alpha\frac{\nabla\delta(x)}{\delta(x)}
\frac{1}{|B(x,t\delta(x))|^m}\prod\limits_{j=1}^m\int_{B(x,t\delta(x))}f_j(y)dy\\
&\quad+\displaystyle \sum\limits_{l=1}^m\frac{\nabla\delta(x)}{\delta(x)}\frac{(t\delta(x))^\alpha}{|B(x,t\delta(x))|^m}\int_{B(x,t\delta(x))}\nabla f_l(y)\cdot(y-x)dy\prod\limits_{1\leq j\neq l\leq m}\int_{B(x,t\delta(x))}f_j(y)dy\\
&\quad+\displaystyle\sum\limits_{l=1}^m\frac{(t\delta(x))^\alpha}{|B(x,t\delta(x))|^m}\Big(\int_{B(x,t\delta(x))}\nabla f_l(y)dy\Big)\prod\limits_{1\leq j\neq l\leq m}\int_{B(x,t\delta(x))}f_j(y)dy
\end{array}\eqno(3.7)$$ For $\vec{f}^l=(f_1,\ldots,
f_{l-1},|\nabla f_l|,f_{l+1},\ldots,f_m)$, (3.7) further gives that
$$\begin{array}{ll}
|\nabla\mathcal{A}_t^\alpha(\vec{f})(x)|
&\leq\displaystyle\alpha\frac{(t\delta(x))^{\alpha-1}}{|B(x,t\delta(x))|^m}\prod\limits_{j=1}^m\int_{B(x,t\delta(x))}f_j(y)dy\\
&\quad+\displaystyle2\sum\limits_{l=1}^m\frac{(t\delta(x))^\alpha}{|B(x,t\delta(x))|^m}\int_{B(x,t\delta(x))}|\nabla f_l(y)|dy\prod\limits_{1\leq j\neq l\leq m}\int_{B(x,t\delta(x))}f_j(y)dy\\
&\leq\displaystyle\alpha\mathfrak{M}_{\alpha-1,\Omega}(\vec{f})(x)
+2\sum\limits_{l=1}^m\mathfrak{M}_{\alpha,\Omega}(\vec{f}^l)(x),
\end{array}$$
This proves (3.1)
for $\vec{f}=(f_1,\ldots,f_m)$ with each $f_j\in W^{1,p_j}
(\Omega)\cap\mathcal{C}^\infty(\Omega)$.

Now we complete the rest of proof by an approximation
argument. Suppose that $f_j\in W^{1,p_j}(\Omega)$ with
$1<p_j<\infty$. Fix $1\leq j\leq m$, there is a sequence
$\{g_{k,j}\}_{k\in\mathbb{Z}}$ of functions in $W^{1,p_j}
(\Omega)\cap\mathcal{C}^\infty(\Omega)$ such that $g_{k,j}
\rightarrow f_j$ in $W^{1,p_j}(\Omega)$ as $k\rightarrow\infty$.
Let $\vec{g}^k=(g_{k,1},\ldots,g_{k,m})$. Fix $0<t<1$. For every $x\in\Omega$, one can
easily check that
$\mathcal{A}_t^\alpha(\vec{f})(x)=\lim\limits_{k\rightarrow\infty}\mathcal{A}_t^\alpha(\vec{g}^k)(x)$. By the proved case for the smooth
functions, we have 
$$|\nabla\mathcal{A}_t^\alpha(\vec{g}^k)(x)|\leq \alpha\mathfrak{M}_{\alpha-1,\Omega}(\vec{g}^k)(x)+2\sum\limits_{l=1}^m\mathfrak{M}_{\alpha,\Omega}(\vec{g}^{k,l})(x),\eqno(3.8)$$
where $\vec{g}^{k,l}=(g_{k,1},
\ldots,g_{k,l-1},|\nabla g_{k,l}|,g_{k,l+1},\ldots,g_{k,m})$.
Let $\frac{1}{q^{*}}=\frac{1}{p_1}+\cdots+
\frac{1}{p_m}-\frac{\alpha}{n}$. Obviously, $q<q^{*}$. By
(3.8), Lemma \ref{l:3.2} and the fact $|\Omega|<\infty$, one obtains that
$$\aligned
\|\nabla\mathcal{A}_t^\alpha(\vec{g}^k)\|_{L^q(\Omega)}
&\leq\displaystyle \alpha\|\mathfrak{M}_{\alpha-1,\Omega}(\vec{g}^k)\|_{L^q(\Omega)}
+2|\Omega|^{\frac{1}{q}-\frac{1}{q^{*}}}\sum\limits_{l=1}^m\|\mathfrak{M}_{\alpha,\Omega}(\vec{g}^{k,l})\|_{L^{q^{*}}(\Omega)}\\
&\leq\displaystyle C(m,n,\alpha,\Omega)\prod\limits_{j=1}^m\|g_{k,j}\|_{W^{1,p_j}(\Omega)}.
\endaligned$$
Thus $\{|\nabla\mathcal{A}_t^\alpha(\vec{g}^k)|\}_{k
\in\mathbb{Z}}$ is a bounded sequence in $L^q(\Omega)$
and has a weakly converging subsequence $\{|\nabla
\mathcal{A}_t^\alpha(\vec{g}^{\imath_k})|\}_{k\in\mathbb{Z}}$.
Since $\mathcal{A}_t^\alpha(\vec{g}^k)\rightarrow
\mathcal{A}_t^\alpha(\vec{f})$ pointwise as $k\rightarrow\infty$,
we conclude that $|\nabla\mathcal{A}_t^\alpha(\vec{f})|$
exists and that $|\nabla\mathcal{A}_t^\alpha(\vec{g}^k)|
\rightarrow|\nabla\mathcal{A}_t^\alpha(\vec{f})|$ weakly
in $L^q(\Omega)$ as $k\rightarrow\infty$.
For all $1\leq l\leq m$, next, we shall prove that
$$\mathfrak{M}_{\alpha-1,\Omega}(\vec{g}^{k})\rightarrow\mathfrak{M}_{\alpha-1,\Omega}(\vec{f})\ \ {\rm in}\ L^q(\Omega)\ \ {\rm as}\ k\rightarrow\infty,\eqno(3.9)$$
$$\mathfrak{M}_{\alpha,\Omega}(\vec{g}^{k,l})\rightarrow\mathfrak{M}_{\alpha,\Omega}(\vec{f}^l)\ \ {\rm in}\ L^{q}(\Omega)\ \ {\rm as}\ k\rightarrow\infty.\eqno(3.10)$$ 
For every $x\in\Omega$,  let $\vec{F}_l^k=(g_{k,1},
\ldots,g_{k,l-1},|g_{k,l}-f_l|,f_{l+1},\ldots,f_m)$, we have
$$\begin{array}{ll}
&|\mathfrak{M}_{\alpha-1,\Omega}(\vec{g}^{k})(x)-\mathfrak{M}_{\alpha-1,\Omega}(\vec{f})(x)|\\
&\leq\displaystyle\sum\limits_{l=1}^m\sup\limits_{0<r<\delta(x)}\frac{r^{\alpha-1}}{|B(x,r)|^m}
\prod\limits_{\mu=1}^{l-1}\int_{B(x,r)}g_{k,\mu}(y)dy\int_{B(x,r)}|g_{k,l}(y)-f_l(y)|dy\\&\quad\times
\prod\limits_{\nu=l+1}^{m}\int_{B(x,r)}f_\nu(y)dy\leq\displaystyle\sum\limits_{l=1}^m\mathfrak{M}_{\alpha-1,\Omega}(\vec{F}_l^k)(x).
\end{array}\eqno(3.11)$$
(3.11) together with Lemma \ref{l:3.2} and Minkowski's
inequality implies
$$\|\mathfrak{M}_{\alpha-1,\Omega}(\vec{g}^{k})-\mathfrak{M}_{\alpha-1,\Omega}(\vec{f})\|_{L^q(\Omega)}
\leq C\sum\limits_{l=1}^{m}\prod\limits_{\mu=1}^{l-1}\|g_{k,\mu}\|_{L^{p_\mu}(\Omega)}
\|g_{k,l}-f_l\|_{L^{p_l}(\Omega)}\prod\limits_{\nu=l+1}^{m}\|f_\nu\|_{L^{p_\nu}(\Omega)},$$
which yields (3.9). 

Below we will prove (3.10). Without loss of
generality, we only prove (3.10) for $l=m$. Let $\vec{A}=(g_{k,1},\ldots,g_{k,m-1},|\nabla g_{k,m}|-|\nabla f_m|)$ and  $\vec{B}_\nu=(g_{k,1},\ldots,g_{k,\nu-1},|g_{k,\nu}-f_\nu|,f_{\nu+1},\ldots,f_{m-1},|\nabla f_m|)$. Then, the similar argument as in (3.11)
yields that
$$|\mathfrak{M}_{\alpha,\Omega}(\vec{g}^{k,m})(x)-\mathfrak{M}_{\alpha,\Omega}(\vec{f}^m)(x)|
\leq\mathfrak{M}_{\alpha,\Omega}(\vec{A})(x)+\sum\limits_{\nu=1}^{m-1}\mathfrak{M}_{\alpha,\Omega}(\vec{B}_\nu)(x).\eqno(3.12)$$
Using (3.12), Lemma \ref{l:3.2} and the Minkowski inequality, we get
$$\aligned 
{}&\|\mathfrak{M}_{\alpha,\Omega}(\vec{g}^{k,m})-\mathfrak{M}_{\alpha,\Omega}(\vec{f}^m)\|_{L^q(\Omega)}\\
&\leq\displaystyle C\prod\limits_{j=1}^{m-1}\|g_{k,j}\|_{L^{p_j}(\Omega)}\||\nabla g_{k,m}|-|\nabla f_m|\|_{L^{p_m}(\Omega)}\\
&\quad+\displaystyle\sum\limits_{\nu=1}^{m-1}\prod\limits_{i=1}^{\nu-1}\|g_{k,i}\|_{L^{p_i}(\Omega)}
\|g_{k,\nu}-f_\nu\|_{L^{p_\nu}(\Omega)}\prod\limits_{j=\nu+1}^{m-1}\|f_j\|_{L^{p_j}(\Omega)}\|\nabla f_m\|_{L^{p_m}(\Omega)},
\endaligned
$$
which gives (3.10) for $l=m$. 

It follows from (3.9) and
(3.10) that
$$\alpha\mathfrak{M}_{\alpha-1,\Omega}(\vec{g}^k)+2\sum_{l=1}^m\mathfrak{M}_{\alpha,\Omega}(\vec{g}^{k,l})\rightarrow\alpha\mathfrak{M}_{\alpha-1,\Omega}(\vec{f})
+2\sum_{l=1}^m\mathfrak{M}_{\alpha,\Omega}(\vec{f}^{l})\ \ {\rm in}\ L^q(\Omega)\ \ {\rm as}\ k\rightarrow\infty.$$
Applying Proposition \ref{p:3.3} to (3.7) with
$a_k=|\nabla\mathcal{A}_t^\alpha(\vec{g}^k)|\ \ {\rm and}\ \ b_k=\alpha\mathfrak{M}_{\alpha-1,\Omega}(\vec{g}^k)+2\sum\limits_{l=1}^m\mathfrak{M}_{\alpha,\Omega}(\vec{g}^{k,l}),$
yields (3.1). Thus, we completed the proof of Lemma \ref{l:3.4}.
\end{proof}

\begin{lemma}\label{l:3.4'}
Let $\vec{f}=(f_1, \ldots,f_m)$ with each $f_j\in W^{1,p_j}(\Omega)$ and $1<p_1,...,p_m<n$. Each $f_j$ enjoys the Sobolev embedding property that $\|f_j\|_{L^{ \widetilde p_j}}\lesssim \|f_j\|_{W^{1,p_j}} .$ Let $1\le \alpha<mn$ and $\frac{1}{q}=\frac{1}{p_1}+\cdots+ \frac{1}{p_m}-\frac{\alpha+m-1}{n}$ with $1/m<q<\infty$ and $0<t<1$. Then
$|\nabla\mathcal{A}_t^\alpha(\vec{f})|\in L^q(\Omega)$.
Moreover, for almost every $x\in\Omega$ and $\vec{f}^l=(f_1,
\ldots,f_{l-1},|\nabla f_l|,f_{l+1},\ldots,f_m)$, it holds that
$$|\nabla\mathcal{A}_t^\alpha(\vec{f})(x)|\leq \alpha\mathfrak{M}_{\alpha-1,\Omega}(\vec{f})+2\sum\limits_{l=1}^m
\mathfrak{M}_{\alpha,\Omega}(\vec{f}^l)(x).\eqno(3.1')$$
\end{lemma}

\begin{proof}
Similarly as in Lemma \ref{l:3.4}, we know that for
$\vec{f}=(f_1,\ldots,f_m)$ with each $f_j\in W^{1,p_j}
(\Omega)\cap\mathcal{C}^\infty(\Omega)$. For almost every $x\in\Omega$, it holds that 
$$
|\nabla\mathcal{A}_t^\alpha(\vec{f})(x)|
\leq\displaystyle\alpha\mathfrak{M}_{\alpha-1,\Omega}(\vec{f})(x)
+2\sum\limits_{l=1}^m\mathfrak{M}_{\alpha,\Omega}(\vec{f}^l)(x).
$$
Now, we may choose $1<q_1,...,q_m<\infty$ and $\alpha_1,...,\alpha_m\ge 0$, such that $\frac 1q=\frac 1{q_1}+\cdots+\frac 1{q_m}$ and $\alpha-1=\alpha_1+\cdots+\alpha_m$ with $\frac 1{q_i}=\frac 1{p_i}-\frac {1+\alpha_i}{n}.$
Define $1<\widetilde{p_1},...,\widetilde{p_m}<\infty$ via the equation $\frac 1{\widetilde p_i}=\frac 1{p_i}-\frac 1n, $ which implies that $\frac 1{q_i}=\frac 1{\widetilde p_i}-\frac {\alpha_i}n.$
Then, for $\vec{f}=(f_1,\ldots,f_m)$ with each $f_j\in W^{1,p_j}
(\Omega)\cap\mathcal{C}^\infty(\Omega)$, we have 
$$\aligned
\|\nabla\mathcal{A}_t^\alpha(\vec{f})\|_{L^q(\Omega)}
&\lesssim\prod_{i=1}^m\|{M}_{\alpha_i,\Omega}({f_i})\|_{L^{q_i}(\Omega)}
+\sum\limits_{l=1}^m\|{M}_{\alpha_l+1}({\nabla f}_l)\|_{L^{q_l}(\Omega)}\prod_{i\neq l}^m\|{M}_{\alpha_i,\Omega}({f_i})\|_{L^{q_i}(\Omega)}\\&\lesssim\prod_{i=1}^m\|{f_i}\|_{L^{\widetilde p_i}(\Omega)}
+\sum\limits_{l=1}^m\|\nabla f_l\|_{L^{p_l}(\Omega)}\prod_{i\neq l}^m\|f_i\|_{L^{\widetilde p_i}(\Omega)}
\endaligned$$
By definition, it holds that $\|\nabla f^l\|_{L^{p_l}(\Omega)}\le \|f^l\|_{W^{1,p_l}(\Omega)}$. For $\frac 1{\widetilde p_i}=\frac 1{p_i}-\frac 1n, $ the Sobolev embedding gives that $\|{f_i}\|_{L^{\widetilde p_i }(\Omega)}\lesssim \|f_i\|_{W^{1,p_i}(\Omega)}$. Then, one may obtain that 
$$\aligned
\|\nabla\mathcal{A}_t^\alpha(\vec{f})\|_{L^q(\Omega)}
&\lesssim
\prod_{i=1}^m\|{f_i}\|_{W^{1,p_i}(\Omega)}.
\endaligned$$
For each $f_i\in W^{1,p_i}(\Omega)$, there exists a sequence $f_{ik}\in C^\infty(\Omega)$, with $f_{ik} \xrightarrow{k} f_i$ in ${W^{1,p_i}}(\Omega)$, denote $\vec {g}^k=(f_{1k},f_{2k},...,f_{mk})$, then Thus $\{|\nabla\mathcal{A}_t^\alpha(\vec{g}^k)|\}_{k
\in\mathbb{Z}}$ is a bounded sequence in $L^q(\Omega)$
and has a weakly converging subsequence $\{|\nabla
\mathcal{A}_t^\alpha(\vec{g}^{\imath_k})|\}_{k\in\mathbb{Z}}$.
Since $\mathcal{A}_t^\alpha(\vec{g}^k)\rightarrow
\mathcal{A}_t^\alpha(\vec{f})$ pointwise as $k\rightarrow\infty$,
we conclude that $|\nabla\mathcal{A}_t^\alpha(\vec{f})|$
exists and that $|\nabla\mathcal{A}_t^\alpha(\vec{g}^k)|
\rightarrow|\nabla\mathcal{A}_t^\alpha(\vec{f})|$ weakly
in $L^q(\Omega)$ as $k\rightarrow\infty$.

By ({3.11}) and Sobolev embedding again, we have $$\aligned{}&\|\mathfrak{M}_{\alpha-1,\Omega}(\vec{g}^{k})-\mathfrak{M}_{\alpha-1,\Omega}(\vec{f})\|_{L^q(\Omega)}\leq\displaystyle\sum\limits_{l=1}^m\mathfrak{M}_{\alpha-1,\Omega}(\vec{F}_l^k)(x)\\&
\leq C\sum\limits_{l=1}^{m}\prod\limits_{\mu=1}^{l-1}\|M_{\alpha_{\mu},\Omega}g_{k,\mu}\|_{L^{q_\mu}(\Omega)}
\|M_{{\alpha_l}, \Omega}(g_{k,l}-f_l)\|_{L^{q_l}(\Omega)}\prod\limits_{\nu=l+1}^{m}\|M_{{\alpha_\nu}, \Omega}f_\nu\|_{L^{q_\nu}(\Omega)}\\&
\leq C\sum\limits_{l=1}^{m}\prod\limits_{\mu=1}^{l-1}\|g_{k,\mu}\|_{L^{\widetilde{p_{\mu}}}(\Omega)} \|g_{k,l}-f_l\|_{L^{\widetilde{p_l}}(\Omega)} \prod\limits_{\nu=l+1}^{m}\|f_\nu\|_{L^{\widetilde{p_\nu}}(\Omega)}\\&\leq
 C\sum\limits_{l=1}^{m}\prod\limits_{\mu=1}^{l-1}\|g_{k,\mu}\|_{W^{1, {p_{\mu}}}(\Omega)} \|g_{k,l}-f_l\|_{W^{1,{p_l}}(\Omega)} \prod\limits_{\nu=l+1}^{m}\|f_\nu\|_{W^{1,{p_\nu}}(\Omega)}\xrightarrow{k\rightarrow \infty} 0.
\endaligned
$$
Hence, $\mathfrak{M}_{\alpha-1,\Omega}(\vec{g}^{k})(x)\xrightarrow{k\rightarrow \infty}\mathfrak{M}_{\alpha-1,\Omega}(\vec{f})(x) \hbox{ in\ } L^q(\Omega).$

We only prove (3.10) for $l=m$, since other cases
are analogous. 
$$\aligned
{}&\|\mathfrak{M}_{\alpha,\Omega}(\vec{g}^{k,m})-\mathfrak{M}_{\alpha,\Omega}(\vec{f}^m)\|_{L^q(\Omega)}\\&\leq\displaystyle\|\mathfrak{M}_{\alpha,\Omega}(\vec{A})\|_{L^{q}(\Omega)}
+\sum\limits_{\nu=1}^{m-1}\|\mathfrak{M}_{\alpha,\Omega}(\vec{B}_\nu)\|_{L^{q}(\Omega)}\\
&\leq\displaystyle C\prod\limits_{j=1}^{m-1}\|g_{k,j}\|_{W^{1, p_j}}\||g_{k,m}-f_m|\|_{W^{1,p_m}}+\displaystyle\sum\limits_{\nu=1}^{m-1}\prod\limits_{i=1}^{\nu-1}\|g_{k,i}\|_{W^{1,p_i}}
\\&\quad\times \|g_{k,\nu}-f_\nu\|_{W^{1, p_\nu}}\prod\limits_{j=\nu+1}^{m-1}\|f_j\|_{W^{1,p_j}}\| f_m\|_{W^{1,p_m}}\xrightarrow{k\rightarrow \infty} 0.\endaligned
$$
\end{proof}

We need the estimate for the gradient of $\mathcal{A}_t^\alpha$ in the proof of
Theorem \ref{t:alpha>0pointwise2} as follows.

\begin{lemma}\label{l:3.5}
Let $\beta=\min_{1\leq j\leq m}
\Big\{({n-1})/{p_j},n-{2n}/{(n-1)p_j}\Big\}$,
$1\leq\alpha<m\beta+1$. Let $\bar{\alpha}={(\alpha-1)}/{m}$,
$\frac{1}{q}=\frac{1}{p_1}+\cdots+\frac{1}{p_m}-
\frac{\alpha-1}{n}$ with $1<q<\infty$ and $0<t<1$. Let $\vec{f}=(f_1,
\ldots,f_m)$ with each $f_j\in L^{p_j}(\Omega)$ and
$p_j>{n}/{(n-1)}$. Then
$|\nabla\mathcal{A}_t^\alpha(\vec{f})|\in L^q(\Omega)$.
Moreover, for almost every $x\in\Omega$, it holds that
$$|\nabla\mathcal{A}_t^\alpha(\vec{f})(x)|\leq C\Big(\mathfrak{M}_{\alpha-1,\Omega}(\vec{f})
+\sum\limits_{l=1}^m\mathcal{S}_{\bar{\alpha},\Omega}f_l(x)
\prod\limits_{1\leq j\neq l\leq m}M_{\bar{\alpha},\Omega}f_j(x)\Big).\eqno(3.13).$$
\end{lemma}

\begin{proof}
We first prove (3.13) for $\vec{f}=(f_1,\ldots,f_m)$ with each $f_j\in L^{p_j}(\Omega)
\cap\mathcal{C}^\infty(\Omega).$ For almost every $x\in\Omega$, it
follows from (3.4) that
$$\begin{array}{ll}
&\quad\nabla\mathcal{A}_t^\alpha(\vec{f})(x)=\displaystyle (\alpha-mn)\frac{\nabla\delta(x)}{\delta(x)}\frac{(t\delta(x))^\alpha}{\varpi_n^m(t\delta(x))^{mn}}
\prod\limits_{j=1}^m\int_{B(x,t\delta(x))}f_j(y)dy\\
&\quad+\displaystyle\frac{(t\delta(x))^\alpha}{\varpi_n^m(t\delta(x))^{mn}}\sum\limits_{l=1}^m\Big(\int_{B(x,t\delta(x))}\nabla f_l(y)dy\Big)\prod\limits_{1\leq j\neq l\leq m}\int_{B(x,t\delta(x))}f_j(y)dy\\
&\quad+\displaystyle\frac{(t\delta(x))^\alpha}{\varpi_n^m(t\delta(x))^{mn}}\sum\limits_{l=1}^m\Big(t\nabla\delta(x)\int_{\partial B(x,t\delta(x))}f_l(y)d\mathcal{H}^{n-1}(y)\Big)\prod\limits_{1\leq j\neq l\leq m}\int_{B(x,t\delta(x))}f_j(y)dy.
\end{array}\eqno(3.14)$$Applying Gauss's theorem
we have
$$\int_{B(x,t\delta(x))}\nabla f_l(y)dy=\int_{\partial B(x,t\delta(x))}f_l(y)\nu(y)d\mathcal{H}^{n-1}(y),\eqno(3.15)$$
where $\nu(y)=\frac{y-x}{t\delta(x)}$ is the unit outer
normal of $B(x,t\delta(x))$. Let $\bar{\alpha}=(\alpha-1)/m$.
For almost every $x\in\Omega$,
it follows form (3.14) and (3.15) that
$$\aligned|\nabla\mathcal{A}_t^\alpha(\vec{f})(x)|
&\leq\displaystyle C(\alpha,m,n)\Big(\mathfrak{M}_{\alpha-1,\Omega}(\vec{f})
+\sum\limits_{l=1}^m\mathcal{S}_{\bar{\alpha},\Omega}(f_l)(x)
\prod\limits_{1\leq j\neq\leq m}M_{\bar{\alpha},\Omega}f_j(x)\Big),
\endaligned$$ which yields (3.13) for
$\vec{f}=(f_1,\ldots,f_m)$ with each $f_j\in L^{p_j}
(\Omega)\cap\mathcal{C}^\infty(\Omega)$.

The rest of proof follows form an approximation argument.
Suppose that $f_j\in L^{p_j}(\Omega)$ with $1<p_j<\infty$.
Fix $1\leq j\leq m$, there is a sequence
$\{g_{k,j}\}_{k\in\mathbb{Z}}$ of functions in $L^{p_j}
(\Omega)\cap\mathcal{C}^\infty(\Omega)$ such that
$g_{k,j}\rightarrow f_j$ in $L^{p_j}(\Omega)$ as
$k\rightarrow\infty$. Let $\vec{g}^k=(g_{k,1},\ldots,
g_{k,m})$. Fix $0<t<1$. For any $x\in\Omega$, one can easily check that
$\mathcal{A}_t^\alpha(\vec{f})(x)=\lim\limits_{k\rightarrow\infty}\mathcal{A}_t^\alpha(\vec{g}^k)(x).$
By the proved case for smooth functions, for almost every $x\in\Omega$,
we have
$$|\nabla\mathcal{A}_t^\alpha(\vec{g}^k)(x)|\leq C\Big(\mathfrak{M}_{\alpha-1,\Omega}(\vec{g}^k)
+\sum\limits_{l=1}^m\mathcal{S}_{\bar{\alpha},\Omega}g_{k,l}(x)
\prod\limits_{1\leq j\neq l\leq m}M_{\bar{\alpha},\Omega}g_{k,j}(x)\Big).\eqno(3.16)$$
Let $\frac{1}{q}=\frac{1}{q_1}
+\cdots+\frac{1}{q_m}$ with $q_j={np_j}/{(n-\bar{\alpha}p_j)}$.
By (3.19), Lemmas \ref{l:3.1} and \ref{l:3.2} and H\"{o}lder's inequality, it holds that
$$\begin{array}{ll}
\|\nabla\mathcal{A}_t^\alpha(\vec{g}^k)\|_{L^q(\Omega)}&\leq\displaystyle C\Big(\|\mathfrak{M}_{\alpha-1,\Omega}(\vec{g}^k)\|_{L^q(\Omega)}
+\sum\limits_{l=1}^m\|\mathcal{S}_{\bar{\alpha},\Omega}g_{k,l}\|_{L^{q_l}(\Omega)}
\prod\limits_{1\leq j\neq l\leq m}\|M_{\bar{\alpha},\Omega}g_{k,j}\|_{L^{q_j}(\Omega)}\Big)\\
&\leq\displaystyle C\prod\limits_{j=1}^m\|g_{k,j}\|_{L^{p_j}(\Omega)}.
\end{array}$$
Thus $\{|\nabla\mathcal{A}_t^\alpha(\vec{g}^k)|\}_{k\in
\mathbb{Z}}$ is a bounded sequence in $L^q(\Omega)$ and
has a weakly converging subsequence
$\{|\nabla\mathcal{A}_t^\alpha(\vec{g}^{\imath_k})|\}_{k\in\mathbb{Z}}$.
Since $\mathcal{A}_t^\alpha(\vec{g}^k)\rightarrow
\mathcal{A}_t^\alpha(\vec{f})$ pointwise, we conclude that
$|\nabla\mathcal{A}_t^\alpha(\vec{f})|$ exists and that
$|\nabla\mathcal{A}_t^\alpha(\vec{g}^k)|\rightarrow
|\nabla\mathcal{A}_t^\alpha(\vec{f})|$ weakly in $L^q(\Omega)$
as $k\rightarrow\infty$.

For convenience we define the families of operators
$\{\mathscr{H}_l\}_{l=1}^m$ by
$$\mathscr{H}_l(\vec{f})(x)=\mathcal{S}_{\bar{\alpha},\Omega}f_l(x)
\prod\limits_{1\leq j\neq l\leq m}M_{\bar{\alpha},\Omega}f_j(x).$$
Below we shall prove that
$$\mathscr{H}_l(\vec{g}^k)\rightarrow\mathscr{H}_l(\vec{f})\ \ {\rm in}\ L^{q}(\Omega)\ \ {\rm as}\ k\rightarrow\infty, \quad 1\leq l\leq m.\eqno(3.17)$$
Without loss of generality we only
prove (3.17) for $l=m$. By
the sublinearity of the fractional maximal operator and
spherical fractional maximal operator, we have
\begin{eqnarray*}
&&|\mathscr{H}_m(\vec{g}^k)(x)-\mathscr{H}_m(\vec{f})(x)|\leq\displaystyle\Big|\mathcal{S}_{\bar{\alpha},\Omega}g_{k,m}
\prod\limits_{\mu=1}^{m-1}M_{\bar{\alpha},\Omega}g_{k,\mu}-\mathcal{S}_{\bar{\alpha},\Omega}f_m
\prod\limits_{\mu=1}^{m-1}M_{\bar{\alpha},\Omega}f_\mu\Big|\\
&&\leq\displaystyle\prod\limits_{\mu=1}^{m-1}M_{\bar{\alpha},\Omega}f_\mu\mathcal{S}_{\bar{\alpha},\Omega}(g_{k,m}-f_m)
+\mathcal{S}_{\bar{\alpha},\Omega}g_{k,m}\sum\limits_{l=1}^{m-1}M_{\bar{\alpha},\Omega}(g_{k,l}-f_l)(x)\prod\limits_{\mu=1}^{l-1}M_{\bar{\alpha},\Omega}g_{k,\mu}(x)
\\&&\quad \times\prod\limits_{\nu=l+1}^{m-1}M_{\bar{\alpha},\Omega}f_\nu(x).
\end{eqnarray*}
This together with H\"{o}lder's inequality and Minkowski's
inequality yields that
\begin{eqnarray*}
&&\|\mathscr{H}_m(\vec{g}^k)(x)-\mathscr{H}_m(\vec{f})(x)\|_{L^q(\Omega)}\\
&&\leq\displaystyle\prod\limits_{\mu=1}^{m-1}\|\mathcal{M}_{\bar{\alpha},\Omega}f_\mu\|_{L^{q_\mu}(\Omega)}
\|\mathcal{S}_{\bar{\alpha},\Omega}(g_{k,m}-f_m)\|_{L^{q_m}(\Omega)}+\displaystyle\|\mathcal{S}_{\bar{\alpha},\Omega}g_{k,m}\|_{L^{q_m}(\Omega)}\\
&&\quad\times\sum\limits_{l=1}^{m-1}
\|M_{\bar{\alpha},\Omega}(g_{k,l}-f_l)\|_{L^{q_l}(\Omega)}
\prod\limits_{\mu=1}^{l-1}\|M_{\bar{\alpha},\Omega}g_{k,\mu}\|_{L^{q_\mu}(\Omega)}
\prod\limits_{\nu=l+1}^{m-1}\|M_{\bar{\alpha},\Omega}f_\nu\|_{L^{q_\nu}(\Omega)}\\
&&\leq\displaystyle\prod\limits_{\mu=1}^{m-1}\|f_\mu\|_{L^{p_\mu}(\Omega)}\|g_{k,m}-f_m\|_{L^{p_m}(\Omega)}+\displaystyle\|g_{k,m}\|_{L^{p_m}(\Omega)}\sum\limits_{l=1}^{m-1}
\|g_{k,l}-f_l\|_{L^{p_l}(\Omega)}\\
&&\quad\times\prod\limits_{\mu=1}^{l-1}\|g_{k,\mu}\|_{L^{p_\mu}(\Omega)}
\prod\limits_{\nu=l+1}^{m-1}\|f_\nu\|_{L^{p_\nu}(\Omega)},
\end{eqnarray*}
which gives (3.17) for $l=m$. 

On the other hand, by arguments
similar to those used in deriving (3.9), one obtains
$\mathfrak{M}_{\alpha-1,\Omega}(\vec{g}^{k})\rightarrow\mathfrak{M}_{\alpha-1,\Omega}(\vec{f})\  {\rm in}\ L^q(\Omega)\  {\rm as}\ k\rightarrow\infty.$
This together with (3.17) implies that
$$\mathfrak{M}_{\alpha-1,\Omega}(\vec{g}^k)+\sum\limits_{l=1}^m\mathscr{H}_l(\vec{g}^k)\rightarrow\mathfrak{M}_{\alpha-1,\Omega}(\vec{f})
+\sum\limits_{l=1}^m\mathscr{H}_l(\vec{f})\ \ {\rm in}\ L^q(\Omega)\ \ {\rm as}\ \ k\rightarrow\infty.$$
Applying Proposition \ref{p:3.3} to (3.16) with
$a_k=|\nabla\mathcal{A}_t^\alpha(\vec{g}^k)|\ \ {\rm and}\ \ b_k=\mathfrak{M}_{\alpha-1,\Omega}(\vec{g}^k)
+\sum\limits_{l=1}^m\mathscr{H}_l(\vec{g}^k),$
we get (3.13). This finishes the proof of Lemma \ref{l:3.5}. 
\end{proof}

The following lemma is needed in the proof of Theorem \ref{t:alpha=0pointwise}.

\begin{lemma}\label{l:3.6}
Let $1<p_1,...,p_m<\infty$, $\frac{1}{q}=\frac{1}{p_1}+\cdots+\frac{1}{p_m}$ with $1<q\leq\infty$. Let $\vec{f}=(f_1,\ldots,f_m)$ with $f_j\in W^{1,p_j}(\Omega)$. Then there exists a constant $C>0$ such that
$$|\nabla\mathcal{A}_t(\vec{f})(x)|\leq 2\sum\limits_{l=1}^m
\mathfrak{M}_{\Omega}(\vec{f}^l)(x), \ \ a.e. \ \ x\in\Omega. \eqno(3.18)$$
\end{lemma}

\begin{proof}
By (3.7) we may obtain that (3.18) holds 
for each $f_j\in W^{1,p_j}
(\Omega)\cap\mathcal{C}^\infty(\Omega)$.
Now an approximation argument similarly as in Lemma \ref{l:3.4} gives the desired result.
\end{proof}

\subsection{ Proofs of Theorems \ref{t:alpha=0pointwise}-\ref{t:boundedness}}

\begin{proof}[Proof of Theorem \ref{t:alpha=0pointwise}.]
Let each $f_j\in W^{1,p_j}(\Omega)$. Without loss of generality we may
assume that all $f_j\geq0$ since $|u|\in W^{1,p}(\Omega)$
if $u\in W^{1,p}(\Omega)$ for $1<p<\infty$. Let $t_k$,
$k=1,2,\ldots,$ be an enumeration of the rationals between
$0$ and $1$. Denote
$\mathfrak{M}_\Omega(\vec{f})(x)=\sup\limits_{k\geq1}\mathcal{A}_{t_k}(\vec{f})(x)$ for every $x\in\Omega.$ For $k\geq1$, we define the function
$g_k:\Omega\rightarrow[-\infty,\infty]$ by
$g_k(x)=\max\limits_{1\leq i\leq k}\mathcal{A}_{t_i}(\vec{f})(x).$
Invoking Lemma \ref{l:3.6} we see that $\mathcal{A}_{t_k}(\vec{f})\in W^{1,q}
(\Omega)$ and for almost every $x\in\Omega$, it holds
$$|\nabla\mathcal{A}_{t_k}(\vec{f})(x)|\leq 2\sum\limits_{l=1}^m
\mathfrak{M}_{\Omega}(\vec{f}^l)(x).\eqno(3.19)$$

Using (3.19) and
the fact that the maximum of two Sobolev functions belongs
to the Sobolev space (\cite[Lemma 7.6]{GT}), we see
that $\{g_k\}_{k\in\mathbb{Z}}$ is an increasing sequence
of functions in $W^{1,q}(\Omega)$ converging to
$\mathfrak{M}_\Omega(\vec{f})$ pointwise and for all $k\geq1$ and almost every $x\in\Omega$. 
$$|\nabla g_k(x)|\leq\max\limits_{1\leq i\leq k}|\nabla
\mathcal{A}_{t_i}(\vec{f})(x)|\leq 2\sum\limits_{l=1}^m
\mathfrak{M}_{\Omega}(\vec{f}^l)(x).\eqno(3.20)$$
On the
other hand, $g_k(x)\leq\mathfrak{M}_\Omega(\vec{f})(x)$
for all $k\geq1$ and every $x\in\Omega$. The rest of the
proof along the lines of the final part of the proof of Lemma \ref{l:3.4}. By (3.20), and using the boundedness of $\mathfrak{M}_\Omega$ and $\|\nabla f_i\|_{L^{p_i}}\le \|\nabla f_i\|_{w^{1, p_i}} $. We have 
$$\|g_k\|_{W^{1,q}(\Omega)}=\|g_k\|_{L^q(\Omega)}+\|\nabla g_k\|_{L^q(\Omega)}\leq C\prod\limits_{j=1}^m\|f_j\|_{W^{1,p_j}(\Omega)}.$$
Thus $\{g_k\}_{k\geq1}$ is a bounded sequence in $W^{1,q}
(\Omega)$ such that $g_k\rightarrow\mathfrak{M}_\Omega
(\vec{f})$ a.e. in $\Omega$ as $k\rightarrow
\infty$. A weak compactness argument shows that
$\mathfrak{M}_\Omega(\vec{f})\in W^{1,q}(\Omega)$,
$g_k\rightarrow\mathfrak{M}_\Omega(\vec{f})$ and
$\nabla g_k\rightarrow\nabla\mathfrak{M}_\Omega(\vec{f})$
weakly in $L^q(\Omega)$ as $k\rightarrow\infty$. We may
proceed to the weak limit in (3.20). The same
argument as in the end of the proof of Lemma \ref{l:3.4} yield the desired result. 
\end{proof}

\begin{proof}[Proof of Theorem \ref{t:alpha>0pointwise}.]

{\it Proof for (i)}:  Since each $f_j\in W^{1,p_j}(\Omega)$ for
$1<p_j<\infty$. Without loss of generality we may assume
that all $f_j\geq0$. Let $t_k$, $k=1,2,\ldots,$ be an
enumeration of the rationals between $0$ and $1$. We denote $\mathfrak{M}_{\alpha,\Omega}(\vec{f})(x)=\sup\limits_{k\geq1}\mathcal{A}_{t_k}(\vec{f})(x).$ Invoking Lemma \ref{l:3.4} gives that $|\nabla\mathcal{A}_{t_k}^\alpha(\vec{f})|\in L^q(\Omega)$
for all $k\geq1$ and for every $x\in\Omega$ and $\vec{f}^l=(f_1,\ldots,
f_{l-1},|\nabla f_l|,f_{l+1},\ldots,f_m)$, one may obtain
$$|\nabla\mathcal{A}_{t_k}^\alpha(\vec{f})(x)|\leq \alpha\mathfrak{M}_{\alpha-1,\Omega}(\vec{f})(x)+2\sum\limits_{l=1}^m
\mathfrak{M}_{\alpha,\Omega}(\vec{f}^l)(x),\eqno(3.21)$$

For $k\geq1$,
we define the function $g_k:\Omega\rightarrow[-\infty,\infty]$
by
$g_k(x)=\max\limits_{1\leq i\leq k}\mathcal{A}_{t_k}(\vec{f})(x).$
Obviously, $\{g_k\}_{k\geq1}$ is an increasing sequence
of functions converging to $\mathfrak{M}_{\alpha,\Omega}
(\vec{f})$ pointwisely. By (3.21), for all $k\geq1$ and almost every $x\in\Omega$, we have
$$|\nabla g_k(x)|\leq\max\limits_{1\leq i\leq k}|\nabla
\mathcal{A}_{t_i}(\vec{f})(x)|\leq \alpha\mathfrak{M}_{\alpha-1,\Omega}(\vec{f})(x)+2\sum\limits_{l=1}^m
\mathfrak{M}_{\alpha,\Omega}(\vec{f}^l)(x),\eqno(3.22)$$
Therefore, by (3.22) and using the same argument as
in the end of the proof of Lemma \ref{l:3.4}, we can obtain
Theorem \ref{t:alpha>0pointwise}. 

{\it Proof for (ii)}: In the proof of (i), using  Lemma \ref{l:3.4'} instead of Lemma \ref{l:3.4} and employing the same argument as in Lemma \ref{l:3.4'} may yield the desired conclusion.
\end{proof}

\begin{proof}[Proof of Theorem \ref{t:alpha>0pointwise2}.]
The proof
is analogous to the proofs of Theorem \ref{t:alpha>0pointwise} and Lemma \ref{l:3.4},
but using Lemma \ref{l:3.5} instead of Lemma \ref{l:3.4}. We omit the
the proof of it. 
\end{proof}
\begin{proof}[Proof of Theorem \ref{BD1}.]
By Lemma \ref{l:3.2} and Theorem \ref{t:alpha=0pointwise} and Minkowski's
inequality, one obtains that
$$\|\mathfrak{M}_\Omega(\vec{f})\|_{W^{1,q}(\Omega)}\leq\|\mathfrak{M}_\Omega(\vec{f})\|_{L^q(\Omega)}+2\sum\limits_{l=1}^m\|\mathfrak{M}_\Omega(\vec{f}^l)\|_{L^q(\Omega)}\leq C\prod\limits_{j=1}^m\|f_j\|_{W^{1,p_j}(\Omega)}.$$
\end{proof}

\begin{proof}[Proof of Theorem \ref{t:boundedness}.]
 Let $\frac{1}{q^{*}}=\frac{1}{p_1}+\cdots+
\frac{1}{p_m}-\frac{\alpha}{n}$. Obviously, $q<q^{*}$.
By Lemma \ref{l:3.2} and Minkowski's inequality, one obtains that
$$\begin{array}{ll}
&\|\mathfrak{M}_{\alpha,\Omega}(\vec{f})\|_{W^{1,q}(\Omega)}\\
&\leq\displaystyle\|\mathfrak{M}_{\alpha,\Omega}(\vec{f})\|_{L^q(\Omega)}
+\alpha\|\mathfrak{M}_{\alpha-1,\Omega}(\vec{f})\|_{L^q(\Omega)}
+2\sum\limits_{l=1}^m\|\mathfrak{M}_{\alpha,\Omega}(\vec{f}^l)\|_{L^q(\Omega)}\\
&\leq\displaystyle|\Omega|^{\frac{1}{q}-\frac{1}{q^{*}}}\|\mathfrak{M}_{\alpha,\Omega}(\vec{f})\|_{L^{q^{*}}(\Omega)}
+\alpha\|\mathfrak{M}_{\alpha-1,\Omega}(\vec{f})\|_{L^q(\Omega)}
+2|\Omega|^{\frac{1}{q}-\frac{1}{q^{*}}}\sum\limits_{l=1}^m\|\mathfrak{M}_{\alpha,\Omega}(\vec{f}^l)\|_{L^{q^{*}}(\Omega)}\\
&\leq\displaystyle C(m,n,\alpha,\Omega)\prod\limits_{j=1}^m\|f_j\|_{W^{1,p_j}(\Omega)}.
\end{array}$$
This yields (i). While (1$'$) is a consequence of Theorem \ref{t:alpha>0pointwise} (ii). 

Let $\frac{1}{q}=\frac{1}{q_1}+\cdots+
\frac{1}{q_m}$ with $q_j={np_j}/{(n-\bar{\alpha}p_j)}$.
By Lemmas \ref{l:3.1}-\ref{l:3.2} and Theorem \ref{t:alpha>0pointwise2}, H\"{o}lder's inequality,
Minkowski's inequality and the fact $|\Omega|<\infty$, we have
$$\begin{array}{ll}
&\|\mathfrak{M}_{\alpha,\Omega}(\vec{f})\|_{W^{1,q}(\Omega)}\\
&\leq\displaystyle\|\mathfrak{M}_{\alpha,\Omega}(\vec{f})\|_{L^q(\Omega)}+C\|\mathfrak{M}_{\alpha-1,\Omega}(\vec{f})\|_{L^q(\Omega)}+
C\sum\limits_{l=1}^m\Big\|\mathcal{S}_{\bar{\alpha},\Omega}f_l
\prod\limits_{1\leq j\neq l\leq m}M_{\bar{\alpha},\Omega}f_j\Big\|_{L^q(\Omega)}\\
&\leq\displaystyle|\Omega|^{\frac{1}{q}-\frac{1}{q^{*}}}\|\mathfrak{M}_{\alpha,\Omega}(\vec{f})\|_{L^{q^{*}}(\Omega)}
+C\|\mathfrak{M}_{\alpha-1,\Omega}(\vec{f})\|_{L^q(\Omega)}\\&\quad+
C\sum\limits_{l=1}^m\|\mathcal{S}_{\bar{\alpha},\Omega}f_l\|_{L^{q_l}(\Omega)}
\prod\limits_{1\leq j\neq l\leq m}\|M_{\bar{\alpha},\Omega}f_j\|_{L^{q_j}(\Omega)}\\
&\leq\displaystyle C(m,n,\alpha,\Omega)\prod\limits_{j=1}^m\|f_j\|_{L^{p_j}(\Omega)}.
\end{array}$$
This proves (ii). 
\end{proof}

\section{Proofs of Theorems \ref{t:Sobolev0}-\ref{t:Sobolev0'}}

We need the following property of the Sobolev space with zero
boundary values.

\begin{lemma}[\cite{KM}]\label{l:4.1}
Let $\Omega\subset\mathbb{R}^n$, $\Omega\neq\mathbb{R}^n$, be an open set. Let $f\in W^{1,p}(\Omega)$ for $1<p<\infty$ and $\int_{\Omega}\Big({f(x)}/{{\rm dist}(x,\Omega^c)}\Big)^pdx<\infty.$  Then $f\in W_0^{1,p}(\Omega)$.
\end{lemma}

\begin{proof}[Proof of Theorem \ref{t:Sobolev0}.]
Let
$\vec{f}=(f_1,\ldots,f_m)$ with each $f_j\in W_0^{1,p_j}
(\Omega)$, $\frac{1}{q}=
\frac{1}{p_1}+\ldots+\frac{1}{p_m}$ with $1<q<\infty$.
Fix $1\leq j\leq m$, there is a sequence
$\{g_{k,j}\}_{k\in\mathbb{Z}}$ of functions in
$W^{1,p_j}(\Omega)\cap\mathcal{C}_0^\infty(\Omega)$ such
that $g_{k,j}\rightarrow f_j$ in $W^{1,p_j}(\Omega)$ as
$k\rightarrow\infty$. Let $\vec{g}^k=(g_{k,1},\ldots,
g_{k,m})$. It follows from Theorem \ref{t:boundedness} that
$\mathfrak{M}_\Omega(\vec{g}^k)\in W^{1,q}(\Omega)$. It
is easy to check that $\mathfrak{M}_\Omega(\vec{g}^k)(x)
=0$ whenever ${\rm dist}(x,\Omega^c)<\frac{1}{2}\min_{1\leq j
\leq m}{\rm dist}({\rm supp}(g_{k,j}),\Omega^c)$. Thus
$\{\mathfrak{M}_\Omega(\vec{g}^k)\}_{k\in\mathbb{Z}}$ is
a bounded sequence in $W_0^{1,q}(\Omega)$. By the arguments
similar to those used to derive (3.13), for every $x\in\Omega$, we have
$|\mathfrak{M}_{\Omega}(\vec{g}^{k})(x)-\mathfrak{M}_{\Omega}(\vec{f})(x)|\leq
\sum\limits_{l=1}^m\mathfrak{M}_{\Omega}(\vec{F}_l^k)(x),$ where $\vec{F}_l^k=(g_{k,1},\ldots,
g_{k,l-1},|g_{k,l}-f_l|,f_{l+1},\ldots,f_m)$. This gives
that
$$\aligned\|\mathfrak{M}_{\Omega}(\vec{g}^{k})-\mathfrak{M}_{\Omega}(\vec{f})\|_{L^q(\Omega)}\leq\sum\limits_{l=1}^m\|g_{k,l}-f_l\|_{L^{p_l}(\Omega)}
\prod\limits_{\mu=1}^{l-1}\|g_{k,\mu}\|_{L^{p_\mu}(\Omega)}\prod\limits_{\nu=l+1}^m\|f_\nu\|_{L^{p_\nu}(\Omega)}.\endaligned$$
Therefore, it implies that $\{\mathfrak{M}_\Omega(\vec{g}^k)\}_{k\in\mathbb{Z}}$
converges to $\mathfrak{M}_\Omega(\vec{f})$ in $L^q(\Omega)$.
A weak compactness argument shows that $\mathfrak{M}_\Omega
(\vec{f})\in W_0^{1,q}(\Omega)$.

(ii) Let each $f_j\in W^{1,p_j}
(\Omega)$. Fix $0<t<1$ and $1\leq l\leq m$.
By the arguments similar to those used in deriving
\cite[Theorem 3.12]{AK}, there exists $C=C(n)>0$ such that
$$\Big||f_l(x)|-\frac{1}{|B(x,t\delta(x))|}\int_{B(x,t\delta(x))}|f_l(y)|dy\Big|\leq Ct\delta(x)M_\Omega|\nabla f_l|(x).\eqno(4.1)$$
Let $t_k$, $k=1,2,\ldots,$ be an enumeration of the rationals
between $0$ and $1$. We still write
$\mathfrak{M}_\Omega(\vec{f})(x)=\sup\limits_{k\geq1}\mathcal{A}_{t_k}(\vec{f})(x)$
for every $x\in\Omega$. For $k\geq1$, define the function
$g_k:\Omega\rightarrow[-\infty,\infty]$ by
$g_k(x)=\max\limits_{1\leq i\leq k}\mathcal{A}_{t_i}(\vec{f})(x).$
Then
$\mathfrak{M}_\Omega(\vec{f})(x)=\lim\limits_{k\rightarrow\infty}g_k(x)$
for every $x\in\Omega$. For convenience, we set $G(\vec{f})(x)
=\prod_{j=1}^mf_j(x)$. Let $\frac{1}{q}=\frac{1}{p_1}+\ldots+
\frac{1}{p_m}$. One can easily check that $G(\vec{f})\in
W^{1,q}(\Omega)$. Thus we have $|G(\vec{f})|\in W^{1,q}
(\Omega)$. Fix $k\geq1$. We get from (4.1) that
$$\begin{array}{ll}
\quad\displaystyle\big||G(\vec{f})(x)|-g_k(x)\big|
&\leq\displaystyle\max\limits_{1\leq i\leq k}\Big|\prod\limits_{j=1}^m|f_j(x)|-\mathcal{A}_{t_i}(\vec{f})(x)\Big|\\
&\leq\displaystyle\sum\limits_{l=1}^m\max\limits_{1\leq i\leq k}\Big||f_l(x)|-\frac{1}{|B(x,t_i\delta(x))|}\int_{B(x,t_i\delta(x))}|f_l(y)|dy\Big|\\
&\quad\times\displaystyle\prod\limits_{\mu=1}^{l-1}|f_\mu(x)|
\prod\limits_{\nu=l+1}^m\frac{1}{|B(x,t_i\delta(x))|}\int_{B(x,t_i\delta(x))}|f_\nu(y)|dy\\
&\leq\displaystyle C\delta(x)\sum\limits_{l=1}^mM_\Omega|\nabla f_l|(x)
\prod\limits_{1\leq j\neq l\leq m}M_\Omega f_j(x).
\end{array}$$
It follows that
$$\big||G(\vec{f})(x)|-\mathfrak{M}_\Omega(\vec{f})(x)\big|\leq
C\delta(x)\sum\limits_{l=1}^mM_\Omega|\nabla f_l|(x)\prod\limits_{1\leq j\neq l\leq m}M_\Omega f_j(x).$$
This combining with H\"{o}lder's inequality and Minkowski's
inequality yields that
$$\begin{array}{ll}
\displaystyle\int_{\Omega}\Big(\frac{\big||G(\vec{f})(x)|-\mathfrak{M}_\Omega(\vec{f})(x)\big|}{{\rm dist}(x,\Omega^c)}\Big)^qdx
&\leq\displaystyle C\sum\limits_{l=1}^m\int_{\Omega}\Big(M_\Omega|\nabla f_l|(x)\prod\limits_{1\leq j\neq l\leq m}M_\Omega f_j(x)\Big)^qdx\\
&\leq\displaystyle C\sum\limits_{l=1}^m\|\nabla f_l\|_{L^{p_l}(\Omega)}\prod\limits_{1\leq j\neq\l\leq m}\|f_j\|_{L^{p_j}(\Omega)}<\infty.
\end{array}$$
By Theorem \ref{t:boundedness} we have $\mathfrak{M}_\Omega(\vec{f})\in
W^{1,q}(\Omega)$. Thus, $|G(\vec{f})|-\mathfrak{M}_\Omega
(\vec{f})\in W^{1,q}(\Omega)$. Using Lemma \ref{l:4.1} we conclude
that $|G(\vec{f})|-\mathfrak{M}_\Omega(\vec{f})\in W_0^{1,q}
(\Omega)$. 
\end{proof}

\begin{proof}[Proof of Theorem \ref{t:Sobolev0'}.]
Theorem \ref{t:boundedness} gives that $\mathfrak{M}_{\alpha,\Omega}(\vec{f})\in W^{1,q}(\Omega)$. For every $x\in\Omega$,  One can easily check that
$$\mathfrak{M}_{\alpha,\Omega}(\vec{f})(x)\leq{\rm dist}(x,\Omega^c)\mathfrak{M}_{\alpha-1,\Omega}(\vec{f})(x).\eqno(4.2)$$
Applying Lemma \ref{l:3.2} and (4.2) we obtain
$$\int_{\Omega}\Big(\frac{\mathfrak{M}_{\alpha,\Omega}(\vec{f})(x)}{{\rm dist}(x,\Omega^c)}\Big)^pdx<\int_{\Omega}(\mathfrak{M}_{\alpha-1,\Omega}(\vec{f})(x))^qdx\leq C\prod\limits_{j=1}^m\|f_j\|_{L^{p_j}(\Omega)}<\infty.$$
Then, by Lemma \ref{l:4.1} it holds that $\mathfrak{M}_{\alpha,\Omega}(\vec{f}) \in W_0^{1,q}(\Omega)$. 
\end{proof}

\section{Proofs of Theorems \ref{t:alpha=0continuous}-\ref{t:alpha>0continuous}}

\subsection{Preliminaries}

For $R>0$, let $B_R$ be the ball of radius $R$ centered
at the origin. For $A\subset\mathbb{R}^n$ and $x\in\mathbb{R}^n$,
let
$d(x,A):=\inf\limits_{a\in A}|x-a|\ \ {\rm and}\ A_{(\lambda)}:=\{x\in\mathbb{R}^n;d(x,A)\leq\lambda\}\ \ {\rm for}\ \lambda\geq0.$
For convenience we denote by $A_{x,r}(f)={|B(x,r)|^{-1}}
\int_{B(x,r)}f(y)dy$. The notation $K\subset\subset\Omega$
means that $K$ is open, bounded and $\overline{K}\subset\Omega$.

For every $x\in\Omega$, we define the function $u_{x,{\vec{f}},\alpha}
:[0,\delta(x)]\mapsto\mathbb{R}$ by
\[
  u_{x,\vec{f},\alpha}(0)= \left\{
\begin{array}{ll}
\displaystyle\prod\limits_{i=1}^m |f_i(x)|,&\ \ {\rm if}\ \alpha=0;\\
0,&\ \ {\rm if}\ \alpha>0.
\end{array}
\right.
\]
and
$u_{x,\vec{f},\alpha}(r)=r^\alpha\prod\limits_{i=1}^mA_{x,r}(f_i)\ \ {\rm when}\ \ r\in(0,\delta(x)],$
whence it holds that
$$\mathfrak{M}_{\alpha,\Omega}(\vec{f})(x)=\sup\limits_{r\in(0,\delta(x))}u_{x,\vec{f},\alpha}(r).$$
We also define the set $\mathcal{R}_\alpha(\vec{f})(x)$ by
$\mathcal{R}_\alpha(\vec{f})(x)=\{r\in[0,\delta(x)];\, \mathfrak{M}_{\alpha,\Omega}(\vec{f})(x)=u_{x,\vec{f},\alpha}(r)\}.$
If $\alpha=0$, we denote by $u_{x,\vec{f},\alpha}
=u_{x,\vec{f}}$ and $\mathcal{R}_\alpha(\vec{f})=
\mathcal{R}(\vec{f})$. Note that $u_{x,\vec{f},\alpha}$ are
continuous on $(0,\delta(x)]$ for all $x\in\Omega$ and at
$r=0$ for almost every $x\in\Omega$. Thus, if $x\in\Omega$
is a Lebesgue point of all $f_j$, then the function
$u_{x,\vec{f},\alpha}$ reaches its maximum on $[0,\delta(x)]$.
It follows that $\mathcal{R}_\alpha(\vec{f})(x)$ is nonempty
and closed for all $x\in\Omega$ being a Lebesgue point of
all $f_j$.

The following lemma is a multilinear version of the result in
\cite[Lemma 2.2]{Lu1}.

\begin{lemma}\label{l:5.1}
Let $0\leq\alpha<mn$
and $\vec{f_j}=(f_{1,j},f_{2,j},\ldots,f_{m,j})$. Suppose that
$f_{i,j}\rightarrow f_i$ in $L^{p_i}(\Omega)$ when $j\rightarrow
\infty$ for all $i=1,2,\ldots,m$, where $1<p_1,\ldots,p_m<\infty$,
$0<\frac{1}{q}-\frac{\alpha}{n}=\frac{1}{p_1}+\ldots+\frac{1}{p_m}<1$.
Let $\Omega_R=\Omega\cap B_R$.Then for all $R>0$ and $\lambda>0$, it holds that
$$\lim\limits_{j\rightarrow\infty}|\{x\in\Omega_R;\,\mathcal{R}_\alpha(\vec{f_j})(x)\nsubseteq\mathcal{R}_\alpha(\vec{f})(x)_{(\lambda)}\}|=0.\eqno(5.1)$$
\end{lemma}

\begin{proof}
We may assume without
loss of generality that all $f_{i,j}\geq0$ and $f_i\geq0$.
Then for
given $\epsilon\in(0,1)$, there exists $N_0=N_0(m,\epsilon)
\in\mathbb{N}$ such that
$$\|f_{i,j}-f_i\|_{L^{p_i}(\Omega)}<\epsilon\ \ {\rm and}\ \|f_{i,j}\|_{L^{p_i}(\Omega)}\leq \|f_i\|_{L^{p_i}(\Omega)}+1\eqno(5.2)$$
for any $j\geq N_0$ and $i=1,2,\ldots,m$. Using the similar
argument as in the proof of \cite[Lemma 2.2]{Lu1}, we see
that the set $\{x\in\Omega_R;\ \mathcal{R}_\alpha(\vec{f_j})
(x)\nsubseteq\mathcal{R}(\vec{f})_\alpha(x)_{(\lambda)}\}$
is measurable for any $j\in\mathbb{Z}$ when all $f_{i,j}$
and $f_j$ are locally integrable functions. Let $\lambda>0$
and $R>0$. It easy to see that, for almost every $x\in\Omega_R$,
there exists $\gamma(x)\in\mathbb{N}\backslash\{0\}$ such
that
$$u_{x,\vec{f},\alpha}(r)<\mathfrak{M}_{\alpha,\Omega}(\vec{f})(x)-\frac{1}{\gamma(x)},\ \ {\rm when}\ d(r,\mathcal{R}_\alpha(\vec{f})(x))>\lambda.\eqno(5.3)$$
From
(3.3) we can conclude that there exists $\gamma=\gamma(R)
\in\mathbb{N}\backslash\{0\}$ such that
$$\Omega_R\subset\Big\{x\in\Omega_R:u_{x,\vec{f},\alpha}(r)<\mathfrak{M}_{\alpha,\Omega}(\vec{f})(x)-\frac{1}{\gamma},
\ \ {\rm if}\ d(r,\mathcal{R}_\alpha(\vec{f})(x))>\lambda\Big\}\cup E:=B\cup E,\eqno(5.4)$$
where $E$ is a zero measurable set. Define
$$B_{1,j}=\Big\{x\in\Omega_R: |u_{x,\vec{f_j},\alpha}(r)-u_{x,\vec{f},\alpha}(r)|\geq\frac{1}{2\gamma}\ \ {\rm for\ some}\ r\ {\rm such\ that}\ d(r,\mathcal{R}_\alpha(\vec{f})(x))>\lambda\Big\},$$
$$B_{2,j}=\Big\{x\in\Omega_R:|\mathfrak{M}_{\alpha,\Omega}(\vec{f_j})(x)-\mathfrak{M}_{\alpha,\Omega}(\vec{f})(x)|\geq\frac{1}{4\gamma}\Big\},$$
$$B_{3,j}=\Big\{x\in\Omega_R: u_{x,\vec{f_j},\alpha}(r)<\mathfrak{M}_{\alpha,\Omega}(\vec{f_j})(x)-\frac{1}{4\gamma},\ \ {\rm if}\ d(r,\mathcal{R}_\alpha(\vec{f})(x))>\lambda\Big\}.$$
Since $\mathfrak{M}_{\alpha,\Omega}(\vec{f})(x)-u_{x,\vec{f},\alpha}(r)$ is controlled by
$$|\mathfrak{M}_{\alpha,\Omega}(\vec{f_j})(x)-\mathfrak{M}_{\alpha,\Omega}(\vec{f})(x)|
+\mathfrak{M}_{\alpha,\Omega}(\vec{f_j})(x)-u_{x,\vec{f_j},\alpha}(r)+|u_{x,\vec{f_j},\alpha}(r)-u_{x,\vec{f},\alpha}(r)|,$$
it follows that
$B\subset B_{1,j}\cup B_{2,j}\cup B_{3,j}.$
Note that
$B_{3,j}\subset\{x\in\Omega_R:\mathcal{R}_\alpha(\vec{f_j})(x)\subset\mathcal{R}_\alpha(\vec{f})(x)_{(\lambda)}\}.$
Therefore, we obtain
$\{x\in\Omega_R;\mathcal{R}_\alpha(\vec{f_j})(x)\nsubseteq\mathcal{R}_\alpha(\vec{f})(x)_{(\lambda)}\}\subset E\cup B_{1,j}\cup B_{2,j}.$
Thus, for any $x\in\Omega$, let $\vec{F_j^i}=(f_1,\ldots,
f_{i-1},f_{i,j}-f_i,f_{i+1,j},\ldots,f_{m,j})$, we have
$$\begin{array}{ll}
&|\mathfrak{M}_{\alpha,\Omega}(\vec{f_j})(x)-\mathfrak{M}_{\alpha,\Omega}(\vec{f})(x)|\\
&\leq\displaystyle\sum\limits_{i=1}^{m}\sup\limits_{0<r<\delta(x)}r^\alpha
\prod\limits_{\mu=1}^{i-1}A_{x,r}(f_\mu)\prod\limits_{\nu=i+1}^mA_{x,r}(f_{\nu,j})A_{x,r}(|f_{i,j}-f_{i}|)\\
&=\displaystyle\sum\limits_{i=1}^m\mathfrak{M}_{\alpha,\Omega}(\vec{F_j^i})(x).
\end{array}\eqno(5.8)$$We get
from (5.8) that
$B_{2,j}\subset\Big\{x\in\Omega_R:\sum\limits_{i=1}^m\mathfrak{M}_{\alpha,\Omega}(\vec{F_j^i})(x)\geq\frac{1}{4\gamma}\Big\}.$
Similarly we get
$B_{1,j}\subset\Big\{x\in\Omega_R:\sum\limits_{i=1}^m\mathfrak{M}_{\alpha,\Omega}(\vec{F_j^i})(x)\geq\frac{1}{2\gamma}\Big\}$

(5.10) together with the above properties implies
that
$$\begin{array}{ll}
|\{x\in\Omega_R;\mathcal{R}_\alpha(\vec{f_j})(x)\nsubseteq\mathcal{R}_\alpha(\vec{f})(x)_{(\lambda)}\}|
&\leq\displaystyle2\Big|\Big\{x\in\Omega_R:\sum\limits_{i=1}^m\mathfrak{M}_{\alpha,\Omega}(\vec{F_j^i})(x)\geq\frac{1}{4\gamma}\Big\}\Big|\\
&\leq\displaystyle2(4m\gamma)^q\sum\limits_{i=1}^m\|\mathfrak{M}_{\alpha,\Omega}(\vec{F_j^i})\|_{L^q(\Omega)}\\
&\leq C(m,\gamma,q,\alpha,d,p_1,\ldots,p_m)\epsilon,
\end{array}$$
for all $j\geq N_0$. This yields (5.1) and completes the
proof of Lemma \ref{l:5.1}. 
\end{proof}

For $1\leq l\leq n$, let $e_l=(0,\ldots,0,1,0,\ldots,0)$
be the canonical $l$-th base vector in $\mathbb{R}^n$. Fix
$1\leq i\leq m$, $h\neq0$ and $f_i\in L^{p}(\Omega)$ with
$1\leq p<\infty$, we define the functions $f_{i,h}^l$ and
$f_{\tau(h)}^{i,l}$ by setting
$f_{i,h}^l(x)=\frac{f_{\tau(h)}^{i,l}(x)-f_i(x)}{h}\ \ {\rm and}\ \ f_{\tau(h)}^{i,l}(x)=f_i(x+he_l).$
We know that
$f_{\tau(h)}^{i,l}\rightarrow f_i$ in $L^{p}(K)$ for all
$K\subset\subset\Omega$ when $h\rightarrow0$, and
if $f_i\in W^{1,p}(\Omega)$ with $p>1$ we have $f_{i,h}^l
\rightarrow D_l f_i$ in $L^{p}(K)$ when $h\rightarrow0$
(see \cite[7.11]{GT}). We also known that $f_{\tau(h)}^{i,l}
\rightarrow f_i$ in $L^{p_i}(\Omega)$ with $p_i\geq1$ when
$h\rightarrow0$.

Let $A,\,B$ be two subsets of $\mathbb{R}^n$, we define the
Hausdorff distance of $A$ and $B$ by
$$\pi(A,B):=\inf\{\delta>0: A\subset B_{(\delta)}\ {\rm and}\ B\subset A_{(\delta)}\}.$$

The following lemma tells us that how close the sets
$\mathcal{R}_\alpha(\vec{f})(x)$ and $\mathcal{R}_\alpha
(\vec{f})(x+he_l)$ are when $h$ is small enough.

\begin{lemma}\label{l:5.2}
Let $\vec{f} =(f_1,\ldots,f_m)\in L^{p_1}(\Omega)\times\cdots\times L^{p_m} (\Omega)$ with $1<p_1,\ldots,p_m<\infty$, $0<\frac{1}{q}- \frac{\alpha}{n}=\frac{1}{p_1}+\ldots+\frac{1}{p_m}<1$.  Then for $K\subset\subset\Omega$, $\lambda>0$ and $l=1,2,\ldots,n$ it holds that $|\{x\in K;\pi(\mathcal{R}_\alpha(\vec{f})(x),\mathcal{R}_\alpha(\vec{f})(x+he_l))>\lambda\}|\rightarrow0\ \ {\rm when}\ h\rightarrow 0.$
\end{lemma}

\begin{proof}
Fix $1\leq l\leq n$. It suffices
to show that
$$\lim\limits_{h\rightarrow0}|\{x\in K:\mathcal{R}_\alpha(\vec{f})(x+he_l)\nsubseteq\mathcal{R}_\alpha(\vec{f})(x)_{(\lambda)}\}|=0,\eqno(5.11)$$
$$\lim\limits_{h\rightarrow0}|\{x\in K:\mathcal{R}_\alpha(\vec{f})(x)\nsubseteq\mathcal{R}_\alpha(\vec{f})(x+he_l)_{(\lambda)}\}|=0.\eqno(5.12)$$
We only prove (5.11). 
Motivated by the idea in the proof of \cite[Lemma 2.3]{Lu2},
we now prove (5.11). By the same argument as in getting
(5.4), there exists $\gamma\in\mathbb{N}
\backslash\{0\}$ such that
$$K\subset\Big\{x\in K:u_{x,\vec{f},\alpha}(r)<\mathfrak{M}_{\alpha,\Omega}(\vec{f})(x)-\frac{1}{\gamma},
\ \ {\rm if}\ d(r,\mathcal{R}_\alpha(\vec{f})(x))>\lambda\Big\}\cup E:=B\cup E,\eqno(5.13)$$
where $E$ is a zero measurable set. Fix $h\in\mathbb{R}$,
let
$$B_{1,h}=\Big\{x\in K:|\mathfrak{M}_{\alpha,\Omega}(\vec{f})(x+he_l)-\mathfrak{M}_{\alpha,\Omega}(\vec{f})(x)|>\frac{1}{4\gamma}\Big\},$$
$$B_{2,h}=\Big\{x\in K: \sum\limits_{i=1}^{m}\mathfrak{M}_{\alpha,\Omega}(\vec{G_{l,h}^i})(x)>\frac{1}{2\gamma}\Big\},$$
where
$$\vec{G_{l,h}^i}=(f_1,\ldots,f_{i-1},f_{\tau(h)}^{i,l}-f_i,f_{\tau(h)}^{i+1,l},\ldots,f_{\tau(h)}^{m,l})\eqno(5.14)$$
and
$$\begin{array}{ll}
B_{3,h}&=\displaystyle\Big\{x\in\Omega: \exists r\in[\delta(x)-2|h|,\delta(x)]\ \ {\rm such\ that}\\
&\displaystyle\Big|r^\alpha\prod\limits_{i=1}^mA_{x+he_l,r}(|f_i|)-(\delta(x+he_l)-|h|)^\alpha
\prod\limits_{i=1}^mA_{x+he_l,\delta(x+he_l)-|h|}(|f_i|)\Big|>\frac{1}{8\gamma}\Big\}.
\end{array}$$

Now, for $h$ small enough, we shall prove that
$$\{x\in K:\mathcal{R}_\alpha(\vec{f})(x+he_l)\nsubseteq\mathcal{R}_\alpha(\vec{f})(x)_{(2\lambda)}\}\subset B_{1,h}\cup B_{2,h}\cup(B_{3,h}-he_l)\cup E=:B_h\eqno(5.15)$$
Choose $h_0\in(0,\lambda)$ such that
$K_{(2h_0)}\subset\Omega$. We want to show that for
$x\in B\backslash B_h$ with $|h|<\frac{1}{2}\min\{h_0,
\delta(x)\}$, there exists $r\in\mathcal{R}_\alpha(\vec{f})
(x+he_l)$ such that $d(r,\mathcal{R}_\alpha(\vec{f})(x))\leq
2\lambda$. Otherwise,
assume that $d(r,\mathcal{R}_\alpha(\vec{f})(x))>2\lambda$.
We consider the following two cases:

{\rm (i)} Suppose that $r<\delta(x)-|h|$. We get from (5.13)
that
$$\begin{array}{ll}
&\mathfrak{M}_{\alpha,\Omega}(\vec{f})(x+he_l)=\displaystyle r^\alpha\prod\limits_{i=1}^mA_{x+he_l,r}(|f_i|)=\displaystyle r^\alpha\prod\limits_{i=1}^mA_{x,r}(|f_{\tau(h)}^{i,l}|)\\
&\leq\displaystyle r^\alpha\Big|\prod\limits_{i=1}^mA_{x,r}(|f_{\tau(h)}^{i,l}|)-\prod\limits_{i=1}^mA_{x,r}(|f_i|)\Big|+r^\alpha\prod\limits_{i=1}^mA_{x,r}(|f_i|)\\
&\leq\displaystyle\sum\limits_{i=1}^{m}r^\alpha A_{x,r}(|f_{\tau(h)}^{i,l}-f_i|)\prod\limits_{\mu=1}^{i-1}A_{x,r}(|f_\mu|)
\prod\limits_{\nu=i+1}^mA_{x,r}(|f_{\tau(h)}^{\nu,l}|)+\mathfrak{M}_{\alpha,\Omega}(\vec{f})(x)-\frac{1}{\gamma}\\
&\leq\displaystyle\sum\limits_{i=1}^{m}\mathfrak{M}_{\alpha,\Omega}(\vec{G_{l,h}^i})(x)
+\mathfrak{M}_{\alpha,\Omega}(\vec{f})(x)-\frac{1}{\gamma}\\
&\leq\displaystyle\frac{1}{2\gamma}+\mathfrak{M}_{\alpha,\Omega}(\vec{f})(x)-\frac{1}{\gamma}
\leq\mathfrak{M}_{\alpha,\Omega}(\vec{f})(x)-\frac{1}{2\gamma},
\end{array}$$
where $\vec{G_{l,h}^i}$ is given as in (5.14). This yields that
$|\mathfrak{M}_{\alpha,\Omega}(\vec{f})(x)-\mathfrak{M}_{\alpha,
\Omega}(\vec{f})(x+he_l)|\geq\frac{1}{2\gamma}$, which yields
$x\in B_{1,h}$ and a contradiction.

{\rm (ii)} Suppose that $r\in[\delta(x)-|h|,\delta(x+he_l)]$.
Observe that $d(\delta(x)-|h|,\mathcal{R}_\alpha(\vec{f})(x))
>\lambda$ and $\delta(x+he_l)-|h|<\delta(x)$. Specially, when
$|h|$ is small enough we have
$d(\delta(x+he_l)-|h|,\mathcal{R}_\alpha(\vec{f})(x))>\lambda$.
We can write
$$\begin{array}{ll}
&\mathfrak{M}_{\alpha,\Omega}(\vec{f})(x+he_l)
=\displaystyle r^\alpha\prod\limits_{i=1}^mA_{x+he_l,r}(|f_i|)\\
&\leq\displaystyle\Big|r^\alpha\prod\limits_{i=1}^mA_{x+he_l,r}(|f_i|)
-(\delta(x+he_l)-|h|)^\alpha\prod\limits_{i=1}^mA_{x+he_l,\delta(x+he_l)-|h|}(|f_i|)\Big|\\
&\quad+\displaystyle(\delta(x+he_l)-|h|)^\alpha\Big|\prod\limits_{i=1}^mA_{x,\delta(x+he_l)-|h|}(|f_{\tau(h)}^{i,l}|)
-\prod\limits_{i=1}^mA_{x,\delta(x+he_l)-|h|}(|f_i|)\Big|\\
&\quad+\displaystyle(\delta(x+he_l)-|h|)^\alpha\prod\limits_{i=1}^mA_{x,\delta(x+he_l)-|h|}(|f_i|).
\end{array}\eqno(5.16)$$
One can easily check that
$$\begin{array}{ll}
&\displaystyle(\delta(x+he_l)-|h|)^\alpha\Big|\prod\limits_{i=1}^mA_{x,\delta(x+he_l)-|h|}(|f_{\tau(h)}^{i,l}|)
-\prod\limits_{i=1}^mA_{x,\delta(x+he_l)-|h|}(|f_i|)\Big|\\
&\leq\displaystyle\sum\limits_{i=1}^{m}\mathfrak{M}_{\alpha,\Omega}(\vec{G_{l,h}^i})(x),
\end{array}\eqno(5.17)$$
where $\vec{G_{l,h}^i}$ is given as in (5.14). Using (5.13) and
(5.16)-(5.17) we get
$$\mathfrak{M}_{\alpha,\Omega}(\vec{f})(x+he_l)
\leq\frac{1}{8\gamma}+\frac{1}{2\gamma}+\mathfrak{M}_{\alpha,\Omega}(\vec{f})(x)-\frac{1}{\gamma}
<\mathfrak{M}_{\alpha,\Omega}(\vec{f})(x)-\frac{1}{4\gamma},$$
which leads to $|\mathfrak{M}_{\alpha,\Omega}(\vec{f})(x)-
\mathfrak{M}_{\alpha,\Omega}(\vec{f})(x+he_l)|>\frac{1}{4\gamma}$.
Thus we have $x\in B_{1,h}$, which is a contradiction and
hence (5.15) holds.

It remains to show that
$$\lim\limits_{h\rightarrow0}|B_h|=0.\eqno(5.18)$$
Obviously, $|B_{3,h}-he_l|\rightarrow0$ when $h\rightarrow0$.
It suffices to show that $|B_{1,h}\cup B_{2,h}|\rightarrow0$
when $h\rightarrow0$. Let $\vec{f_h^l}=(f_{\tau(h)}^{1,l},
\ldots,f_{\tau(h)}^{m,l})$. By the similar argument as in
getting (5.8) we have
$|\mathfrak{M}_{\alpha,\Omega}(\vec{f})(x+he_l)-\mathfrak{M}_{\alpha,\Omega}(\vec{f})(x)|
\leq\sum\limits_{i=1}^{m}\mathfrak{M}_{\alpha,\Omega}(\vec{G_{l,h}^i})(x),$
where $\vec{G_{l,h}^i}$ is given as in (5.14). thus
$$\begin{array}{ll}
|B_{1,h}\cup B_{2,h}|
&\leq\displaystyle2\Big|\Big\{x\in K:\sum\limits_{i=1}^m\mathfrak{M}_{\alpha,\Omega}(\vec{G_{l,h}^i})(x)\geq\frac{1}{4\gamma}\Big\}\Big|\\
&\leq\displaystyle2\sum\limits_{i=1}^m\Big|\Big\{x\in K:\mathfrak{M}_{\alpha,\Omega}(\vec{G_{l,h}^i})(x)\geq\frac{1}{4m\gamma}\Big\}\Big|\\
&\leq\displaystyle2(4m\gamma)^q\sum\limits_{i=1}^m\prod\limits_{\mu=1}^{i-1}\|f_\mu\|_{L^{p_\mu}(\Omega)}
\|f_{\tau(h)}^{i,l}-f_i\|_{L^{p_i}(\Omega)}\prod\limits_{\nu=i+1}^m\|f_{\tau(h)}^{\nu,l}\|_{L^{p_\nu}(\Omega)},
\end{array}$$
which yields that $|B_{1,h}\cup B_{2,h}|\rightarrow0$
when $h\rightarrow0$, and then (5.18) holds. (5.18)
together with (5.15) yields (5.11). This completes
the proof of Lemma \ref{l:5.2}. 
\end{proof}

The following key lemma 
enable us to give the proofs of Theorems \ref{t:alpha=0continuous}-\ref{t:alpha>0continuous}.

\begin{lemma}\label{l:5.3}
Let $1<p_1,\ldots,p_m<\infty$,
$0<\sum_{i=1}^m1/p_i-\alpha/n=1/q<1$.  Let $\vec{f}
\in W^{1,p_1}(\Omega)\times\cdots
\times W^{1,p_m}(\Omega)$. Then for
$1\le l\le n$ and almost every $x\in\Omega$ we have
$$D_l\mathfrak{M}_{\alpha,\Omega}(\vec{f})(x)=\sum\limits_{i=1}^mr^\alpha\prod\limits_{1\leq j\neq i\leq m}A_{x,r}(|f_j|)A_{x,r}(D_l|f_i|)\ \ {\rm for\ }\ r\in\mathcal{R}_\alpha(\vec{f})(x),\ 0<r<\delta(x);\eqno(5.19)$$
\[
 D_l\mathfrak{M}_{\alpha,\Omega}(\vec{f})(x)= \left\{
\begin{array}{ll}
\displaystyle\sum\limits_{i=1}^m D_l|f_i|(x)\prod\limits_{1\leq j\neq i\leq m}|f_j(x)|&\ \ {\rm if}\ \alpha=0\ {\rm and}\ 0\in\mathcal{R}_\alpha(\vec{f})(x),\\
0,&\ \ {\rm if}\ \alpha>0\ {\rm and}\ 0\in\mathcal{R}_\alpha(\vec{f})(x).
\end{array}
\right.\eqno(5.20)
\]
\end{lemma}

\begin{proof}
Fix $1\leq l\leq n$. Without loss
of generality we may assume that all $f_i\geq0$. Let $K\subset\subset\Omega$. By Lemma
\ref{l:5.2} we can choose a sequence $\{s_k\}_{k=1}^\infty$, $s_k>0$
and $s_k\rightarrow0$ such that $\lim\limits_{k\rightarrow\infty}\pi
(\mathcal{R}_\alpha(\vec{f})(x),\mathcal{R}_\alpha(\vec{f})
(x+h_ke_l))=0$ a.e. $x\in K$. Then for any
$i=1,2,\ldots,m$, we have
$$\|f_{\tau(s_k)}^{i,l}-f_i\|_{L^{p_i}(K)}\rightarrow0\ \ {\rm as}\ k\rightarrow\infty,$$
$$\|f_{i,s_k}^l-D_lf_i\|_{L^{p_i}(K)}\rightarrow0\ \ {\rm as}\ k\rightarrow\infty,$$
$$\|M_\Omega(f_{\tau(s_k)}^{i,l}-f_i)\|_{L^{p_i}(K)}\rightarrow0\ \ {\rm as}\ k\rightarrow\infty,$$
$$\|M_\Omega (f_{i,s_k}^l-D_lf_i)\|_{L^{p_i}(K)}\rightarrow0\ \ {\rm as}\ k\rightarrow\infty,$$
$$\|(\mathfrak{M}_{\alpha,\Omega}(\vec{f}))_{s_k}^l-D_l\mathfrak{M}_{\alpha,\Omega}(\vec{f})\|_{L^q(K)}\rightarrow0\ \ {\rm as}\ k\rightarrow\infty,$$
where
$$(\mathfrak{M}_{\alpha,\Omega}(\vec{f}))_{s_k}^l(x)=\frac{\mathfrak{M}_{\alpha,\Omega}(\vec{f})(x+s_ke_l)
-\mathfrak{M}_{\alpha,\Omega}(\vec{f})(x)}{s_k}.$$
Furthermore, there exists a subsequence $\{h_k\}_{k=1}^\infty$
of $\{s_k\}_{k=1}^\infty$ and a measurable set $B_1\subset K$
such that $|K\backslash B_1|=0$ and
\begin{enumerate}
\item[{(i)}] $f_{\tau(h_k)}^{i,l}(x)\rightarrow f_i(x)$, $f_{i,h_k}^l(x)
\rightarrow D_lf_i(x)$, $M_\Omega(f_{\tau(h_k)}^{i,l}-f_i)(x)
\rightarrow0$, $M_\Omega(f_{i,h_k}^l-D_lf_i)(x)\rightarrow0$
and $(\mathfrak{M}_{\alpha,\Omega}(\vec{f}))_{h_k}^l(x)
\rightarrow D_l\mathfrak{M}_{\alpha,\Omega}(\vec{f})(x)$ when
$k\rightarrow\infty$ for any $x\in B_1$ and $i=1,2,\ldots,m$;

\item[{(ii)}]$\lim_{k\rightarrow\infty}\pi(\mathcal{R}_\alpha(\vec{f})(x),
\mathcal{R}_\alpha(\vec{f})(x+h_ke_l))=0$ for any $x\in B_1$.
\end{enumerate}
Let
$B_2:=\bigcap\limits_{k=1}^\infty\{x\in K:\mathfrak{M}_{\alpha,\Omega}(\vec{f})(x+h_ke_l)=u_{x+h_ke_l,\vec{f},\alpha}(0)\ {\rm if}\ 0\in\mathcal{R}_\alpha(\vec{f})(x+h_ke_l)\},$
$B_3:=\{x\in K:\mathfrak{M}_{\alpha,\Omega}(\vec{f})(x)=u_{x,\vec{f},\alpha}(0)\ {\rm if}\ 0\in\mathcal{R}_\alpha(\vec{f})(x)\}.$

One can easily check that $|K\backslash B_i|=0$ for any $i=2,3$.
Let $x\in B_1\cap B_2\cap B_3$ be a Lebesgue point of all $f_i$,
$f_{\tau(h_k)}^{i,l}$ and $D_l f_i$ and $r\in\mathcal{R}_\alpha
(\vec{f})(x)$ with $r<\delta(x)$, there exists radii
$r_k\in\mathcal{R}_\alpha(\vec{f})(x+h_ke_l)$ such that
$\lim_{k\rightarrow\infty}r_k=r$. We consider
two cases:

{\bf Case A} ($r>0$). Without loss of generality we assume
that all $r_k>0$. Write
$$\begin{array}{ll}
D_l\mathfrak{M}_{\alpha,\Omega}(\vec{f})(x)&=\displaystyle\lim\limits_{k\rightarrow\infty}
\frac{1}{h_k}(\mathfrak{M}_{\alpha,\Omega}(\vec{f})(x+h_ke_l)-\mathfrak{M}_{\alpha,\Omega}(\vec{f})(x))\\
&\leq\displaystyle\lim\limits_{k\rightarrow\infty}\frac{1}{h_k}(u_{x+h_ke_l,\vec{f},\alpha}(r_k)-u_{x,\vec{f},\alpha}(r_k))\\
&=\displaystyle\sum\limits_{i=1}^m\lim\limits_{k\rightarrow\infty}r_k^\alpha A_{x,r_k}(f_{i,h_k}^l)
\prod\limits_{\mu=1}^{i-1}A_{x,r_k}(f_{\tau(h_k)}^{\mu,l})\prod\limits_{\nu=i+1}^{m}A_{x,r_k}(f_\nu)
\end{array}\eqno(5.21)$$
Since $\lim\limits_{k\rightarrow\infty}|B(x,r_k)|=|B(x,r)|$,
$f_{\tau(h_k)}^{i,l}\chi_{B(x,r_k)}\rightarrow f_i\chi_{B(x,r)}$
and $f_{j,h_k}^{l}\chi_{B(x,r_k)}\rightarrow D_lf_i\chi_{B(x,r)}$
in $L^1(\Omega)$ as $k\rightarrow\infty$. Thus we have
$D_l\mathfrak{M}_{\alpha,\Omega}(\vec{f})(x)\leq\sum\limits_{i=1}^mr^\alpha A_{x,r}(D_lf_i)\prod\limits_{1\leq j\neq i\leq m}A_{x,r}(f_j).$
On the other hand,
$$\begin{array}{ll}
D_l\mathfrak{M}_{\alpha,\Omega}(\vec{f})(x)
&=\displaystyle\lim\limits_{k\rightarrow\infty}\frac{1}{h_k}(\mathfrak{M}_{\alpha,\Omega}(\vec{f})(x+h_ke_l)
-\mathfrak{M}_{\alpha,\Omega}(\vec{f})(x))\\
&\geq\displaystyle\sum\limits_{i=1}^mr^\alpha
\lim\limits_{k\rightarrow\infty}A_{x,r}(f_{i,h_k}^l)\prod\limits_{\mu=1}^{i-1}A_{x,r}(f_{\tau(h_k)}^{\mu,l})
\prod\limits_{\nu=i+1}^{m}A_{x,r}(f_\nu)\\
&=\displaystyle\sum\limits_{i=1}^mr^\alpha A_{x,r}(D_lf_i)\prod\limits_{1\leq j\neq i\leq m}A_{x,r}(f_j).
\end{array}\eqno(5.22)$$
Combining (5.22) with (5.21) yields (5.19).

{\bf Case B} ($r=0$). We consider two cases: 

(i) $\alpha=0$.
We can write
$$\begin{array}{ll}
D_l\mathfrak{M}_{\Omega}(\vec{f})(x)&=\displaystyle\lim\limits_{k\rightarrow\infty}\frac{1}{h_k}(\mathfrak{M}_{\Omega}(\vec{f})(x+h_ke_l)
-\mathfrak{M}_{\Omega}(\vec{f})(x))\\
&\geq \displaystyle\lim\limits_{k\rightarrow\infty}\frac{1}{h_k}\Big(\prod\limits_{i=1}^mf_i(x+h_ke_l)-\prod\limits_{i=1}^mf_i(x)\Big)\\
&=\displaystyle\sum\limits_{i=1}^m D_lf_i(x)\prod\limits_{1\leq j\neq i\leq m}f_j(x).
\end{array}\eqno(5.23)$$
If we have $r_k=0$ for infinitely many $k$, then
$$\begin{array}{ll}
D_l\mathfrak{M}_{\Omega}(\vec{f})(x)&=\displaystyle\lim\limits_{k\rightarrow\infty}\frac{1}{h_k}(\mathfrak{M}_{\Omega}(\vec{f})(x+h_ke_l)
-\mathfrak{M}_{\Omega}(\vec{f})(x))\\
&=\displaystyle\lim\limits_{k\rightarrow\infty}\frac{1}{h_k}\Big(\prod\limits_{i=1}^mf_i(x+h_ke_l)-\prod\limits_{i=1}^mf_i(x)\Big)\\
&=\displaystyle\sum\limits_{i=1}^m D_lf_i(x)\prod\limits_{1\leq j\neq i\leq m}f_j(x).
\end{array}$$
If there exists $k_0\in\mathbb{N}$ such that $r_k>0$
when $k\geq k_0$. We get from (5.21) that
$$D_l\mathfrak{M}_\Omega(\vec{f})(x)\leq\displaystyle\sum\limits_{i=1}^m\lim\limits_{k\rightarrow\infty}A_{x,r_k}(f_{i,h_k}^l)
\prod\limits_{\mu=1}^{i-1}A_{x,r_k}(f_{\tau(h_k)}^{\mu,l})\prod\limits_{\nu=i+1}^{m}A_{x,r_k}(f_\nu).
\eqno(5.24)$$
Since
$$\lim\limits_{k\rightarrow\infty}A_{x,r_k}(f_{j,h_k}^l)=\lim\limits_{k\rightarrow\infty}A_{x,r_k}(f_{j,h_k}^l-D_lf_j)
+\lim\limits_{k\rightarrow\infty}A_{x,r_k}(D_lf_j)\eqno(5.25)$$
and
$$|\lim\limits_{k\rightarrow\infty}A_{x,r_k}(f_{j,h_k}^l-D_lf_j)|
\leq\lim\limits_{k\rightarrow\infty}M_\Omega(f_{j,h_k}^l-D_lf_j)(x)=0.\eqno(5.26)$$
Form (5.25)-(5.26), we have
$$\lim\limits_{k\rightarrow\infty}A_{x,r_k}(f_{j,h_k}^l)=D_lf_j(x).\eqno(5.27)$$
Similarly, we have
$$\lim\limits_{k\rightarrow\infty}A_{x,r_k}(f_{\tau(h_k)}^{j,l})=f_j(x).\eqno(5.28)$$
It follows from (5.24) and (5.27)-(5.28) that
$$D_l\mathfrak{M}_\Omega(\vec{f})(x)\leq\sum\limits_{i=1}^m D_lf_i(x)\prod\limits_{1\leq j\neq i\leq m}f_j(x),$$
which together with (5.23) implies that
$$D_l\mathfrak{M}_\Omega(\vec{f})(x)=\sum\limits_{i=1}^m D_lf_i(x)\prod\limits_{1\leq j\neq i\leq m}f_j(x).$$

(ii) $0<\alpha<md$. One can easily check that
$\mathfrak{M}_{\alpha,\Omega}(\vec{f})(x)=0$. It follows
that
$$D_l\mathfrak{M}_{\alpha,\Omega}(\vec{f})(x)\geq\lim\limits_{k\rightarrow\infty}\frac{1}{h_k}\mathfrak{M}_{\alpha,\Omega}(\vec{f})(x+h_ke_l)\geq0.\eqno(5.29)$$
If we have $r_k=0$ for infinitely many $k$, we can conclude
that
$$D_l\mathfrak{M}_{\alpha,\Omega}(\vec{f})(x)=\lim\limits_{k\rightarrow\infty}\frac{1}{h_k}(\mathfrak{M}_{\alpha,\Omega}(\vec{f})(x+h_ke_l)
-\mathfrak{M}_{\alpha,\Omega}(\vec{f})(x))=0.$$
If there exists $k_0\in\mathbb{N}$ such that $r_k>0$ when
$k\geq k_0$. We get from (5.21) that
$$D_l\mathfrak{M}_{\alpha,\Omega}(\vec{f})(x)\leq\sum\limits_{i=1}^m\lim\limits_{k\rightarrow\infty}r_k^\alpha A_{x,r_k}(f_{i,h_k}^l)
\prod\limits_{\mu=1}^{i-1}A_{x,r_k}(f_{\tau(h_k)}^{\mu,l})\prod\limits_{\nu=i+1}^{m}A_{x,r_k}(f_\nu).$$
Combining this inequality with (5.27)-(5.28) implies that
$$D_l\mathfrak{M}_{\alpha,\Omega}(\vec{f})(x)\leq 0.$$
This together with (5.29) yields that
$$D_l\mathfrak{M}_{\alpha,\Omega}(\vec{f})(x)=0.$$
This proves (5.19) and (5.20) for a.e. $x\in K$.
Since $K\subset\subset\Omega$ is arbitrary, this gives
the claim in $\Omega$. 
\end{proof}

\begin{lemma}[Lemma 2.11, \cite{Lu2}]\label{l:5.4}
Let $A_j\subset\mathbb{R}^n$ be measurable sets
and let $h_k\in\mathbb{R}^n$ such that $|h_k|\rightarrow0$
when $k\rightarrow\infty$. Then we can find a subsequence
of $\{h_{k_i}\}$ such that for every $j$ and for almost
every $x\in A_j$ we have $x+h_{k_i}\in A_j$ when $i$ is
large enough.
\end{lemma}

\begin{lemma}\label{l:5.5}
Let $1\leq\alpha<mn$
and $K\subset\subset\Omega$. Let $\vec{f}=(f_1,\ldots,f_m)$
and $\vec{f_j}=(f_{1,j},\ldots,f_{m,j})$. Suppose that one
of the following conditions holds:
\begin{enumerate}
\item[{(i)}] $\alpha=0$, $f_{i,j}\rightarrow f_i$ in $W^{1,p_i}
(\Omega)$ as $j\rightarrow\infty$ for $1<p_i<\infty$,
$\frac{1}{q}=\frac{1}{p_1}+\ldots+\frac{1}{p_m}$ with $1<q<\infty$;

\item[{(ii)}] $1\leq\alpha<mn$, $f_{i,j}\rightarrow f_i$ in
$W^{1,p_i}(\Omega)$ as $j\rightarrow\infty$ for $1<p_i<\infty$,
$0<\frac{1}{q}=\frac{1}{p_1}+\ldots+\frac{1}{p_m}-
\frac{\alpha-1}{n}<1$ and $|\Omega|<\infty$.
\end{enumerate}
Then for all $1\leq l\leq n$, we have
$$\lim\limits_{j\rightarrow\infty}\|D_l\mathfrak{M}_{\alpha,\Omega}(\vec{f_j})-D_l\mathfrak{M}_{\alpha,\Omega}(\vec{f})\|_{L^q(K_j)}= 0,\eqno(5.30)$$
where
$K_j:=\{x\in K:\delta(x)\in\mathcal{R}_\alpha(\vec{f_j})(x)\cap\mathcal{R}_\alpha(\vec{f})(x)\}.$
\end{lemma}

\begin{proof}
Since $\mathfrak{M}_{\alpha,
\Omega}(\vec{f_j})-\mathfrak{M}_{\alpha,\Omega}(\vec{f})\in
W^{1,q}(\Omega)$ for $q>1$. Thus we can choose a sequence
$\{h_k\}_{k=1}^\infty$, $h_k\rightarrow 0^{+}$ such that
for all $j\geq1$ and $1\leq l\leq n$ we have
$$\frac{\mathfrak{M}_{\alpha,\Omega}(\vec{f_j})(x+h_k e_l)-\mathfrak{M}_{\alpha,\Omega}(\vec{f})(x)}{h_k}\rightarrow D_l(\mathfrak{M}_{\alpha,\Omega}(\vec{f_j})-\mathfrak{M}_{\alpha,\Omega}(\vec{f}))(x)\ \ {\rm as}\ k\rightarrow\infty$$
for almost every $x\in K$. Fix $1\leq l\leq n$. By Lemma
\ref{l:5.4}, there exists a subsequence $\{s_k\}_{k=1}^\infty$ of
$\{h_k\}_{k=1}^\infty$, $s_k\rightarrow0$ as $k\rightarrow
\infty$ such that for almost every $x\in K_j$, we have
$x+s_k e_l\in K_j$ for all $j$ when $k$ is large enough.
Obviously, $\mathfrak{M}_{\alpha,\Omega}(\vec{f_j})(x)=
u_{x,\vec{f_j},\alpha}(\delta(x))$ and $\mathfrak{M}_{\alpha,
\Omega}(\vec{f})(x)=u_{x,\vec{f},\alpha}(\delta(x))$.
Thus we have
$$\begin{array}{ll}
\mathfrak{M}_{\alpha,\Omega}(\vec{f_j})(x)-\mathfrak{M}_{\alpha,\Omega}(\vec{f})(x)
&=\displaystyle\delta(x)^\alpha\Big(\prod\limits_{i=1}^mA_{x,\delta(x)}(f_{i,j})
-\prod\limits_{i=1}^mA_{x,\delta(x)}(f_{i})\Big)\\
&=\displaystyle\sum\limits_{i=1}^{m}u_{x,\vec{F_j^i},\alpha}(\delta(x)),
\end{array}\eqno(5.31)$$
where $\vec{F_j^i}$ is given as in (5.8). Similarly we have
$$\mathfrak{M}_{\alpha,\Omega}(\vec{f_j})(x+s_ke_l)-\mathfrak{M}_{\alpha,\Omega}(\vec{f})(x+s_ke_l)
=\sum\limits_{i=1}^mu_{x+s_ke_l,\vec{F_j^i},\alpha}(\delta(x+s_ke_l)).\eqno(5.32)$$
We get from (5.31)-(5.32) that
$$\begin{array}{ll}
&|D_l\mathfrak{M}_{\alpha,\Omega}(\vec{f_j})(x)-D_l\mathfrak{M}_{\alpha,\Omega}(\vec{f})(x)|\\
&\leq\displaystyle\sum\limits_{i=1}^m\Big|\lim\limits_{k\rightarrow\infty}
\frac{u_{x+s_ke_l,\vec{F_j^i},\alpha}(\delta(x+s_ke_l))-u_{x,\vec{F_j^i},\alpha}(\delta(x))}{s_k}\Big|
\end{array}\eqno(5.33)$$
for almost every $x\in K_j$. By the continuity of
$u_{x,\vec{F_j^i},\alpha}(r)$. For almost every
$x\in\Omega$, there exists a sequence of numbers
$\{r_\ell\}_{\ell=1}^\infty$, $r_\ell>0$, $r_\ell\rightarrow1$
as $\ell\rightarrow\infty$ such that
$$u_{x,\vec{F_j^i},\alpha}(\delta(x))=\lim\limits_{\ell\rightarrow\infty}u_{x,\vec{F_j^i},\alpha}(r_\ell\delta(x))
=\lim\limits_{\ell\rightarrow\infty}\mathcal{A}_{r_\ell}^\alpha(\vec{F_j^i})(x).\eqno(5.34)$$
We consider the following three cases:

{\bf Case 1.} Assume (i) hold. By Lemma \ref{l:3.6} we have
$\mathcal{A}_{r_\ell}(\vec{F_j^i})\in W^{1,q}(\Omega)$.
Thus we can choose a subsequence $\{\iota_k\}_{k=1}^\infty$
of $\{s_k\}_{k=1}^\infty$, $\iota_k\rightarrow 0^{+}$ such
that
$$\frac{\mathcal{A}_{r_\ell}(\vec{F_j^i})(x+\iota_ke_l)-\mathcal{A}_{r_\ell}(\vec{F_j^i})(x)}{\iota_k}\rightarrow D_l(\mathcal{A}_{r_\ell}^\alpha(\vec{F_j^i}))(x)\ \ {\rm as}\ k\rightarrow\infty$$
for all $j,\,\ell\geq1$ $1\leq l\leq n$ and almost every
$x\in K$. Using (5.33)-(5.34) and Lemma \ref{l:3.6} again, for almost every $x\in K_j$, we obtain
that
$$\begin{array}{ll}
&|D_l\mathfrak{M}_{\Omega}(\vec{f_j})(x)-D_l\mathfrak{M}_{\Omega}(\vec{f})(x)|\\
&\leq\displaystyle\sum\limits_{i=1}^m\Big|\lim\limits_{k\rightarrow\infty}\lim\limits_{\ell\rightarrow\infty}
\frac{\mathcal{A}_{r_\ell}(\vec{F_j^i})(x+\iota_ke_l)-\mathcal{A}_{r_\ell}(\vec{F_j^i})(x)}{\iota_k}\Big|\\
&\leq\displaystyle\sum\limits_{i=1}^m\lim\limits_{\ell\rightarrow\infty}\Big|\lim\limits_{k\rightarrow\infty}
\frac{\mathcal{A}_{r_\ell}(\vec{F_j^i})(x+\iota_ke_l)-\mathcal{A}_{r_\ell}(\vec{F_j^i})(x)}{\iota_k}\Big|\\
&\leq\displaystyle2\sum\limits_{i=1}^{m}|D_l(\mathcal{A}_{r_\ell}(\vec{F_j^i}))(x)|\\
&\leq\displaystyle2\Big(\sum\limits_{\mu=1}^{i-1}\mathfrak{M}_{\Omega}(\vec{F^{i,\mu,j}})(x)+\mathfrak{M}_{\Omega}(\vec{F_{i,j}})(x)
+\sum\limits_{\nu=i+1}^{m}\mathfrak{M}_{\Omega}(\vec{F_{i,\nu,j}})(x)\Big)\\
&=:E_{i,j}(x)
\end{array}\eqno(5.35)$$
 where
$$\vec{F^{i,\mu,j}}=(f_1,\ldots,f_{\mu-1},|\nabla f_\mu|,f_{\mu+1},\ldots,f_{i-1},f_{i,j}-f_i,f_{i+1,j},\ldots,f_{m,j}),$$
$$\vec{F_{i,j}}=(f_1,\ldots,f_{i-1},|\nabla(f_{i,j}-f_i)|,f_{i+1,j},\ldots,f_{m,j}),$$
$$\vec{F_{i,\nu,j}}=(f_1,\ldots,f_{i-1},f_{i,j}-f_i,f_{i+1,j},\ldots,f_{\nu-1,j},|\nabla f_{\nu,j}|,f_{\nu+1,j},\ldots,f_{m,j}).$$
(5.30) follows form (5.35), Minkowski's inequality and
Lemma 3.2.

{\bf Case 2.}\quad Assume that (ii) holds. By Lemma \ref{l:3.4}
we have $\mathcal{A}_{r_\ell}^\alpha(\vec{F_j^i})\in
W^{1,q}(\Omega)$. Thus we can choose a subsequence
$\{\iota_k\}_{k=1}^\infty$ of $\{s_k\}_{k=1}^\infty$,
$\iota_k\rightarrow 0^{+}$ such that for all
$j,\,\ell\geq1$ and $1\leq l\leq n$ we have that
$$\frac{\mathcal{A}_{r_\ell}^\alpha(\vec{F_j^i})(x+\iota_ke_l)-\mathcal{A}_{r_\ell}^\alpha(\vec{F_j^i})(x)}{\iota_k}\rightarrow D_l(\mathcal{A}_{r_\ell}^\alpha(\vec{F_j^i}))(x)\ \ {\rm as}\ k\rightarrow\infty$$
for almost every $x\in K$. By Lemma \ref{l:3.4} again, for almost every $x\in\Omega$, we have
$$\begin{array}{ll}
&\displaystyle\Big|\lim\limits_{k\rightarrow\infty}\frac{\mathcal{A}_{r_\ell}^\alpha(\vec{F_j^i})(x+\iota_ke_l)
-\mathcal{A}_{r_\ell}^\alpha(\vec{F_j^i})(x)}{\iota_k}\Big|\\
&\leq\displaystyle\alpha\mathfrak{M}_{\alpha-1,\Omega}(\vec{F_j^i})(x)
+2\Big(\sum\limits_{\mu=1}^{i-1}\mathfrak{M}_{\alpha,\Omega}(\vec{F^{i,\mu,j}})(x)+\mathfrak{M}_{\alpha,\Omega}(\vec{F_{i,j}})(x)
+\sum\limits_{\nu=i+1}^{m}\mathfrak{M}_{\alpha,\Omega}(\vec{F_{i,\nu,j}})(x)\Big)\\
&=:G_{i,j}(x).
\end{array}\eqno(5.36)$$
For almost every $x\in K_j$, it follows from (5.33)-(5.34)
and (5.36) that
$$\begin{array}{ll}
&|D_l\mathfrak{M}_{\alpha,\Omega}(\vec{f_j})(x)-D_l\mathfrak{M}_{\alpha,\Omega}(\vec{f})(x)|\\
&\leq\displaystyle\sum\limits_{i=1}^m\Big|\lim\limits_{k\rightarrow\infty}\lim\limits_{\ell\rightarrow\infty}
\frac{\mathcal{A}_{r_\ell}^\alpha(\vec{F_j^i})(x+\iota_ke_l)-\mathcal{A}_{r_\ell}^\alpha(\vec{F_j^i})(x)}{\iota_k}\Big|\\
&\leq\displaystyle\sum\limits_{i=1}^m\lim\limits_{\ell\rightarrow\infty}\Big|\lim\limits_{k\rightarrow\infty}
\frac{\mathcal{A}_{r_\ell}^\alpha(\vec{F_j^i})(x+\iota_ke_l)-\mathcal{A}_{r_\ell}^\alpha(\vec{F_j^i})(x)}{\iota_k}\Big|\\
&\leq\displaystyle\sum\limits_{i=1}^mG_{i,j}(x).
\end{array}\eqno(5.37)$$
The fact that $|\Omega|<\infty$, together with
Lemma 3.2, Minkowski's inequality and H\"{o}lder's inequality gives that
$\|G_{i,j}\|_{L^q(K_j)}\rightarrow0\ \ {\rm as}\ \ j\rightarrow\infty.$
This together with (5.37) yields (5.30). This completes
the proof of Lemma \ref{l:5.5}. 
\end{proof}

\begin{lemma}[Lemma 2.9, \cite{Lu2}]\label{l:5.6}
Let $f\in W^{1,p}(\Omega)$ for $1<p<\infty$. Let $r_k$ and
$h_k$ be positive real numbers so that $h_k\rightarrow 0$,
$r_k\leq\delta(x)$ for every $k$ and $r_k\rightarrow\delta(x)$
as $k\rightarrow\infty$. Moreover, assume that
$r_k\leq\delta(x+h_ke_l)$ for all $k$ and some $1\leq l\leq n$. Then, it holds that
$$\lim\limits_{k\rightarrow\infty}\frac{1}{|B(x,r_k)|}\int_{B(x,r_k)}\frac{f(y+h_ke_l)-f(y)}{h_k}dy
=\frac{1}{|B(x,\delta(x))|}\int_{B(x,\delta(x))}D_lf(y)dy.$$
\end{lemma}

\begin{lemma}[Corollary 2.7, \cite{Lu2}).]\label{l:5.7}
 Let $1<p<\infty$ and $A$ be a measurable subset of $\Omega$.
Let $f_j$ be a sequence in $W_{\rm loc}^{1,1}(\Omega)$ so that
$f_j$ converges to zero in the sense of distributions :
$$\int_\Omega f_j(x)\varphi(x)dx\rightarrow 0\ \ {\rm as}\ j\rightarrow\infty \quad \hbox{for every\ } \varphi\in\mathcal{C}_0^\infty(\Omega).$$
Suppose
that $|\nabla f_j(x)|\leq F(x)+F_j(x)$ for almost every
$x\in\Omega$ and $\|F\|_{L^p(\Omega)}<\infty$ and
$\|F_j\|_{L^p(\Omega)}\rightarrow0$ as $j\rightarrow\infty$.
Suppose also that for all $\epsilon>0$ and $1\leq l\leq n$, it holds that
$|\{x\in A: D_l f_j(x)>\epsilon\}|\rightarrow0\ \ {\rm as}\ j\rightarrow\infty$
or
$|\{x\in A: D_l f_j(x)<-\epsilon\}|\rightarrow0\ \ {\rm as}\ j\rightarrow\infty.$
Then$$\lim\limits_{ j\rightarrow\infty}\|D_lf_j\|_{L^p(A)}=0.$$
\end{lemma}

\subsection{Proofs of Theorems \ref{t:alpha=0continuous}-\ref{t:alpha>0continuous}}

\begin{proof}[ Proof of Theorem \ref{t:alpha=0continuous}.]
Let each $f_j\in W^{1,p_j}(\Omega)$. Let $\frac{1}{q}=\frac{1}{p_1}+\ldots+
\frac{1}{p_m}$ with $1<q<\infty$. For any $i=1,2,\ldots,m$,
let $f_{i,j}\rightarrow f_i$ in $W^{1,p_i}(\Omega)$ when
$j\rightarrow\infty$. We may assume that all $f_i\geq0$ and
$f_{i,j}\geq0$. For convenience, we denote by $\vec{f_j}=
(f_{1,j},\ldots,f_{m,j})$. For $1\le i\le m$, let
$$\vec{f^i}=(f_1,\ldots,f_{i-1},D_nf_i,f_{i+1},\ldots,f_m),$$
$$\vec{f_j^i}=(f_{1,j},\ldots,f_{i-1,j},D_nf_{i,j},f_{i+1,j},
\ldots,f_{m,j}),$$
$$\vec{g^i}=(f_1,\ldots,f_{i-1},|\nabla f_i|,f_{i+1},\ldots,f_m),$$
$$\vec{g_j^i}=(f_{1,j},\ldots,f_{i-1,j},|\nabla f_{i,j}|,f_{i+1,j},\ldots,f_{m,j}).$$
By (5.8) and Minkowski's inequality we have
$$\|\mathfrak{M}_\Omega(\vec{f_j})-\mathfrak{M}_\Omega(\vec{f})\|_{L^q(\Omega)}
\leq\displaystyle\sum\limits_{i=1}^m\prod\limits_{\mu=1}^{i-1}\|f_\mu\|_{L^{p_\mu}(\Omega)}
\|f_{i,j}-f_i\|_{L^{p_i}(\Omega)}\prod\limits_{\nu=i+1}^m\|f_{\nu,j}\|_{L^{p_\nu}(\Omega)},$$
where $\vec{F_j^i}$ is given as in (5.8). It follows that
$$\|\mathfrak{M}_\Omega(\vec{f_j})-\mathfrak{M}_\Omega(\vec{f})\|_{L^q(\Omega)}\rightarrow0\ \ {\rm as}\ j\rightarrow\infty.\eqno(5.38)$$
Therefore, it suffices to show that
$$\|D_l(\mathfrak{M}_\Omega(\vec{f_j})-\mathfrak{M}_\Omega(\vec{f}))\|_{L^q(\Omega)}\rightarrow 0\ \ {\rm when}\ j\rightarrow\infty\eqno(5.39)$$
for any $l=1,2,\ldots,n$. We will prove (5.39) for $l=n$ and
the other cases are analogous.

Fix $\epsilon>0$, there exists $K\subset\subset\Omega$
such that $\sum_{i=1}^m\|\mathfrak{M}_\Omega(\vec{g^i})
\|_{L^q(\Omega\backslash K)}<\epsilon$. By absolute
continuity, there exists $\eta>0$ such that $\sum_{i=1}^m
\|\mathfrak{M}_\Omega(\vec{g^i})\|_{L^q(A)}<\epsilon$
whenever $A$ is a measurable set such that $A\subset K$ and
$|A|<\eta$. By Theorem \ref{t:alpha=0pointwise} we have
$$\begin{array}{ll}
|\nabla(\mathfrak{M}_\Omega(\vec{f_j})-\mathfrak{M}_\Omega(\vec{f}))(x)|
&\leq\displaystyle2\sum\limits_{i=1}^m(\mathfrak{M}_\Omega(\vec{g^i})(x)+\mathfrak{M}_\Omega(\vec{g_j^i})(x))\\
&\leq\displaystyle 4\sum\limits_{i=1}^m\mathfrak{M}_\Omega(\vec{g^i})(x)+2\sum\limits_{i=1}^m|\mathfrak{M}_\Omega(\vec{g_j^i})(x)
-\mathfrak{M}_\Omega(\vec{g^i})(x)|\\
&=:\displaystyle 4\sum\limits_{i=1}^m\mathfrak{M}_\Omega(\vec{g^i})(x)+F_j(x)
\end{array}\eqno(5.40)$$
for almost every $x\in\Omega$. One can easily check that
$$F_j(x)\leq2\sum\limits_{i=1}^m\Big(\sum\limits_{\mu=1}^{i-1}\mathfrak{M}_\Omega(\vec{I_{\mu,j}})(x)
+\sum\limits_{\nu=i+1}^{m}\mathfrak{M}_\Omega(\vec{J_{\nu,j}})(x)+\mathfrak{M}_\Omega(\vec{K_{i,j}})(x)\Big),$$
where $\vec{I_{\mu,j}}=(f_1,\ldots,f_{\mu-1},f_{\mu,j}-f_{\mu},f_{\mu+1,j},\ldots,f_{i-1,j},|\nabla f_{i,j}|,f_{i+1,j}\ldots,f_{m,j}),$
$$\vec{J_{\nu,j}}=(f_1,\ldots,f_{i-1},|\nabla f_i|,f_{i+1},\ldots,f_{\nu-1},f_{\nu,j}-f_{\nu},f_{\nu+1,j},\ldots,f_{m,j}),$$
$$\vec{K_{i,j}}=(f_1,\ldots,f_{i-1},|\nabla(f_{i,j}-f_i)|,f_{i+1,j},\ldots,f_{m,j}).$$
It is easy to see that
$$\|F_j\|_{L^q(\Omega)}\rightarrow0\ \ {\rm as}\ j\rightarrow\infty.\eqno(5.41)$$
This yields that there exists $N_0\in\mathbb{N}\backslash\{0\}$
such that $\|F_j\|_{L^q(\Omega)}<\epsilon$ for all $j\geq N_0$,
which together with (5.40) and Minkowski's inequality implies
that
$$\begin{array}{ll}
\|D_n(\mathfrak{M}_\Omega(\vec{f_j})-\mathfrak{M}_\Omega(\vec{f}))\|_{L^q(\Omega\backslash K)}
&\leq\displaystyle\Big\|4\sum\limits_{i=1}^m\mathfrak{M}_\Omega(\vec{g^i})\Big\|_{L^q(\Omega\backslash K)}+\|F_j\|_{L^q(\Omega)}\\
&\leq\displaystyle4\sum\limits_{i=1}^m\|\mathfrak{M}_\Omega(\vec{g^i})\|_{L^q(\Omega\backslash K)}+\|F_j\|_{L^q(\Omega)}\leq 5\epsilon
\end{array}$$
for any $j\geq N_0$. It follows that
$$\|D_n(\mathfrak{M}_\Omega(\vec{f_j})-\mathfrak{M}_\Omega(\vec{f}))\|_{L^q(\Omega\backslash K)}\rightarrow 0\ \ {\rm as}\ j\rightarrow\infty.\eqno(5.42)$$
Thus, to prove (5.39), we only need to show that
$$\|D_n(\mathfrak{M}_\Omega(\vec{f_j})-\mathfrak{M}_\Omega(\vec{f}))\|_{L^q(K)}\rightarrow 0\ \ {\rm as}\ j\rightarrow\infty.\eqno(5.43)$$
Let
$$G=\{x\in K:\delta(x)\not\in\mathcal{R}(\vec{f})(x)\}.$$
To prove (5.43), it is enough to prove that
$$\|D_n(\mathfrak{M}_\Omega(\vec{f_j})-\mathfrak{M}_\Omega(\vec{f}))\|_{L^q(G)}\rightarrow 0\ \ {\rm as}\ j\rightarrow\infty,\eqno(5.44)$$
$$\|D_n(\mathfrak{M}_\Omega(\vec{f_j})-\mathfrak{M}_\Omega(\vec{f}))\|_{L^q(K\backslash G)}\rightarrow 0\ \ {\rm as}\ j\rightarrow\infty.\eqno(5.45)$$

{\bf Step 1. Proof of (5.44)}. Since the sets $\mathcal{R}
(\vec{f})(x)$ are compact, we can choose $\gamma>0$ such
that
$$|\{x\in G:\mathcal{R}(\vec{f})(x)\nsubseteq[0,\delta(x)-\gamma]\}|=:|A_\gamma|<\frac{\eta}{4}.\eqno(5.46)$$
As we already observed, for almost every $x\in\Omega$ and
$i=1,2,\ldots,m$, the function $u_{x,\vec{f^i}}$ is
uniformly continuous on $[0,\delta(x)]$. Thus, for almost
every $x\in\Omega$, the function $\sum_{i=1}^mu_{x,\vec{f^i}}$
is uniformly continuous on $[0,\delta(x)]$ and we can find
$\gamma(x)\in(0,\gamma)$ such that
$$\Big|\sum\limits_{i=1}^mu_{x,\vec{f^i}}(r_1)-\sum\limits_{i=1}^mu_{x,\vec{f^i}}(r_2)\Big|<\epsilon\ \ {\rm whenever}\ |r_1-r_2|<\gamma(x).$$
We can write $K$ as
$K=\Big(\bigcup\limits_{k=1}^\infty\Big\{x\in K;\frac{1}{k}<\gamma(x)<\gamma\Big\}\Big)\bigcup\mathcal{N},$
where $|\mathcal{N}|=0$. It follows that there exists
$\beta\in(0,\gamma)$ such that
$$\aligned\Big|\Big\{x\in K:\Big|\sum\limits_{i=1}^mu_{x,\vec{f^i}}(r_1)&-\sum\limits_{i=1}^mu_{x,\vec{f^i}}(r_2)\Big|\geq\epsilon\ {\rm for\ some}\ r_1,r_2\ {\rm with}\ |r_1-r_2|<\beta\Big\}\Big|\\&=:|A_\beta|<\frac{\eta}{4}.\endaligned$$
By Lemma \ref{l:5.1}, there exists $N_1\in\mathbb{N}\backslash\{0\}$
such that
$$|\{x\in K;\mathcal{R}(\vec{f_j})(x)\nsubseteq\mathcal{R}(\vec{f})(x)_{(\beta)}\}|:=|K^j|<\frac{\eta}{4}\ \ {\rm when}\ j\geq N_1.$$
Invoking Lemma \ref{l:5.3}, for almost every $x\in\Omega$, we have
$$\begin{array}{ll}
|D_n(\mathfrak{M}_\Omega(\vec{f_j})-\mathfrak{M}_\Omega(\vec{f}))(x)|
&=\displaystyle\Big|\sum\limits_{i=1}^mu_{x,\vec{f_j^i}}(r_1)-\sum\limits_{i=1}^mu_{x,\vec{f^i}}(r_2)\Big|\\
&\leq\displaystyle\sum\limits_{i=1}^m|u_{x,\vec{f_j^i}}(r_1)-u_{x,\vec{f^i}}(r_1)|+\Big|\sum\limits_{i=1}^mu_{x,\vec{f^i}}(r_1)
-\sum\limits_{i=1}^mu_{x,\vec{f^i}}(r_2)\Big|.
\end{array}$$
for any $r_1\in\mathcal{R}(\vec{f_j})(x)$ and $r_2\in\mathcal{R}
(\vec{f})(x)$ with $r_1,\,r_2<\delta(x)$.
Next we consider two cases:

(a) $r_1>0$. We can write
$$\begin{array}{ll}
|u_{x,\vec{f_j^i}}(r_1)-u_{x,\vec{f^i}}(r_1)|&\leq\displaystyle\sum\limits_{\mu=1}^{i-1}\mathfrak{M}_\Omega(\vec{G_{\mu}^{j}})(x)
+\sum\limits_{\nu=i+1}^{m}\mathfrak{M}_\Omega(\vec{H_{\nu}^j})(x)+\mathfrak{M}_\Omega(\vec{I_i^j})(x)\\
&:=\mathcal{G}_{i,j}(x),
\end{array}\eqno(5.47)$$
where $$\vec{G_{\mu}^j}=(f_1,\ldots,f_{\mu-1},f_{\mu,j}-f_{\mu},f_{\mu+1,j},\ldots,f_{i-1,j},D_nf_{i,j},f_{i+1,j}\ldots,f_{m,j}),$$
$$\vec{H_\nu^j}=(f_1,\ldots,f_{i-1},D_nf_i,f_{i+1},\ldots,f_{\nu-1},f_{\nu,j}-f_{\nu},f_{\nu+1,j},\ldots,f_{m,j}),$$
$$\vec{I_i^j}=(f_1,\ldots,f_{i-1},D_n(f_{i,j}-f_i),f_{i+1,j},\ldots,f_{m,j}).$$

(b) $r_1=0$. We can control $|u_{x,\vec{f_j^i}}(r_1)-u_{x,\vec{f^i}}(r_1)|$ by 
\begin{eqnarray*}
&&\displaystyle\sum\limits_{\mu=1}^{i-1}\Big(\prod\limits_{l_1=1}^{\mu-1}f_{l_1}(x)\Big)(f_{\mu,j}(x)-f_{\mu}(x))
\Big(\prod\limits_{l_2=\mu+1}^{i-1}f_{l_2,j}(x)\Big)|D_nf_{i,j}(x)|\Big(\prod\limits_{l_3=i+1}^{m}f_{l_3,j}(x)\Big)\\
&&+\displaystyle\sum\limits_{\nu=i+1}^{m}\Big(\prod\limits_{l_1=1}^{i-1}f_{l_1}(x)\Big)|D_nf_i(x)|
\Big(\prod\limits_{l_2=i+1}^{\nu-1}f_{l_2}(x)\Big)|f_{\nu,j}(x)-f_\nu(x)|\Big(\prod\limits_{l_3=\nu+1}^{m}f_{l_3,j}(x)\Big)\\
&&+\displaystyle\big(\prod\limits_{l_1=1}^{i-1}f_{l_1}(x)\big)|D_n(f_{i,j}-f_i)(x)|\Big(\prod\limits_{l_2=i+1}^{m}f_{l_2,j}(x)\Big).
\end{eqnarray*}
From the above we can claim that for almost every $x\in\Omega$,
$$|D_n(\mathfrak{M}_\Omega(\vec{f_j})-\mathfrak{M}_\Omega(\vec{f}))(x)|\leq \sum\limits_{i=1}^m\mathcal{G}_{i,j}(x)+\Big|\sum\limits_{i=1}^mu_{x,\vec{f^i}}(r_1)
-\sum\limits_{i=1}^mu_{x,\vec{f^i}}(r_2)\Big|\eqno(5.48)$$
for any $r_1\in\mathcal{R}(\vec{f_j})(x)$ and $r_2\in
\mathcal{R}(\vec{f})(x)$ with $r_1,\,r_2<\delta(x)$. 

One
can easily check that
$$\lim\limits_{j\rightarrow\infty}\|\mathcal{G}_{i,j}\|_{L^q(\Omega)}=0\ \ \forall 1\leq i\leq m.\eqno(5.49)$$
Then there exists $N_2\in\mathbb{N}\backslash\{0\}$ such
that
$$\sum\limits_{i=1}^m\|\mathcal{G}_{i,j}\|_{L^q(\Omega)}<\epsilon\ \ \forall j\geq N_2.\eqno(5.50)$$

If $x\in G\backslash(A_\gamma\cup A_\beta\cup K^j)$ we can
choose $r_1\in\mathcal{R}(\vec{f_j})(x)$ and
$r_2\in\mathcal{R}(\vec{f})(x)$ such that $r_1,\,r_2<\delta(x)$,
$|r_1-r_2|<\beta$ and
$$\Big|\sum\limits_{i=1}^mu_{x,\vec{f^i}}(r_1)-\sum\limits_{i=1}^mu_{x,\vec{f^i}}(r_2)\Big|<\epsilon.\eqno(5.51)$$
On the other hand, for any $r_1\in\mathcal{R}(\vec{f_j})(x)$
and $r_2\in\mathcal{R}(\vec{f})(x)$ with $r_1,\,r_2<\delta(x)$,
we have that for any $i=1,2,\ldots,m$,
$$\Big|\sum\limits_{i=1}^mu_{x,\vec{f^i}}(r_1)-\sum\limits_{i=1}^mu_{x,\vec{f^i}}(r_2)\Big|
\leq 2\sum\limits_{i=1}^m\mathfrak{M}_\Omega(\vec{f^i})(x).\eqno(5.52)$$
Note that $|A_\gamma\cup A_\beta\cup K^j|<\eta$ for $j\geq N_1$.
Thus by (5.48) and (5.50)-(5.52) we have
$$\begin{array}{ll}
&\|D_n(\mathfrak{M}_\Omega(\vec{f_j})-\mathfrak{M}_\Omega(\vec{f}))\|_{L^q(G)}\\
&\leq\displaystyle\Big\|\sum\limits_{i=1}^m\mathcal{G}_{i,j}\Big\|_{L^q(\Omega)}+\|\epsilon\|_{L^q(G\backslash(A_\gamma\cup A_\beta\cup K^j))}
+2\Big\|\sum\limits_{i=1}^m\mathfrak{M}_\Omega(\vec{f^i})\Big\|_{L^q(A_\gamma\cup A_\beta\cup K^j)}\\
&\leq (3+|K|)\epsilon,
\end{array}$$
for any $j\geq\max\{N_1,N_2\}$, which gives (5.44).

{\bf Step 2. Proof of (5.45).} Let $\{h_k\}_{k=1}^\infty$
be a sequence of numbers such that $h_k\rightarrow 0^{+}$
as $k\rightarrow\infty$. Following from the notations in
\cite{Lu2}, we set
$$B^j:=\{x\in K\backslash G:\delta(x)\in\mathcal{R}(\vec{f_j})(x)\},$$
$$B^{+}:=\{x\in K\backslash G:\delta(x+h_ke_l)\geq\delta(x)\ {\rm for\ infinitely\ many}\ k\},$$
$$B^{-}:=\{x\in K\backslash G:\delta(x+h_ke_l)\leq\delta(x)\ {\rm for\ infinitely\ many}\ k\}.$$
Note that $K\backslash G\subset B^j\cup B^{+}\cup B^{-}$.
Invoking Lemma \ref{l:5.5} we obtain
$$\|D_n(\mathfrak{M}_\Omega(\vec{f_j})-\mathfrak{M}_\Omega(\vec{f_j}))\|_{L^q(B^j)}\rightarrow0\ \ {\rm as}\ j\rightarrow\infty.\eqno(5.53)$$

Below we treat the case when $x\in B^{+}$. We know that
for almost every $x\in B^{+}$ we have $x+h_ke_n\in B^{+}$
when $k$ is large enough. It follows that
$$\mathfrak{M}_\Omega(\vec{f})(x+h_ke_n)\geq u_{x+h_ke_n,\vec{f}}(\delta(x))$$
for $k$ large enough. Thus by Lemma \ref{l:5.6} we have
$$\begin{array}{ll}
D_n\mathfrak{M}_\Omega(\vec{f})(x)&=\displaystyle\lim\limits_{k\rightarrow\infty}
\frac{1}{h_k}\Big(\mathfrak{M}_\Omega(\vec{f})(x+h_ke_n)-\mathfrak{M}_\Omega(\vec{f})(x)\Big)\\
&\geq\displaystyle\limsup\limits_{k\rightarrow\infty}\frac{1}{h_k}(u_{x+h_ke_n,\vec{f}}(\delta(x))-u_{x,\vec{f}}(\delta(x)))\\
&=\displaystyle\limsup\limits_{k\rightarrow\infty}\sum\limits_{i=1}^mA_{x,\delta(x)}(f_{i,h_k}^n)
\prod\limits_{\mu=i}^{i-1}A_{x,\delta(x)}(f_\mu)\prod\limits_{\nu=i+1}^mA_{x,\delta(x)}(f_{\tau(h_k)}^{\nu,n})\\
&=\displaystyle\sum\limits_{i=1}^m u_{x,\vec{f^i}}(\delta(x))
\end{array}$$
for almost every $x\in B^{+}$. Combining this inequality
with Lemma \ref{l:5.3} implies that
$$D_l\mathfrak{M}_\Omega(\vec{f})(x)\geq\sum\limits_{i=1}^m u_{x,\vec{f^i}}(r)\eqno(5.54)$$
for all $r\in\mathcal{R}(\vec{f})(x)$ (equality if $r<\delta(x)$).
Let us recall the definition of $\beta$ and $K^j$. We know that
$\mathcal{R}(\vec{f_j})(x)\subset\mathcal{R}(\vec{f})(x)_{(\beta)}$
when $j\geq N_1$ and $x\in\Omega\backslash K^j$. It follows that
for every $x\in B^{+}\backslash(K^j\cup B^j)$ with $j\geq N_1$,
there exists $r_j\in\mathcal{R}(\vec{f_j})(x)$, $r_j<\delta(x)$
such that $|r_j-r|\leq\beta$ for some $r\in\mathcal{R}(\vec{f})(x)$
(Here $r$ may be $\delta(x)$). Since $x\in B^{+}\backslash(K^j
\cup B^j)$, then $r_j<\delta(x)$. By Lemma \ref{l:5.3} and (5.54) we
obtain that
$$\aligned
D_n(\mathfrak{M}_\Omega(\vec{f})-\mathfrak{M}_\Omega(\vec{f_j}))(x)&\geq\displaystyle\sum\limits_{i=1}^m u_{x,\vec{f^i}}(r)-\sum\limits_{i=1}^m u_{x,\vec{f_j^i}}(r_j)\\
&\geq\displaystyle\sum\limits_{i=1}^m u_{x,\vec{f^i}}(r)-\sum\limits_{i=1}^m u_{x,\vec{f^i}}(r_j)+\sum\limits_{i=1}^m u_{x,\vec{f^i}}(r_j)-\sum\limits_{i=1}^m u_{x,\vec{f_j^i}}(r_j)\\
&=:I_{1,j}(x)+I_{2,j}(x)
\endaligned\eqno(5.55)$$
for almost every $x\in B^{+}\backslash(K^j\cup B^j)$ with
$j\geq N_1$. By the continuity of the functions $u_{x,
\vec{f^i}}$ on $[0,\delta(x)]$ we have
$$|\{x\in\Omega:|I_{1,j}(x)|\geq\epsilon/2\}|\rightarrow0\ \ {\rm as}\ j\rightarrow\infty.\eqno(5.56)$$
On the other hand, by the similar argument as in getting
(5.47) we have
$$|I_{2,j}(x)|\leq\sum\limits_{i=1}^m\mathcal{G}_{i,j}(x),$$
where $\mathcal{G}_{i,j}$ is given as in (5.47). We get from
(5.49) that
$$\Big|\Big\{x\in\Omega:\sum\limits_{i=1}^m\mathcal{G}_{i,j}(x)\geq\epsilon\Big\}\Big|\rightarrow0\ \ {\rm as}\ j\rightarrow\infty.$$
It follows that
$$|\{x\in\Omega:|I_{2,j}(x)|\geq\epsilon/2\}|\rightarrow0\ \ {\rm as}\ j\rightarrow\infty.\eqno(5.57)$$
Then by (5.55)-(5.57) we have
$$\begin{array}{ll}
&\quad|\{x\in B^{+}\backslash(K^j\cup B^j):D_n(\mathfrak{M}_\Omega(\vec{f})-\mathfrak{M}_\Omega(\vec{f_j}))(x)\leq -\epsilon\}|\\
&\leq\displaystyle\Big|\Big\{x\in B^{+}\backslash(K^j\cup B^j):I_{1,j}(x)+I_{2,j}(x)\leq-\epsilon\Big\}\Big|\rightarrow0\ \ {\rm as}\ j\rightarrow\infty.
\end{array}\eqno(5.58)$$
By (5.40)-(5.41), (5.58) and Lemma \ref{l:5.7} we have
$$\|D_n(\mathfrak{M}_\Omega(\vec{f_j})-\mathfrak{M}_\Omega(\vec{f}))\|_{L^q(B^{+}\backslash(K^j\cup B^j))}\rightarrow0\ \ {\rm as}\ j\rightarrow\infty.\eqno(5.59)$$
One the other hand, by (5.48)-(5.49) and (5.52) we have
$$\|D_n(\mathfrak{M}_\Omega(\vec{f_j})-\mathfrak{M}_\Omega(\vec{f}))\|_{L^q(B^{+}\cap K^j)}\leq\Big\|\sum\limits_{i=1}^m\mathcal{G}_{i,j}\Big\|_{L^q(\Omega)}+2\Big\|\sum\limits_{i=1}^m\mathfrak{M}_\Omega(\vec{f^i})\Big\|_{L^q(K^j)}
\leq 3\epsilon$$
for any $j\geq\max\{N_1,N_2\}$, which gives
$$\|D_n(\mathfrak{M}_\Omega(\vec{f_j})-\mathfrak{M}_\Omega(\vec{f}))\|_{L^q(B^{+}\cap K^j)}\rightarrow0\ \ {\rm as}\ j\rightarrow\infty.\eqno(5.60)$$
Combining (5.59) with (5.60) and (5.53) yields that
$$\|D_n(\mathfrak{M}_\Omega(\vec{f_j})-\mathfrak{M}_\Omega(\vec{f}))\|_{L^q(B^{+})}\rightarrow0\ \ {\rm as}\ j\rightarrow\infty.\eqno(5.61)$$

It remains to treat the case when $x\in B^{-}$. We know that
for almost every $x\in B^{-}$ we have $x+h_ke_n\in B^{-}$
when $k$ is large enough. It follows that
$$\mathfrak{M}_\Omega(\vec{f})(x+h_ke_n)=u_{x+h_ke_n,\vec{f}}(\delta(x+h_ke_n))\leq u_{x+h_ke_n,\vec{f}}(\delta(x))$$
for $k$ large enough. By the similar arguments as in getting
(5.54) we have
$$D_n\mathfrak{M}_\Omega(\vec{f})(x)\leq\sum\limits_{i=1}^m u_{x,\vec{f^i}}(r)\eqno(5.62)$$
for all $r\in\mathcal{R}(\vec{f})(x)$ (equality if $r<\delta(x)$).
Let us recall the definition of $\beta$ and $K^j$. We know that
$\mathcal{R}(\vec{f_j})(x)\subset\mathcal{R}(\vec{f})(x)_{(\beta)}$
when $j\geq N_1$ and $x\in\Omega\backslash K^j$. It follows that
for every $x\in B^{-}\backslash K^j$ with $j\geq N_1$, there
exists $r_j\in\mathcal{R}(\vec{f_j})(x)$, $r_j<\delta(x)$ such
that $|r_j-r|\leq\beta$ for some $r\in\mathcal{R}(\vec{f})(x)$
(Here $r$ may be $\delta(x)$). By the similar argument as in
getting (5.55) we have
$$D_n(\mathfrak{M}_\Omega(\vec{f})-\mathfrak{M}_\Omega(\vec{f_j}))(x)\leq I_{1,j}(x)+I_{2,j}(x)\eqno(5.63)$$
for almost every $x\in B^{-}\backslash(K^j\cup B^j)$ with
$j\geq N_1$. This together with (5.56)-(5.57) yields that
$$\begin{array}{ll}
&\quad|\{x\in B^{-}\backslash(K^j\cup B^j):D_n(\mathfrak{M}_\Omega(\vec{f})-\mathfrak{M}_\Omega(\vec{f_j}))(x)\geq\epsilon\}|\\
&\leq |\{x\in B^{-}\backslash(K^j\cup B^j):I_{1,j}(x)+I_{2,j}(x)\geq\epsilon\}|\rightarrow0\ \ {\rm as}\ j\rightarrow\infty.
\end{array}\eqno(5.64)$$
On the other hand, by (5.64) and the similar argument as
in getting (5.59) we have
$$\|D_n(\mathfrak{M}_\Omega(\vec{f})-\mathfrak{M}_\Omega(\vec{f_j}))\|_{L^q(B^{-}\backslash(K^j\cup B^j))}\rightarrow0\ \ {\rm as}\ j\rightarrow\infty.\eqno(5.65)$$
Using the similar argument as in getting (5.60) we get
$$\|D_n(\mathfrak{M}_\Omega(\vec{f_j})-\mathfrak{M}_\Omega(\vec{f_j}))\|_{L^q(B^{-}\cap K^j)}\rightarrow0\ \ {\rm as}\ j\rightarrow\infty.\eqno(5.66)$$
Combining (5.65)-(5.66) with (5.53) implies that
$$\|D_n(\mathfrak{M}_\Omega(\vec{f})-\mathfrak{M}_\Omega(\vec{f_j}))\|_{L^q(B^{-})}\rightarrow0\ \ {\rm as}\ j\rightarrow\infty.\eqno(5.67)$$
(5.45) follows from (5.53), (5.61) and (5.67). This finishes
the proof of Theorem \ref{t:alpha=0continuous}. 
\end{proof}

\begin{proof}[Proof of Theorem \ref{t:alpha>0continuous}.]
Let $\vec{f}=(f_1,\ldots,f_m)$ with each $f_j\in W^{1,p_j}(\Omega)$
and $1<p_j<\infty$. Let $0<\frac{1}{q}=\frac{1}{p_1}+\ldots+
\frac{1}{p_m}-\frac{\alpha-1}{n}<1$ and $\frac{1}{q^{*}}=
\frac{1}{p_1}+\ldots+\frac{1}{p_m}-\frac{\alpha}{n}$.
Obviously, $q<q^{*}$. For any $1\leq i\leq m$, let
$f_{i,j}\rightarrow f_i$ in $W^{1,p_i}(\Omega)$ when
$j\rightarrow\infty$. We may assume that all $f_i\geq0$
and $f_{i,j}\geq0$. We also let $\vec{f_j}=(f_{1,j},\ldots,
f_{m,j})$. By (5.8), H\"{o}lder's inequality,
Minkowski's inequality and the fact
that $|\Omega|<\infty$ we have
$$\|\mathfrak{M}_{\alpha,\Omega}(\vec{f_j})-\mathfrak{M}_{\alpha,\Omega}(\vec{f})\|_{L^q(\Omega)}
\leq\sum\limits_{i=1}^m\|\mathfrak{M}_{\alpha,\Omega}(\vec{F_j^i})\|_{L^q(\Omega)}\leq|\Omega|^{\frac{1}{q}-\frac{1}{q^{*}}}
\sum\limits_{i=1}^m\|\mathfrak{M}_{\alpha,\Omega}(\vec{F_j^i})\|_{L^{q^{*}}(\Omega)},\eqno(5.68)$$
where $\vec{F_j^i}$ is given as in (5.8). Using (5.68) and
Lemma \ref{l:3.2} we have
$$\|\mathfrak{M}_{\alpha,\Omega}(\vec{f_j})-\mathfrak{M}_{\alpha,\Omega}(\vec{f})\|_{L^q(\Omega)}\rightarrow0\ \ {\rm as}\ j\rightarrow\infty.\eqno(5.69)$$
For any $1\leq l\leq n$, it suffices to show that
$$\|D_l(\mathfrak{M}_{\alpha,\Omega}(\vec{f_j})-\mathfrak{M}_{\alpha,\Omega}(\vec{f}))\|_{L^q(\Omega)}\rightarrow 0\ \ {\rm as}\ j\rightarrow\infty\eqno(5.70)$$

We will prove (5.70)
for $l=n$. Let $\vec{f^i},\,\vec{f_j^i},\,\vec{g^i},\,\vec{g_j^i}$ be
given as in the proof of Theorem \ref{t:alpha=0continuous}. Fix $\epsilon>0$,
there exists $K\subset\subset\Omega$ such that
$\sum_{i=1}^m\|\mathfrak{M}_{\alpha,\Omega}
(\vec{f^i})\|_{L^q(\Omega\backslash K)}<\epsilon$. By
absolute continuity, there exists $\eta>0$ such that
$\|\mathfrak{M}_{\alpha-1,\Omega}(\vec{f})\|_{L^q(A)}<\epsilon$ and
$\sum_{i=1}^m\|\mathfrak{M}_{\alpha,\Omega}(\vec{g^i})
\|_{L^q(A)}<\epsilon$ whenever $A$ is a measurable set
such that $A\subset K$ and $|A|<\eta$. By Theorem \ref{t:alpha>0pointwise}
we have
$$\begin{array}{ll}
&|\nabla(\mathfrak{M}_{\alpha,\Omega}(\vec{f_j})-\mathfrak{M}_{\alpha,\Omega}(\vec{f}))(x)|\\
&\leq\displaystyle\alpha(\mathfrak{M}_{\alpha-1,\Omega}(\vec{f})(x)
+\mathfrak{M}_{\alpha-1,\Omega}(\vec{f_j})(x))+
2\sum\limits_{i=1}^m(\mathfrak{M}_{\alpha,\Omega}(\vec{g^i})(x)+\mathfrak{M}_{\alpha,\Omega}(\vec{g_j^i})(x))\\
&\leq\displaystyle 2\alpha\mathfrak{M}_{\alpha-1,\Omega}(\vec{f})(x)+
4\sum\limits_{i=1}^m\mathfrak{M}_{\alpha,\Omega}(\vec{g^i})(x)+\alpha|\mathfrak{M}_{\alpha-1,\Omega}(\vec{f_j})(x)-
\mathfrak{M}_{\alpha-1,\Omega}(\vec{f})(x)|\\
&\quad+\displaystyle 2\sum\limits_{i=1}^m|\mathfrak{M}_{\alpha,\Omega}(\vec{g_j^i})(x)
-\mathfrak{M}_{\alpha,\Omega}(\vec{g^i})(x)|\\
&=:\displaystyle 2\alpha\mathfrak{M}_{\alpha-1,\Omega}(\vec{f})(x)+
4\sum\limits_{i=1}^m\mathfrak{M}_{\alpha,\Omega}(\vec{g^i})(x)+G_j(x), \quad\hbox{for} \ a.\ e.\ x \in \Omega
\end{array}\eqno(5.71)$$ One can easily check that
$$G_j(x)\leq\sum\limits_{i=1}^m\mathfrak{M}_{\alpha-1,\Omega}(\vec{F_j^i})(x)
+2\sum\limits_{i=1}^m(\sum\limits_{\mu=1}^{i-1}\mathfrak{M}_{\alpha,\Omega}(\vec{I_{\mu,j}})(x)
+\sum\limits_{\nu=i+1}^{m}\mathfrak{M}_{\alpha,\Omega}(\vec{J_{\nu,j}})(x)+\mathfrak{M}_{\alpha,\Omega}(\vec{K_{i,j}})(x)),$$
where $\vec{F_j^i}$ is given as in (5.8) and $\vec{I_{\mu,j}},
\,\vec{J_{\nu,j}},\,\vec{K_{i,j}}$ are given as in the
proof of Theorem \ref{t:alpha=0continuous}. It is easy to see that
$$\|G_j\|_{L^q(\Omega)}\rightarrow0\ \ {\rm as}\ j\rightarrow\infty,\eqno(5.72)$$
which yields that there exists $N_0\in\mathbb{N}\backslash\{0\}$
such that $\|G_j\|_{L^q(\Omega)}<\epsilon$ for all $j\geq N_0$.
This together with (5.71) and Minkowski's inequality we have
$$\begin{array}{ll}
&\|D_n(\mathfrak{M}_{\alpha,\Omega}(\vec{f_j})-\mathfrak{M}_{\alpha,\Omega}(\vec{f}))\|_{L^q(\Omega\backslash K)}
\\&\leq\displaystyle\Big\|2\alpha\mathfrak{M}_{\alpha-1,\Omega}(\vec{f})(x)+
4\sum\limits_{i=1}^m\mathfrak{M}_{\alpha,\Omega}(\vec{g^i})(x)\Big\|_{L^q(\Omega\backslash K)}+\|G_j\|_{L^q(\Omega)}\leq(2\alpha+5)\epsilon, \quad\hbox{for}\  j\geq N_0.
\end{array}$$
It follows that
$$\|D_n(\mathfrak{M}_{\alpha,\Omega}(\vec{f_j})-\mathfrak{M}_{\alpha,\Omega}(\vec{f}))\|_{L^q(\Omega\backslash K)}\rightarrow 0\ \ {\rm as}\ j\rightarrow\infty.\eqno(5.73)$$
Thus, to prove (5.70), it suffices to show that
$$\|D_n(\mathfrak{M}_{\alpha,\Omega}(\vec{f_j})-\mathfrak{M}_{\alpha,\Omega}(\vec{f}))\|_{L^q(K)}\rightarrow 0\ \ {\rm as}\ j\rightarrow\infty.\eqno(5.74)$$
Let
$G=\{x\in K:\delta(x)\not\in\mathcal{R}_\alpha(\vec{f})(x)\}.$
To prove (5.74), it is sufficient to prove that
$$\|D_n(\mathfrak{M}_{\alpha,\Omega}(\vec{f_j})-\mathfrak{M}_{\alpha,\Omega}(\vec{f}))\|_{L^q(G)}\rightarrow 0\ \ {\rm as}\ j\rightarrow\infty,\eqno(5.75)$$
$$\|D_n(\mathfrak{M}_\Omega(\vec{f_j})-\mathfrak{M}_\Omega(\vec{f}))\|_{L^q(K\backslash G)}\rightarrow 0\ \ {\rm as}\ j\rightarrow\infty.\eqno(5.76)$$

{\bf Step 1. Proof of (5.75).}\quad Since the sets
$\mathcal{R}_\alpha(\vec{f})(x)$ are compact, we can
choose $\gamma>0$ such that
$$|\{x\in G:\mathcal{R}_\alpha(\vec{f})(x)\nsubseteq[0,\delta(x)-\gamma]\}|=:|A_\gamma|<\frac{\eta}{4}.\eqno(5.77)$$
As we already observed, for almost every $x\in\Omega$ and
$i=1,2,\ldots,m$, the function $u_{x,\vec{f^i},\alpha}$
is uniformly continuous on $[0,\delta(x)]$. Thus, for almost
every $x\in\Omega$, the function $\sum_{i=1}^mu_{x,\vec{f^i},\alpha}$
is uniformly continuous on $[0,\delta(x)]$ and we can find
$\gamma(x)\in(0,\gamma)$ such that
$$\Big|\sum\limits_{i=1}^mu_{x,\vec{f^i},\alpha}(r_1)-\sum\limits_{i=1}^mu_{x,\vec{f^i},\alpha}(r_2)\Big|<\epsilon\ \ {\rm whenever}\ |r_1-r_2|<\gamma(x).$$
We can write $K$ as
$K=\Big(\bigcup\limits_{k=1}^\infty\Big\{x\in K;\frac{1}{k}<\gamma(x)<\gamma\Big\}\Big)\bigcup\mathcal{N},$
where $|\mathcal{N}|=0$. It follows that there exists
$\beta\in(0,\gamma)$ such that
$$\Big|\Big\{x\in K:\Big|\sum\limits_{i=1}^mu_{x,\vec{f^i},\alpha}(r_1)-\sum\limits_{i=1}^mu_{x,\vec{f^i},\alpha}(r_2)\Big|\geq\epsilon\ {\rm for\ some}\ r_1,r_2\ {\rm with}\ |r_1-r_2|<\beta\Big\}\Big|:=|A_\beta|<\frac{\eta}{4}.$$
By Lemma \ref{l:5.1}, there exists $j_1\in\mathbb{N}\backslash\{0\}$
such that
$$|\{x\in K;\mathcal{R}_\alpha(\vec{f_j})(x)\nsubseteq\mathcal{R}_\alpha(\vec{f})(x)_{(\beta)}\}|:=|K^j|<\frac{\eta}{4}\ \ {\rm when}\ j\geq j_1.$$
Let $\vec{f_j^i}=(f_{1,j},\ldots,f_{i-1,j},D_lf_{i,j},f_{i+1,j},
\ldots,f_{m,j})$. Invoking Lemma \ref{l:5.3}, for almost every $x\in\Omega$,
we have
$$\begin{array}{ll}
&|D_n(\mathfrak{M}_{\alpha,\Omega}(\vec{f_j})-\mathfrak{M}_{\alpha,\Omega}(\vec{f}))(x)|\\
&\leq\displaystyle\sum\limits_{i=1}^m|u_{x,\vec{f_j^i},\alpha}(r_1)-u_{x,\vec{f^i},\alpha}(r_1)|+\Big|\sum\limits_{i=1}^mu_{x,\vec{f^i},\alpha}(r_1)
-\sum\limits_{i=1}^mu_{x,\vec{f^i},\alpha}(r_2)\Big|.
\end{array}$$
for any $r_1\in\mathcal{R}_\alpha(\vec{f_j})(x)$ and
$r_2\in\mathcal{R}_\alpha(\vec{f})(x)$ with $r_1,\,r_2<\delta(x)$.
When $r_1=0$, we have $|u_{x,\vec{f_j^i},\alpha}(r_1)-u_{x,
\vec{f^i},\alpha}(r_1)|=0$. When $r_1>0$. We can write
$$\begin{array}{ll}
|u_{x,\vec{f_j^i},\alpha}(r_1)-u_{x,\vec{f^i},\alpha}(r_1)|
&\leq\displaystyle\sum\limits_{\mu=1}^{i-1}\mathfrak{M}_{\alpha,\Omega}(\vec{G_{\mu}^{j}})(x)
+\sum\limits_{\nu=i+1}^{m}\mathfrak{M}_{\alpha,\Omega}(\vec{H_{\nu}^{j}})(x)+\mathfrak{M}_{\alpha,\Omega}(\vec{I_{i}^{j}})(x)\\
&:=\mathcal{K}_{i,j}(x),
\end{array}\eqno(5.78)$$
where $\vec{G_{\mu}^j},\,\vec{H_{\nu}^{j}},\,\vec{I_{i}^j}$
are given as in (5.47). Thus we have that for almost every
$x\in\Omega$,
$$|D_n(\mathfrak{M}_{\alpha,\Omega}(\vec{f_j})-\mathfrak{M}_{\alpha,\Omega}(\vec{f}))(x)|\leq \sum\limits_{i=1}^m\mathcal{K}_{i,j}(x)+\Big|\sum\limits_{i=1}^mu_{x,\vec{f^i},\alpha}(r_1)
-\sum\limits_{i=1}^mu_{x,\vec{f^i},\alpha}(r_2)\Big|\eqno(5.79)$$
for any $r_1\in\mathcal{R}_\alpha(\vec{f_j})(x)$ and
$r_2\in\mathcal{R}_\alpha(\vec{f})(x)$ with $r_1,\,r_2<\delta(x)$.
Clearly,
$$\lim\limits_{j\rightarrow\infty}\|\mathcal{K}_{i,j}\|_{L^{q^{*}}(\Omega)}=0\ \ \forall 1\leq i\leq m.$$
It follows that
$$\lim\limits_{j\rightarrow\infty}\|\mathcal{K}_{i,j}\|_{L^{q}(\Omega)}=0\ \ \forall 1\leq i\leq m.\eqno(5.80)$$
Then there exists $N_2\in\mathbb{N}\backslash\{0\}$ such that
$$\sum\limits_{i=1}^m\|\mathcal{K}_{i,j}\|_{L^q(\Omega)}<\epsilon\ \ \forall j\geq N_2.\eqno(5.81)$$

If $x\in G\backslash(A_\gamma\cup A_\beta\cup K^j)$ we can
choose $r_1\in\mathcal{R}_\alpha(\vec{f_j})(x)$ and
$r_2\in\mathcal{R}_\alpha(\vec{f})(x)$ such that $r_1,\,r_2<\delta(x)$,
$|r_1-r_2|<\beta$ and
$$\Big|\sum\limits_{i=1}^mu_{x,\vec{f^i},\alpha}(r_1)-\sum\limits_{i=1}^mu_{x,\vec{f^i},\alpha}(r_2)\Big|<\epsilon.\eqno(5.82)$$
On the other hand, for any $r_1\in\mathcal{R}_\alpha(\vec{f_j})(x)$
and $r_2\in\mathcal{R}_\alpha(\vec{f})(x)$ with $r_1,\,r_2<\delta(x)$,
we have that for any $i=1,2,\ldots,m$,
$$\Big|\sum\limits_{i=1}^mu_{x,\vec{f^i},\alpha}(r_1)-\sum\limits_{i=1}^mu_{x,\vec{f^i},\alpha}(r_2)\Big|
\leq\sum\limits_{i=1}^m|u_{x,\vec{f^i},\alpha}(r_1)-u_{x,\vec{f^i},\alpha}(r_2)|\leq 2\sum\limits_{i=1}^m\mathfrak{M}_{\alpha,\Omega}(\vec{f^i})(x).\eqno(5.83)$$
Note that $|A_\gamma\cup A_\beta\cup K^j|<\eta$ for $j\geq N_1$.
Thus by (5.79), (5.81)-(5.83) and H\"{o}lder's inequality we have
$$\begin{array}{ll}
&\|D_n(\mathfrak{M}_{\alpha,\Omega}(\vec{f_j})-\mathfrak{M}_{\alpha,\Omega}(\vec{f}))\|_{L^q(G)}\\
&\leq\displaystyle\Big\|\sum\limits_{i=1}^m\mathcal{K}_{i,j}\Big\|_{L^q(\Omega)}+\|\epsilon\|_{L^q(G\backslash(A_\gamma\cup A_\beta\cup K^j))}
+2\Big\|\sum\limits_{i=1}^m\mathfrak{M}_{\alpha,\Omega}(\vec{f^i})\Big\|_{L^q(A_\gamma\cup A_\beta\cup K^j)}\leq (3+|\Omega|)\epsilon,
\end{array}$$
for any $j\geq\max\{N_1,N_2\}$, which gives (5.75).

{\bf Step 2. Proof of (5.76).} Let $\{h_k\}_{k=1}^\infty$
be a sequence of numbers such that $h_k\rightarrow 0^{+}$
as $k\rightarrow\infty$. Following from the notations in
\cite{Lu2}, we set
$B^j:=\{x\in K\backslash G:\delta(x)\in\mathcal{R}_\alpha(\vec{f_j})(x)\},$
$B^{+}:=\{x\in K\backslash G:\delta(x+h_ke_n)\geq\delta(x)\ {\rm for\ infinitely\ many}\ k\}$ and $B^{-}:=\{x\in K\backslash G:\delta(x+h_ke_n)\leq\delta(x)\ {\rm for\ infinitely\ many}\ k\}.$
Note that $K\backslash G\subset B^j\cup B^{+}\cup B^{-}$.
Invoking Lemma \ref{l:5.5} we obtain
$$\|D_n(\mathfrak{M}_{\alpha,\Omega}(\vec{f_j})-\mathfrak{M}_{\alpha,\Omega}(\vec{f}))\|_{L^q(B^j)}\rightarrow0\ \ {\rm as}\ j\rightarrow\infty.\eqno(5.84)$$

Below we treat the case when $x\in B^{+}$. We know that for
almost every $x\in B^{+}$ we have $x+h_ke_n\in B^{+}$ when
$k$ is large enough. Therefore, for $k$ large enough, it holds that
$$\mathfrak{M}_{\alpha,\Omega}(\vec{f})(x+h_ke_n)\geq u_{x+h_ke_n,\vec{f},\alpha}(\delta(x)),\quad\hbox{for \ }k \hbox{\ large \ enough.}$$
Thus, for almost every $x\in B^{+}$, Lemma \ref{l:5.6} gives that
$$\begin{array}{ll}
D_n\mathfrak{M}_{\alpha,\Omega}(\vec{f})(x)&=\displaystyle\lim\limits_{k\rightarrow\infty}\frac{1}{h_k}
\Big(\mathfrak{M}_{\alpha,\Omega}(\vec{f})(x+h_ke_n)-\mathfrak{M}_{\alpha,\Omega}(\vec{f})(x)\Big)\\
&\geq\displaystyle\limsup\limits_{k\rightarrow\infty}\frac{1}{h_k}(u_{x+h_ke_n,\vec{f},\alpha}(\delta(x))-u_{x,\vec{f},\alpha}(\delta(x)))\\
&=\displaystyle\limsup\limits_{k\rightarrow\infty}\delta(x)^\alpha\sum\limits_{i=1}^mA_{x,\delta(x)}(f_{i,h_k}^n)
\prod\limits_{\mu=i}^{i-1}A_{x,\delta(x)}(f_\mu)\prod\limits_{\nu=i+1}^mA_{x,\delta(x)}(f_{\tau(h_k)}^{\nu,n})\\
&=u_{x,\vec{f^i},\alpha}(\delta(x)).
\end{array}$$ Combining this inequality
with Lemma \ref{l:5.3} implies that
$$D_l\mathfrak{M}_{\alpha,\Omega}(\vec{f})(x)\geq u_{x,\vec{f^i},\alpha}(r),\quad \hbox{for all \ }r\in\mathcal{R}_\alpha(\vec{f})(x).\eqno(5.85)$$
(equality if
$r<\delta(x)$). Let us recall the definition of $\beta$
and $K^j$. We know that $\mathcal{R}_\alpha(\vec{f_j})(x)
\subset\mathcal{R}_\alpha(\vec{f})(x)_{(\beta)}$ when
$j\geq N_1$ and $x\in\Omega\backslash K^j$. It follows
that for every $x\in B^{+}\backslash K^j$, there exists
$r_j\in\mathcal{R}_\alpha(\vec{f_j})(x)$, $r_j<\delta(x)$
such that $|r_j-r|\leq\beta$ for some
$r\in\mathcal{R}_\alpha(\vec{f})(x)$ (Here $r$ may be
$\delta(x)$). Since $x\in B^{+}\backslash(K^j\cup B^j)$,
then $r_j<\delta(x)$. By Lemma \ref{l:5.3} and (5.85), for almost every $x\in B^{+}\backslash(K^j\cup B^j)$ with
$j\geq N_1$, we obtain
that
$$\begin{array}{ll}
&D_n(\mathfrak{M}_{\alpha,\Omega}(\vec{f})-\mathfrak{M}_{\alpha,\Omega}(\vec{f_j}))(x)\\&\geq\displaystyle\sum\limits_{i=1}^m u_{x,\vec{f^i},\alpha}(r)-\sum\limits_{i=1}^m u_{x,\vec{f^i},\alpha}(r_j)+\sum\limits_{i=1}^m u_{x,\vec{f^i},\alpha}(r_j)-\sum\limits_{i=1}^m u_{x,\vec{f_j^i},\alpha}(r_j)\\
&=:J_{1,j}(x)+J_{2,j}(x).
\end{array}\eqno(5.86)$$
By the continuity of the functions $u_{x,
\vec{f^i},\alpha}$ on $[0,\delta(x)]$ we have
$$|\{x\in\Omega:|J_{1,j}(x)|\geq\epsilon/2\}|\rightarrow0\ \ {\rm as}\ j\rightarrow\infty.\eqno(5.87)$$
By the similar argument as in getting (5.78) we have
$$|J_{2,j}(x)|\leq\sum\limits_{i=1}^m\mathcal{K}_{i,j}(x),\eqno(5.88)$$
where $\mathcal{K}_{i,j}$ is given as in (5.78). We get from
(5.80) that
$$\Big|\Big\{x\in\Omega:\sum\limits_{i=1}^m\mathcal{K}_{i,j}(x)\geq\epsilon\Big\}\Big|\rightarrow0\ \ {\rm as}\ j\rightarrow\infty.$$
It follows from (5.88) that
$$|\{x\in\Omega:|J_{2,j}|\geq\epsilon/2\}|\rightarrow0\ \ {\rm as}\ j\rightarrow\infty.\eqno(5.89)$$
Then by (5.86)-(5.87) and (5.89) we have
$$\begin{array}{ll}
&|\{x\in B^{+}\backslash(K^j\cup B^j):D_n(\mathfrak{M}_{\alpha,\Omega}(\vec{f})-\mathfrak{M}_{\alpha,\Omega}(\vec{f_j}))(x)\leq -\epsilon\}|\\
&\leq\displaystyle\Big|\Big\{x\in B^{+}\backslash(K^j\cup B^j):J_{1,j}(x)+J_{2,j}(x)\leq-\epsilon\Big\}\Big|\rightarrow0\ \ {\rm as}\ j\rightarrow\infty.
\end{array}\eqno(5.90)$$
By (5.71)-(5.72), (5.90) and Lemma \ref{l:5.7} we have
$$\|D_n(\mathfrak{M}_{\alpha,\Omega}(\vec{f_j})-\mathfrak{M}_{\alpha,\Omega}(\vec{f}))\|_{L^q(B^{+}\backslash(K^j\cup B^j))}\rightarrow0\ \ {\rm as}\ j\rightarrow\infty.\eqno(5.91)$$
One the other hand, by (5.79), (5.81) and (5.82), or any $j\geq\max\{N_1,N_2\}$, we have
$$\|D_n(\mathfrak{M}_{\alpha,\Omega}(\vec{f_j})-\mathfrak{M}_{\alpha,\Omega}(\vec{f}))\|_{L^q(B^{+}\cap K^j)}\leq\Big\|\sum\limits_{i=1}^m\mathcal{K}_{i,j}\Big\|_{L^q(\Omega)}+2\Big\|\sum\limits_{i=1}^m\mathfrak{M}_{\alpha,\Omega}(\vec{f^i})\Big\|_{L^q(K^j)}
\leq 3\epsilon,$$which gives
$$\|D_n(\mathfrak{M}_{\alpha,\Omega}(\vec{f_j})-\mathfrak{M}_{\alpha,\Omega}(\vec{f}))\|_{L^q(B^{+}\cap K^j)}\rightarrow0\ \ {\rm as}\ j\rightarrow\infty.\eqno(5.92)$$
Combining (5.91) with (5.92) and (5.84) yields that
$$\|D_n(\mathfrak{M}_{\alpha,\Omega}(\vec{f_j})-\mathfrak{M}_{\alpha,\Omega}(\vec{f}))\|_{L^q(B^{+})}\rightarrow0\ \ {\rm as}\ j\rightarrow\infty.\eqno(5.93)$$

It remains to treat the case when $x\in B^{-}$. We know that
for almost every $x\in B^{-}$ we have $x+h_ke_n\in B^{-}$
when $k$ is large enough. It follows that
$$\mathfrak{M}_{\alpha,\Omega}(\vec{f})(x+h_ke_n)=u_{x+h_ke_l,\vec{f},\alpha}(\delta(x+h_ke_l))\leq u_{x+h_ke_l,\vec{f},\alpha}(\delta(x))$$
for $k$ large enough. By the similar arguments as in getting
(5.85) we have
$$D_n\mathfrak{M}_{\alpha,\Omega}(\vec{f})(x)\leq u_{x,\vec{f^i},\alpha}(r)\eqno(5.94)$$
for all $r\in\mathcal{R}_\alpha(\vec{f})(x)$ (equality if
$r<\delta(x)$). Let us recall the definition of $\beta$
and $K^j$. We know that
$\mathcal{R}_\alpha(\vec{f_j})(x)\subset\mathcal{R}_\alpha
(\vec{f})(x)_{(\beta)}$ when $j\geq N_1$ and
$x\in\Omega\backslash K^j$. It follows that for every
$x\in B^{-}\backslash K^j$, there exists
$r_j\in\mathcal{R}_\alpha(\vec{f_j})(x)$, $r_j<\delta(x)$
such that $|r_j-r|\leq\beta$ for some
$r\in\mathcal{R}_\alpha(\vec{f})(x)$ (Here $r$ may be
$\delta(x)$). By the similar argument as in getting
(5.86) we have
$$D_n(\mathfrak{M}_{\alpha,\Omega}(\vec{f})-\mathfrak{M}_{\alpha,\Omega}(\vec{f_j}))(x)
\leq J_{1,j}(x)+J_{2,j}(x)\eqno(5.95)$$
for almost every $x\in B^{-}\backslash(K^j\cup B^j)$ with
$j\geq N_1$. Combining (5.95) with (5.87) and(5.89) yields
that
$$\begin{array}{ll}
&\quad|\{x\in B^{-}\backslash(K^j\cup B^j):D_n(\mathfrak{M}_{\alpha,\Omega}(\vec{f_j})-\mathfrak{M}_{\alpha,\Omega}(\vec{f}))(x)\geq\epsilon\}|\\
&\leq |\{x\in B^{-}\backslash(K^j\cup B^j):J_{1,j}(x)+J_{2,j}(x)\geq\epsilon\}|\rightarrow0\ \ {\rm as}\ j\rightarrow\infty.
\end{array}\eqno(5.96)$$
On the other hand, by the similar argument as in getting
(5.92) we have
$$\|D_n(\mathfrak{M}_{\alpha,\Omega}(\vec{f_j})-\mathfrak{M}_{\alpha,\Omega}(\vec{f}))\|_{L^q(B^{-}\cap K^j)}\rightarrow0\ \ {\rm as}\ j\rightarrow\infty.\eqno(5.97)$$
It follows from (5.96)-(5.97) and (5.84) that
$$\|D_n(\mathfrak{M}_{\alpha,\Omega}(\vec{f_j})-\mathfrak{M}_{\alpha,\Omega}(\vec{f}))\|_{L^q(B^{-})}\rightarrow0\ \ {\rm as}\ j\rightarrow\infty.\eqno(5.98)$$
(5.98) together with (5.84) and (5.93) yields (5.76). This
finishes the proof of Theorem \ref{t:alpha>0continuous}. 
\end{proof}



\begin{thebibliography}{99}


\bibitem{AK}D. Aalto and J. Kinnunen,
\emph{Maximal functions in Sobolev spaces},
Sobolev Spaces in Mathematics I, International Mathematical Series. 2009.


\bibitem{ACL}J. M.  Aldaz, L. Colzani and J. P\'{e}rez L\'{a}zaro,
\emph{Optimal bounds on the modulus of continuity of the uncentered Hardy-Littlewood maximal function},
J. Geom. Anal. {\bf 22} (2012), 132--167.


\bibitem{AL}J. M. Aldaz and J. P\'{e}rez L\'{a}zaro,
\emph{Functions of bounded variation, the derivative of the one dimensional maximal
function, and applications to inequalities}, Trans. Amer. Math. Soc. {\bf 359} (5) (2007), 2443--2461.


\bibitem{AL4}{F. J. Almgren, and E. H. Lieb}, \emph{Symmetric decreasing
rearrangement is sometimes continuous},J. Amer. Math. Soc. {2}, 1989, 683--773.


\bibitem{BL1}S. Barza and M. Lind, \emph{A new variational characterization of Sobolev spaces},
J. Geom. Anal. {\bf 25} (4) (2015), 2185--2195.

\bibitem{CM1}E. Carneiro and J. Madrid, \emph{Derivative bounds
for fractional maximal functions}, Trans. Amer. Math. Soc. (to appear) (2016).

\bibitem{CM2}E. Carneiro and D. Moreira, \emph{On the regularity of
maximal operators},
Proc. Amer. Math. Soc. {\bf 136} (12) (2008), 4395--4404.

\bibitem{CS}E. Carneiro and B.F. Svaiter,
\emph{On the variation of maximal operators of convolution type},
J. Funct. Anal. {\bf 265} (2013), 837--865.

\bibitem{CX}X. Chen, Q. Xue\emph{Weighted estimates for a class of multilinear fractional type operators},
J. Math. Anal. Appl. \textbf{362} (2010), 355--373
\bibitem{CJ}M. Chirst and J.L. Journ\'{e},
\emph{Polynomial growth estimates for multilinear singular integral operators},
Acta Math. {\bf 159} (1987), 51--80.



\bibitem{GT}D. Gilbarg and N.S. Trudinger,
Elliptic partial differential equations of second order,
2nd edn, Springer-Verlage, Berlin, 1983.


\bibitem{GTo1}L. Grafakos and R. H. Torres,
Mulitlinear Calder\'{o}n-Zygmund theory, Adv. Math. {\bf 165} (1) (2002), 124--164.


\bibitem{GTo2}L. Grafakos and R. H. Torres,
Maximal operator and weighted norm inequalities for multilinear singular integrals,
Indiana. Univ. Math. J. {\bf 51} (5) (2002), 1261--1276.



\bibitem{HM}P. Hajlasz and J. Maly,
On approximate differentiability of the maximal function,
Proc. Amer. Math. Soc. {\bf 138} (1) (2010), 165--174.

\bibitem{HO}{P. Haj{\l}asz and J. Onninen},
On boundedness of maximal functions in Sobolev spaces,
Ann. Acad. Sci. Fenn. Math. {\bf 29} (1) (2004), 167--176.

\bibitem{HKKT}{T. Heikkinen, J. Kinnunen, J. Korvenp\"{a}\"{a} and H. Tuominen},
Regularity of the local fractional maximal function,
Bull. London Math. Soc. {\bf 53} (1) (2015), 127--154.


\bibitem{KSt}C.E. Kenig and E.M. Stein, Multilinear estimates and fractional integration,
Math. Res. Lett. {\bf 6} (1999), 1--15.


\bibitem{Ki}J. Kinnunen, The Hardy-Littlewood maximal function
of a Sobolev function, Israel J. Math. {\bf 100} (1997), 117--124.


\bibitem{KL}J. Kinnunen and P. Lindqvist,
The derivative of the maximal function,
J. Reine Angew. Math. {\bf 503} (1998), 161--167.


\bibitem{KM}{J. Kinnunen and O. Martio},
Hardy's inequalities for Sobolev functions,
Math. Res. Lett. {\bf 4} (4) (1997), 489--500.



\bibitem{KS}J. Kinnunen and E. Saksman,
Regularity of the fractional maximal function,
Bull. London Math. Soc. {\bf 35} (4) (2003), 529--535.

\bibitem{Ko1}S. Korry, A class of bounded operators on Sobolev spaces,
Arch. Math. (Basel) {\bf 82} (1) (2004), 40--50.


\bibitem{Ko2}S. Korry,
Boundedness of Hardy-Littlewood maximal operator in the framework of Lizorkin-Triebel spaces,
Rev. Mat. Complut. {\bf 15} (2) (2002), 401--416.


\bibitem{Ku}O. Kurka, On the variation of the Hardy-Littlewood maximal
function, Ann. Acad. Sci. Fenn. Math. {\bf 40} (2015), 109--133.


\bibitem{KK}{N. Kruglyak and E. A. Kuznetsov},
Sharp integral estimates for the fractional maximal function and intepolation,
Ark. Mat. {\bf 44} (2) (2006), 309--326.


\bibitem{LMPT}{M.T. Lacey, K. Moen, C. P\'{e}rez and R.H. Torres},
Sharp weighted bounds for fractional integral operators,
J. Funct. Anal. {\bf 259} (5) (2010), 1073--1097.



\bibitem{LOPTT}{A. K. Lerner, S. Ombrosi, C. P\'{e}rez, R. H. Torres and R. Trujillo-Gonz\'{a}lez},
New maximal functions and multiple weighted for the multilinear Calder\'{o}n-Zygmund theory,
Adv. Math. {\bf 220} (2009), 1222--1264.


\bibitem{LCW}F. Liu, T. Chen and H. Wu,
A note on the endpoint regularity of the Hardy-Littlewood maximal functions,
Bull. Austra. Math. Soc. {\bf 94} (2016), 121-130.

\bibitem{LM}F. Liu and S. Mao,
On the regularity of the one-sided Hardy-Littlewood maximal functions,
Czech. Math. J. (to appear) (2016).

\bibitem{LW1}F. Liu and H. Wu,
On the regularity of the multisublinear maximal functions,
Canad. Math. Bull. {\bf 58} (4) (2015), 808--817.

\bibitem{LW2}F. Liu and H. Wu,
Endpoint regularity of multisublinear fractional maximal functions,
Canad. Math. Bull. (to appear) (2016).


\bibitem{Lu1}H. Luiro,
Continuity of the maximal operator in Sobolev spaces,
Proc. Amer. Math. Soc. {\bf 135} (1) (2007), 243--251.


\bibitem{Lu2}H. Luiro, On the regularity of the Hardy-Littlewood
maximal operator on subdomains of $\mathbb{R}^n$,
Proc. Edinburgh Math. Soc. {\bf 53} (1) (2010), 211--237.

\bibitem{Na}L. P. Natanson, Theory of Functions of a Real Variable,
Frederick Ungar Publishing Co., New York (1950).

\bibitem{HKT}P. Haj\l asz, P. Koskela, H. Tuominen, Sobolev embeddings, extensions and measure density condition, J. Func. Anal., {\bf{254}} (2008), 1217-1234.

\bibitem{Shv} P. Shvartsman, On Sobolev extension domains in $\R^n$, J. Funct. Anal. {\bf 258} (7) (2010), 2205--2245.

\bibitem{Ta}H. Tanaka, A remark on the derivative of the
one-dimensional Hardy-Littlewood maximal function,
Bull. Austral. Math. Soc. {\bf 65} (2) (2002), 253--258.

\bibitem{TD} Tomasz Dlotko, Sobolev spaces and embedding theorems, http://conteudo.icmc.usp.br/pessoas/andcarva\\ /sobolew.pdf, 1-15.

\end{thebibliography}
\end{document}